\documentclass[10pt]{article}
\fussy 
\usepackage[margin=1in]{geometry} 
\usepackage{authblk}

\usepackage{amsthm,amsmath,amssymb}
\usepackage[numbers]{natbib}
\usepackage{hyperref}
\hypersetup{
    colorlinks,%
    citecolor=black,%
    filecolor=black,%
    linkcolor=black,%
    urlcolor=black
}
\usepackage{hypernat}
\usepackage{mathabx}

\usepackage{tikz}
\usetikzlibrary{shapes.geometric, positioning}
\usepackage{soul}

\usepackage{multirow}
\usepackage{float}
\usepackage{enumitem}
\usepackage{comment}

\newtheorem{theorem}{Theorem}
\newtheorem{proposition}{Proposition}
\newtheorem{lemma}{Lemma}
\newtheorem{corollary}{Corollary}
\newtheorem{remark}{Remark}
\newtheorem{definition}{Definition}

\newcommand{\E}{{\mathbb E}}
\newcommand{\F}{{\mathcal{F}}}
\newcommand{\G}{{\mathcal{G}}}
\newcommand{\R}{{\mathbb R}}
\renewcommand{\P}{{\mathbb P}}

\newcommand{\fst}{{\mathcal{F}^{d, s}_{\text{ST}}}}
\newcommand{\fstmid}{{\mathcal{F}^{d, s}_{\text{ST}, \bullet}}}  
\newcommand{\fstinf}{{\mathcal{F}^{d, s}_{\infty-\text{ST}}}}
\newcommand{\fstinfd}{{\mathcal{F}^{d, d}_{\infty-\text{ST}}}}
\newcommand{\oldfstinf}{{\widetilde{\mathcal{F}}^{d, s}_{\infty-\text{ST}}}}
\newcommand{\vxgb}{{V^{d, s}_{\text{XGB}}}}
\newcommand{\vinfxgb}{{V^{d, s}_{\infty-\text{XGB}}}}
\newcommand{\oldvinfxgb}{{\widetilde{V}^{d, s}_{\infty-\text{XGB}}}}
\newcommand{\tv}{{\text{TV}}}
\newcommand{\vit}{{\text{Vit}}}
\newcommand{\hk}{{\text{HK}}}

\newcommand{\lsefstinf}{{\hat{f}^{d, s}_{n, V}}} 
\newcommand{\fcnu}{{f^{d, s}_{c, \{\nu_{L, U}\}}}} 
\newcommand{\minimax}{{\mathfrak{M}^{d, s}_{n, V}}}

\newcommand{\xbf}{{\mathbf{x}}}
\newcommand{\wbf}{{\mathbf{w}}}
\newcommand{\lbf}{{\mathbf{l}}}
\newcommand{\ubf}{{\mathbf{u}}}
\newcommand{\abf}{{\mathbf{a}}}
\newcommand{\pbf}{{\mathbf{p}}}
\newcommand{\qbf}{{\mathbf{q}}}
\newcommand{\rbf}{{\mathbf{r}}}
\newcommand{\mbf}{{\mathbf{m}}}
\newcommand{\Mbf}{{\mathbf{M}}}
\newcommand{\zbf}{{\mathbf{z}}}
\newcommand{\ibf}{{\mathbf{i}}}
\newcommand{\tbf}{{\mathbf{t}}}
\newcommand{\etabf}{{\boldsymbol{\eta}}}
\newcommand{\betabf}{{\boldsymbol{\beta}}}
\newcommand{\deltabf}{{\boldsymbol{\delta}}}
\newcommand{\phibf}{{\boldsymbol{\phi}}}

\newcommand{\zerovec}{{\mathbf{0}}}
\newcommand{\onevec}{{\mathbf{1}}}
\newcommand{\ind}{{\mathbf{1}}}
\newcommand{\ceil}[1]{\lceil #1 \rceil}
\newcommand{\floor}[1]{\lfloor #1 \rfloor}
\newcommand\inner[2]{\langle #1, #2 \rangle}

\def\qt#1{\qquad\text{#1}}
\def\argmin{\mathop{\rm argmin}}

\begin{document}

\title{What Functions Does XGBoost Learn?}
\author{Dohyeong Ki\thanks{\texttt{dohyeong\_ki@berkeley.edu}} }
\author{Adityanand Guntuboyina\thanks{\texttt{aditya@stat.berkeley.edu}}}
\affil{Department of Statistics, University of California, Berkeley}
\date{} 

\maketitle
\vspace{-1.5em}

\begin{abstract}

This paper establishes a rigorous theoretical foundation for the
function class implicitly learned by XGBoost, bridging the gap 
between its empirical success and our theoretical understanding. 
We introduce an infinite-dimensional function class 
$\mathcal{F}^{d, s}_{\infty-\text{ST}}$ that extends finite ensembles 
of bounded-depth regression trees, together with a complexity measure 
$V^{d, s}_{\infty-\text{XGB}}(\cdot)$ that generalizes the $L^1$ 
regularization penalty used in XGBoost. We show that every optimizer 
of the XGBoost objective is also an optimizer of an equivalent 
penalized regression problem over $\mathcal{F}^{d, s}_{\infty-\text{ST}}$ 
with penalty $V^{d, s}_{\infty-\text{XGB}}(\cdot)$, providing an 
interpretation of XGBoost as implicitly targeting a broader function 
class. We also develop a smoothness-based interpretation of 
$\mathcal{F}^{d, s}_{\infty-\text{ST}}$ and 
$V^{d, s}_{\infty-\text{XGB}}(\cdot)$ in terms of 
Hardy--Krause variation. We prove that the least squares estimator over 
$\{f \in \mathcal{F}^{d, s}_{\infty-\text{ST}}: 
V^{d, s}_{\infty-\text{XGB}}(f) \le V\}$ achieves a nearly 
minimax-optimal rate of convergence 
$n^{-2/3} (\log n)^{4(\min(s, d) - 1)/3}$, thereby avoiding the 
curse of dimensionality. Our results provide the first rigorous 
characterization of the function space underlying XGBoost, clarify 
its connection to classical notions of variation, and identify an 
important open problem: whether the XGBoost algorithm itself 
achieves minimax optimality over this class.   
\vspace{0.5em} 
\end{abstract}

\section{Introduction}

Consider the standard regression problem with data 
$(\xbf^{(1)}, y_1), \dots, (\xbf^{(n)}, y_n)$ where each 
$\xbf^{(i)} \in \R^d$ and $y_i \in \R$. XGBoost, introduced by 
\citet{chen2016xgboost}, fits a finite sum of regression trees by 
(approximately) minimizing an objective function consisting of a 
least squares loss and a regularization penalty. We describe below 
the optimization problem that XGBoost seeks to solve; see the official 
documentation \cite{xgboost_params_docs} for further implementation 
details. 

XGBoost constructs individual regression trees using right-continuous 
splits, meaning that each split is of the form $x_j \ge t_j$ versus 
$x_j < t_j$ where $x_j$ denotes the $j^{\text{th}}$ coordinate of 
the covariate vector $\xbf$. Each tree is further constrained to 
have a user-specified maximum depth, controlled by the 
hyperparameter \texttt{max\_depth} (whose default value is 6). 
Recall that the depth of a tree refers to the maximum number of 
splits along any root-to-leaf path. Let $\fst$ denote the class 
of all finite sums of right-continuous\footnote{Throughout the paper, 
the term ``right-continuous'' refers to coordinate-wise 
right-continuity.} regression trees of depth at most $s$ (ST here 
stands for ``sum of trees''). More precisely, $\fst$ consists of 
all functions of the form $\sum_{k=1}^K f_k$ for some $K \ge 1$,
where each $f_k$ is a regression tree with right-continuous splits 
and depth $\le s$. Each $f_k$ can possibly be a constant function, 
in which case we call it a constant regression tree and say that it 
has depth 0. We place no restriction on $K$ beyond finiteness. In 
implementations, XGBoost allows the user to specify an upper bound 
on $K$, typically on the order of several hundred to a few thousand. 
Since this bound is usually chosen to be large, we leave $K$ 
unrestricted in the definition of $\fst$ for theoretical convenience.

For regression, XGBoost minimizes the least squares loss over the
class $\fst$, augmented with an explicit regularization penalty
described below. This explicit regularization is a key innovation that
distinguishes XGBoost from earlier gradient boosting methods such as
Gradient Boosting Machines (see \citet{friedman2001greedy}), and from
ensemble methods such as Random Forests (see
\citet{breiman2001random}).

For a function $f \in \fst$, suppose $f$
is represented as a finite sum of 
trees, and let $\wbf_k$ denote the vector of leaf weights
associated with the
$k^{\text{th}}$ tree. The XGBoost regularization penalty takes the 
form $\alpha \sum_k \|\wbf_k\|_1$, where $\alpha > 0$ controls 
the strength of regularization and $\| \cdot \|_1$ denotes the 
$L^1$ norm. If the $k^{\text{th}}$ tree is a constant tree, we set 
$\|\wbf_k\|_1 = 0$. This penalty discourages overly complex 
trees by constraining the magnitude of the leaf weights. Since each 
function $f \in \fst$ generally admits multiple representations as a 
finite sum of trees, and since the quantity $\sum_k \|\wbf_k\|_1$ 
depends on the particular representation chosen, we obtain a 
representation-invariant measure of complexity by taking the infimum 
over all possible tree decompositions. Specifically, for 
$f \in \fst$, define 
\begin{equation}\label{eq:vxgb-original-def} 
    \vxgb(f) = \inf \Big\{\sum_k \|\wbf_k\|_1 \Big\}
\end{equation}
where the infimum is taken over all representations of $f$ as a 
finite sum of regression trees with right-continuous splits and 
depth at most $s$, and $\wbf_k$ denotes the leaf-weight vector 
of the $k^{\text{th}}$ tree.

In this notation, XGBoost is a greedy algorithm for solving the
optimization problem:  
\begin{equation}\label{xgb_opti}
    \argmin_{f} \Big\{\sum_{i=1}^n \big(y_i - f(\xbf^{(i)})\big)^2 
    + \alpha \vxgb(f): f \in \fst \Big\}.
\end{equation}
It is also common to include
an additional leaf-count penalty of the form $\gamma \sum_k T_k$ where 
$T_k$ is the number of leaves in the $k^{\text{th}}$ tree. Since the 
default value for $\gamma$ is $\gamma = 0$ (see 
\cite{xgboost_params_docs}), we omit this term throughout. Some 
implementations also replace the $L^1$ penalty with a squared $L^2$ 
penalty $\|\wbf_k\|_2^2$. However, such a penalty is not well 
defined for sums of trees. To illustrate this issue, consider the 
function $(x_1, \dots, x_d) \mapsto \ind(x_1 \ge 0)$, for which the 
squared $L^2$ penalty is 1. The same function can alternatively be 
expressed as $\sum_{k=1}^K (1/K) \cdot \ind(x_1 \ge 0)$, for which 
the total squared $L^2$ penalty equals $1/K$, which tends to zero as 
$K \rightarrow \infty$. The fact that the squared $L^2$ penalty can 
be made arbitrarily close to zero by suitably increasing the number of 
trees in the decomposition—unlike the $L^1$ penalty—renders the squared 
$L^2$ penalty ill-posed for tree ensembles. For this reason, we focus 
exclusively on the $L^1$ penalty \eqref{eq:vxgb-original-def}. 
More broadly, the advantages of $L^1$ over $L^2$ penalties are 
well established in high-dimensional statistics (see, e.g., 
\citet{donoho1994ideal, tibshirani1996regression, 
johnstone2011gaussian}, and \citet{tibshirani2014adaptive}). 
Additional discussion on the differences between $L^1$ and $L^2$ 
penalties for tree ensembles is provided in 
Section \ref{sec:squared-l2-penalty}. 

XGBoost has become one of the most widely used machine learning
algorithms, celebrated for its predictive accuracy and efficiency. 
Indeed, it has played a decisive role in many high-profile machine 
learning competitions, and practitioners often note that for tabular 
data (which includes our regression data setting), XGBoost can 
outperform deep learning methods (see, e.g., 
\citet{shwartz2022tabular, grinsztajn2022tree, borisov2022deep}). 
Yet, despite its empirical success, the theoretical properties of 
XGBoost remain poorly understood. This paper takes a step toward 
closing this gap by providing theoretical insights into the behavior 
of solutions to the XGBoost objective \eqref{xgb_opti}.

Our main contribution is the construction of a function class
$\fstinf$ along with an associated complexity measure
$\vinfxgb(\cdot)$ having the following properties:
\begin{enumerate}
\item $\fst \subsetneq \fstinf$ and $\vinfxgb(f) = \vxgb(f)$ whenever
  $f \in \fst$. In other words, $\fstinf$ is a strictly larger 
  function class than $\fst$, and $\vinfxgb(\cdot)$ is an 
  extension of $\vxgb(\cdot)$ to this larger function class. In fact, 
  $\fstinf$ contains many continuous functions, unlike $\fst$, which 
  only includes piecewise constant functions.
\item Every solution to the XGBoost optimization problem
  \eqref{xgb_opti} also solves the following problem:
\begin{equation}\label{lse_fstinf}
    \argmin_f \Big\{\sum_{i=1}^n \big(y_i - f(\xbf^{(i)})\big)^2 
    + \alpha \vinfxgb(f): f \in \fstinf \Big\}.
\end{equation}
\item Under standard regression assumptions with random design and
  squared error loss, the minimax rate of convergence over the 
  function class 
\begin{equation}\label{fst_V}
    \big\{f \in \fstinf: \vinfxgb(f) \le V \big\}
\end{equation}
for fixed $V > 0$ satisfies
\begin{equation}\label{min_bounds}
    \Omega(n^{-2/3} (\log n)^{2(\min(s, d)-1)/3})
    ~ \le ~ \text{minimax rate} ~ \le ~
    O(n^{-2/3} (\log n)^{4(\min(s, d)-1)/3}),  
\end{equation}
where the constants underlying the $\Omega(\cdot)$ and $O(\cdot)$ 
notations depend on $d, s$, and $V$. In particular, both bounds 
increase with $V$, indicating that the minimax rate deteriorates as $V$ 
increases, as one would intuitively expect.

The upper bound in \eqref{min_bounds} is achieved by a least squares 
estimator over \eqref{fst_V}, which can be viewed as solving a 
constrained version of \eqref{lse_fstinf}. By a standard duality
argument, for each $V > 0$, there exists $\alpha$, possibly depending
on both $V$ and the data, such that a solution to the problem
\eqref{lse_fstinf} achieves the upper bound in
\eqref{min_bounds}. 
\end{enumerate}

Taken together, these results show that XGBoost can be interpreted as
implicitly targeting the larger and more expressive function class
$\fstinf$, even though its fitted solutions are constructed as finite
sums of regression trees. The fast convergence rates in
\eqref{min_bounds}, which do not suffer from the usual curse of
dimensionality, suggest that XGBoost can accurately estimate functions
$f^* \in \fstinf$ provided that their complexity $\vinfxgb(f^*)$ is
not too large. This perspective offers a theoretical explanation, at
least in part, for the strong empirical performance of XGBoost in
practice: although real-world regression functions are rarely
piecewise constant, XGBoost can perform well as long as the underlying
function lies in $\fstinf$ with moderate complexity.

At the same time, our results point to potential limitations of
XGBoost. If the true regression function $f^*$ cannot be well 
approximated by elements of $\fstinf$ with controlled $\vinfxgb(f)$, 
then accurate estimation should not be expected (see 
Section \ref{notfstinf} for more details). Finally, we emphasize that 
this paper does not address algorithmic aspects of XGBoost. Our results 
characterize the statistical properties of solutions to the objective 
function that XGBoost aims to optimize, rather than guaranteeing that 
a specific implementation of the algorithm attains these rates 
(see Section \ref{algo_open} for more details). 

We also provide a smoothness-based characterization of $\fstinf$ that
does not rely on any explicit connection to trees. Specifically, we
show that $\fstinf$ is closely related to the class of functions with
finite Hardy–Krause (HK) variation. More precisely, in 
Proposition \ref{prop:fstinf-charac}, we prove that $\fstinf$ 
coincides with the class of right-continuous functions that have 
finite HK variation and do not exhibit interactions of order greater 
than $s$, in the precise sense formalized in 
\eqref{eq:interaction-restriction-cond}.

HK variation is a classical notion of multivariate variation (see,
e.g., \citet{leonov1996total, aistleitner2015functions, 
owen2005multidimensional}) and can be interpreted as a measure of
smoothness. Indeed, for sufficiently smooth functions $f$, 
HK variation is closely related to the $L^1$ norms of mixed partial
derivatives of $f$ of maximal order one, that is,
\begin{equation*}
    \frac{\partial^{r_1 + \dots + r_d} f}{\partial x_1^{r_1} \cdots
    \partial x_d^{r_d}} \qt{with $\max_j r_j = 1$}. 
\end{equation*}
See equation \eqref{eq:hk-smooth-characterization} in 
Section \ref{sec:hk-variation} for the precise relationship. We 
further show that the complexity measure $\vinfxgb(\cdot)$ is 
tightly connected to HK variation: it is bounded above by HK 
variation, and bounded below by a constant (depending on $s$ and $d$) 
times HK variation (see Proposition \ref{prop:vinfxgb-hka-rel}).

This perspective suggests that XGBoost can be viewed as performing
smoothness-constrained nonparametric regression, where smoothness is
quantified through control of mixed derivatives of maximal order one. 
Tree-based methods such as XGBoost are often classified as belonging
to the ``algorithmic modeling'' tradition, distinct from statistical
modeling (see, e.g., \citet{breiman2001statistical}). In contrast, 
our results place XGBoost squarely within a traditional statistical 
framework of regularized estimation governed by an interpretable 
smoothness penalty. 

HK variation has previously been employed as a
regularization penalty in nonparametric regression in
\citet{fang2021multivariate} and in \citet{benkeser2016highly,
van2023efficient, schuler2022lassoed} (in the latter group of papers, 
the method is called ``Highly Adaptive Lasso''). However, HK variation 
suffers from a lack of symmetry that makes it somewhat unnatural as a 
regularization penalty. For example, when $d = 2$, the indicator 
functions $\ind(x_1 \ge t_1, x_2 \ge t_2)$ and 
$\ind(x_1 < t_1, x_2 < t_2)$ have different HK variation values. This 
asymmetry arises because HK variation needs the specification of 
an anchor point \cite{aistleitner2015functions, 
owen2005multidimensional}, and any particular choice of anchor breaks 
symmetry. Prior work \cite{fang2021multivariate, benkeser2016highly} 
typically anchors at the lower-left corner of the domain 
($(-\infty, \dots, -\infty)$ in our setting), but, in principle, any 
point $(a_1, \dots, a_d)$ with $a_j \in \{-\infty, +\infty\}$ may be 
used as the anchor point. All such choices induce a form of 
asymmetry in the resulting HK variation (see 
Section~\ref{sec:connection-vinfxgb-hk}).

In contrast, the complexity measure $\vinfxgb(\cdot)$ does not suffer
from this lack of symmetry: the two indicator functions above receive
identical values under $\vinfxgb(\cdot)$. Owing to this symmetry,
$\vinfxgb(\cdot)$ provides a more natural regularizer than HK
variation. Moreover, since $\vinfxgb(\cdot)$ is uniformly smaller than
HK variation (for any choice of anchor), its use avoids excessive
shrinkage while still offering effective control of model complexity. 

The remainder of the paper is organized as follows. 
Sections \ref{sec:fstinf} and \ref{sec:vinfxgb} introduce the 
function class $\fstinf$ and the complexity measure 
$\vinfxgb(\cdot)$, and describe their connections to Hardy–Krause 
variation. Section \ref{sec:optimization} studies the relationship 
between the XGBoost optimization \eqref{xgb_opti} and the 
optimization \eqref{lse_fstinf} over the broader class $\fstinf$. 
Section \ref{sec:minimax-rate} analyzes minimax rates of 
convergence over \eqref{fst_V}. The discussion section highlights 
several issues related to our main results. Proofs of all results 
appear in the Appendix.

\section{The Function Class $\fstinf$}\label{sec:fstinf}

Our definition of $\fstinf$ is built upon a specific class of basis
functions associated with regression trees. These basis functions
take the form
\begin{equation}\label{basis_func}
    b^{L, U}_{\lbf, \ubf}(x_1, \dots, x_d) 
    = \prod_{j \in L} \ind(x_j \ge l_j) 
    \cdot \prod_{j \in U} \ind(x_j < u_j),
\end{equation}
where $L$ and $U$ are (not necessarily disjoint) subsets of 
$[d] := \{1, \dots, d\}$ with $0 < |L| + |U| \le s$ (with $|\cdot|$ 
denoting set cardinality), and $\lbf := (l_j, j \in L)$ and 
$\ubf := (u_j, j \in U)$ are vectors of real-valued thresholds. 
Since the thresholds may take arbitrary real values, the collection 
of basis functions \eqref{basis_func} is uncountable.   

Any non-constant regression tree with right-continuous splits and 
depth $\le s$, and hence any finite sum of such trees, can be 
expressed as a finite linear combination of these basis functions. 
This representation is obtained by decomposing the tree into 
indicator functions corresponding to individual paths from the root 
to each leaf. For each root-to-leaf path, take 
\begin{equation*}
    L := \{j \in [d]: \text{the path contains at least one split of 
    the form } x_j \ge t \},
\end{equation*}
\begin{equation*}
    U := \{j \in [d]: \text{the path contains at least one split of 
    the form } x_j < t \},
\end{equation*}
and for each $j \in L$ (respectively $j \in U$), take $l_j$ 
(respectively $u_j$) to be the maximum (respectively minimum) of the
thresholds $t$ appearing in those splits along the path. Note that a
coordinate may belong to both $L$ and $U$ if it is split in both
directions along the same path. Thus, $|L| + |U|$ is bounded above 
by the depth of the tree, which is at most $s$.

Since there are uncountably many choices for the threshold vectors 
$\lbf$ and $\ubf$, it is convenient to represent finite linear 
combinations of $b^{L, U}_{\lbf, \ubf}$ using signed measures. 
More precisely, signed measures can be used to encode the 
coefficients multiplying $b^{L, U}_{\lbf, \ubf}$ for different 
threshold vectors $\lbf$ and $\ubf$. For finite signed Borel
measures $\nu_{L, U}$ (indexed by $L, U \subseteq [d]$ with 
$0 < |L| + |U| \le s$) on $\R^{|L| + |U|}$ and $c \in \R$, define 
\begin{equation}\label{f_fst}
    \fcnu(x_1, \dots, x_d) 
    = c + \sum_{0 < |L| + |U| \le s} \int_{\R^{|L| + |U|}} 
    b^{L, U}_{\lbf, \ubf}(x_1, \dots, x_d) \, d\nu_{L, U}(\lbf, \ubf).
\end{equation}
This expression provides a simple and unified way to represent
finite linear combinations of the basis functions 
$b^{L, U}_{\lbf, \ubf}$. Any finite linear combination of 
$b^{L, U}_{\lbf, \ubf}$—and hence every element of $\fst$—can be 
written in this form with discrete signed measures $\nu_{L, U}$ 
having finitely many support points. The next result 
(proved in Appendix \ref{pf:alt-charac-fst})
shows that the converse is also true: all such functions 
\eqref{f_fst} with discrete signed measures $\nu_{L, U}$ having 
finite support belong to $\fst$.

\begin{proposition}\label{prop:alt-charac-fst}
    The class $\fst$ of all finite sums of right-continuous 
    regression trees of depth at most $s$ can be characterized as 
    \begin{equation*}
        \fst = \big\{\fcnu: \nu_{L, U} \text{ are discrete signed 
        measures with finitely many support points}\big\}.
    \end{equation*}
\end{proposition}

In light of Proposition \ref{prop:alt-charac-fst}, a natural 
extension of $\fst$ can be obtained by allowing $\nu_{L, U}$ in 
\eqref{f_fst} to be arbitrary (that is, not necessarily discrete) 
finite signed measures. This leads to the function class $\fstinf$.
\begin{definition}
For fixed $d \ge 1$ and $s \ge 1$, $\fstinf$ consists of all 
functions $\fcnu$ (defined in \eqref{f_fst}) where $c \in \R$, and 
each $\nu_{L, U}$ is a finite signed Borel measure on $\R^{|L|+|U|}$.
\end{definition}
The following result (proved in Appendix \ref{pf:fstinf-stabilization}) 
records some basic properties of $\fstinf$.

\begin{proposition}\label{prop:fstinf-stabilization}
\begin{enumerate}[label = (\alph*)]
    \item Every function in $\fstinf$ is right-continuous.
    \item For $s_1 \le s_2$, 
    $\F^{d, s_1}_{\infty-\text{ST}} 
    \subseteq \F^{d, s_2}_{\infty-\text{ST}}$.
    \item For every $s \ge d$, 
    $\F^{d, s}_{\infty-\text{ST}} 
    = \F^{d, d}_{\infty-\text{ST}}$.
    \item The function class $\fstinf$ is convex.
\end{enumerate}
\end{proposition}

We next show that $\fstinf$ can be characterized via 
Hardy–Krause (HK) variation. To this end, we first recall the 
definition of HK variation.

\subsection{Hardy–Krause (HK) Variation}
\label{sec:hk-variation}

Hardy–Krause (HK) variation is typically defined for functions on 
compact domains (see, e.g., \citet{leonov1996total, 
owen2005multidimensional, aistleitner2015functions}), but, in this 
paper, we work with functions defined on the whole space $\R^d$. 
We therefore modify the standard definitions slightly to 
accommodate the unbounded domain $\R^d$. Before introducing our 
version of HK variation, we first recall Vitali variation, which 
serves as a key building block of HK variation.

\begin{definition}[Quasi-volume]
Let $g$ be a real-valued function defined on $\R^m$. For 
$(u_1, \dots, u_m), (v_1, \dots, v_m) \in \R^m$ with 
$u_j < v_j$ for all $j \in [m]$, the quasi-volume of $g$ over the 
rectangle $\prod_{j = 1}^{m} [u_j, v_j]$ is defined as
\begin{equation*}
    \Delta \Big(g; \prod_{j = 1}^{m} [u_j, v_j]\Big) 
    = \sum_{\deltabf \in \{0, 1\}^m}
    (-1)^{\delta_1 + \cdots + \delta_m} \cdot 
    g\big((1 - \delta_1)v_1 + \delta_1 u_1, \dots, 
    (1 - \delta_m)v_m + \delta_m u_m\big).
\end{equation*}  
\end{definition}

\begin{definition}[Axis-aligned split]
Let $(a_1, \dots, a_m)$ and $(b_1, \dots, b_m)$ be vectors in $\R^m$ 
with $a_j < b_j$ for all $j$. A collection $\mathcal{P}$ of subsets 
of $\prod_{j = 1}^{m} [a_j, b_j]$ is called an axis-aligned split if 
it consists of rectangles of the form 
\begin{equation*}
    \prod_{j = 1}^{m}\big[u^{(j)}_{l_j}, u^{(j)}_{l_j + 1}\big] 
    \quad \text{for } l_j \in [n_j] \text{ and } j \in [m],
\end{equation*} 
where, for each $j \in [m]$, 
$a_j = u^{(j)}_1 < u^{(j)}_2 < \cdots < u^{(j)}_{n_j + 1} = b_j$ 
is a partition of $[a_j, b_j]$.  
\end{definition}

\begin{definition}[Vitali variation]
\begin{enumerate}[label = (\alph*)]
    \item The Vitali variation of $g$ on 
    $\prod_{j = 1}^{m} [a_j, b_j]$ is defined as
    \begin{equation*}
        \vit\Big(g; \prod_{j = 1}^{m} [a_j, b_j]\Big) 
        = \sup_{\mathcal{P}} \sum_{R \in \mathcal{P}} |\Delta(g; R)|,
    \end{equation*}
    where the supremum is taken over all axis-aligned splits 
    $\mathcal{P}$ of $\prod_{j = 1}^{m} [a_j, b_j]$.
    \item The Vitali variation of $g$ on the whole space $\R^m$ is 
    defined by
    \begin{equation*}
        \vit(g) = \sup_{\prod_{j = 1}^{m} [a_j, b_j] \subseteq \R^m}
        \vit\Big(g; \prod_{j = 1}^{m} [a_j, b_j]\Big).
    \end{equation*}  
\end{enumerate}
\end{definition}
If $g$ is sufficiently smooth, the Vitali variation of $g$ on 
$\R^{m}$ admits the following representation (see, e.g., 
\citet[Section 9]{owen2005multidimensional}):
\begin{equation}\label{vit-smooth-characterization}
    \vit(g) = \int_{\R^m} \Big| 
    \frac{\partial^m g(\xbf)}{\partial x_1 \cdots \partial x_m} 
    \Big| \, d\xbf.
\end{equation}

We are ready to define HK variation for functions on $\R^d$. The
definition of HK variation requires specification of an anchor 
point. When the domain is a bounded axis-aligned rectangle, the 
anchor is chosen to be one of its vertices. For example, when the 
domain is $[0, 1]^d$, a common choice for the anchor is the 
lower-left corner $\zerovec = (0, \dots, 0)$. However, in our 
setting, where the domain is the entire space $\R^d$, the anchor 
point needs to be placed at infinity (either $-\infty$ or $+\infty$). This requires 
functions to be suitably well behaved at infinity, in the sense 
described below.

Let $\abf = (a_1, \dots, a_d) \in \{-\infty, +\infty\}^d$ denote 
the anchor point. For each coordinate, there are two possible 
choices: $-\infty$ or $+\infty$. For a function $f: \R^d \to \R$, 
a subset $S \subseteq [d]$ with $S^c := [d] \setminus S$, and $(x_j, j
\in S) \in \R^{|S|}$, define  
\begin{equation}\label{eq:f_section}
    f^S_{(a_j, j \in S^c)}(x_j, j \in S) 
    = \lim_{(x_j, j \in S^c) \to (a_j, j \in S^c)} 
    f(x_1, \dots, x_d). 
\end{equation}
For each $S \subseteq [d]$, we say that the function
$f^S_{(a_j, j \in S^c)}$ is well defined if the above limit exists 
and is finite for all $(x_j, j \in S) \in \R^{|S|}$. This function 
may be viewed as the restriction of $f$ to the section of the 
domain obtained by fixing the coordinates in $S^c$ at the anchoring 
values $a_j$. It can be verified that $f^S_{(a_j, j \in S^c)}$ is 
well defined for all $S \subseteq [d]$ whenever $f \in \fstinf$. 

\begin{definition}[HK variation]
Fix $\abf \in \{-\infty, +\infty\}^d$. Let $f: \R^d \to \R$ be a
function for which  
$f^S_{(a_j, j \in S^c)}$ is well defined for all $S \subseteq [d]$. 
The HK variation of $f$ anchored at $\abf$ is defined by
\begin{equation*}
    \hk_{\abf}(f) 
    = \sum_{\emptyset \neq S \subseteq [d]} 
    \vit(f^S_{(a_j, j \in S^c)}).
\end{equation*}  
\end{definition}
In words, the HK variation of $f$ is the sum of the Vitali 
variations of the restrictions of $f$ to sections of the domain 
obtained by anchoring some coordinates at $a_j$. This explains 
the term ``anchor'' for $\abf$. For sufficiently smooth functions 
$f$, \eqref{vit-smooth-characterization} implies that 
$\hk_{\abf}(f)$ can also be expressed as
\begin{equation}\label{eq:hk-smooth-characterization}
    \hk_{\abf}(f) = \sum_{\emptyset \neq S \subseteq [d]} 
    \int_{\R^{|S|}}
    \Big| \frac{\partial^{|S|}}{\prod_{j \in S} \partial x_j} 
    f^S_{(a_j, j \in S^c)}(x_j, j \in S) \Big| \, d(x_j, j \in S). 
\end{equation}

\subsection{Connection Between $\fstinf$ and HK Variation}
\label{sec:connection-fstinf-hk}
The following result (proved in Appendix \ref{pf:fstinf-charac}) 
shows that $\fstinf$ consists precisely of all right-continuous 
functions with finite HK variation that satisfy an interaction 
restriction condition: for every subset $S \subseteq [d]$ with 
$|S| > s$, 
\begin{equation}\label{eq:interaction-restriction-cond}
    \sum_{\deltabf \in \{0, 1\}^{|S|}} (-1)^{\sum_{j \in S} 
    \delta_j} \cdot f^S_{(a_j, j \in S^c)}\big((1 - \delta_j) 
    w_j + \delta_j v_j, j \in S\big) = 0
    \quad \text{for all } v_j < w_j, j \in S.
\end{equation}
This condition excludes interactions between variables of 
order greater than $s$. The result holds for any choice of 
the anchor point $\abf$, since finiteness of HK 
variation is equivalent across different anchor points.

\begin{proposition}\label{prop:fstinf-charac}
    The following statements are equivalent:
    \begin{enumerate}[label = (\alph*)]
        \item $f \in \fstinf$.
        \item $\hk_{\abf}(f) < \infty$ for some 
        $\abf \in \{-\infty, +\infty\}^d$, and $f$ is right-continuous 
        and satisfies \eqref{eq:interaction-restriction-cond} for 
        all subsets $S \subseteq [d]$ with $|S| > s$.
        \item $\hk_{\abf}(f) < \infty$ for all 
        $\abf \in \{-\infty, +\infty\}^d$, and $f$ is right-continuous 
        and satisfies \eqref{eq:interaction-restriction-cond} for 
        all subsets $S \subseteq [d]$ with $|S| > s$.
    \end{enumerate}
\end{proposition}

\begin{remark}[$d = 1$]
    When $d = 1$, Proposition \ref{prop:fstinf-charac}
    simplifies as follows. For each $s \ge 1$, 
    \begin{equation*}
        \F^{1, s}_{\infty-\text{ST}}  
        = \big\{f: \tv(f) < \infty \text{ and } 
        f \text{ is right-continuous}\big\}.
    \end{equation*}
    Here, $\tv(f)$ denotes the usual total variation of $f$ 
    on $\R$, defined by
    $\tv(f) = \sup_{a < b} \tv(f; [a, b])$, where
    \begin{equation*}
        \tv(f; [a, b]) = \sup
        \sum_{k = 1}^{m} |f(z_{k + 1}) - f(z_k)|,
    \end{equation*}
    with the supremum taken over all $m \ge 1$ and 
    all partitions $a = z_1 < \cdots < z_{m + 1} = b$ of $[a, b]$. 
\end{remark}

Proposition \ref{prop:fstinf-charac} confirms that $\fstinf$ 
contains many continuous functions, in contrast to the subclass 
$\fst$, which consists only of piecewise constant functions. For 
example, any sufficiently smooth function whose mixed partial 
derivatives of maximal order one have finite $L^1$ norms (recall 
\eqref{eq:hk-smooth-characterization}) belongs to 
$\F^{d, d}_{\infty-\text{ST}}$.

\section{The Complexity Measure $\vinfxgb(\cdot)$}
\label{sec:vinfxgb}

Here is the definition of the complexity measure $\vinfxgb(\cdot)$ 
on $\fstinf$. 

\begin{definition}[$\vinfxgb(\cdot)$]
    For $f \in \fstinf$, define
    \begin{equation}\label{vinfxgb_def}
        \vinfxgb(f) := \inf \bigg\{ 
        \sum_{0 < |L| + |U| \le s} \|\nu_{L, U}\|_{\text{TV}}: 
        \fcnu \equiv f \bigg\},
    \end{equation}
    where the infimum is taken over all representations 
    $\fcnu$ of $f$. Here, 
    $\|\nu\|_{\text{TV}} := |\nu|(\R^{|L| + |U|})$ denotes the 
    total variation of the signed measure $\nu$. 
\end{definition}
Basic properties of this complexity measure (proved in 
Appendix \ref{pf:vinfxgb-stabilization}) are summarized below.
\begin{proposition}\label{prop:vinfxgb-stabilization}
\begin{enumerate}[label = (\alph*)]
    \item For $s_1 \le s_2$, 
    $V^{d, s_1}_{\infty-\text{XGB}}(f) 
    \ge V^{d, s_2}_{\infty-\text{XGB}}(f)$ for all 
    $f \in \F^{d, s_1}_{\infty-\text{ST}}$.
    \item For every $s \ge 2d$, 
    $V^{d, s}_{\infty-\text{XGB}}(\cdot) 
    \equiv V^{d, 2d}_{\infty-\text{XGB}}(\cdot)$.
    \item $\vinfxgb(\cdot)$ is convex on $\fstinf$; that is, 
    for all $f, g \in \fstinf$ and $\lambda \in [0, 1]$, 
    \begin{equation*}
        \vinfxgb((1 - \lambda) f + \lambda g) 
        \le (1 - \lambda) \cdot \vinfxgb(f) 
        + \lambda \cdot \vinfxgb(g).
    \end{equation*}
\end{enumerate}
\end{proposition}
The next result (proved in Appendix \ref{pf:complexity-equivalence}) 
shows that $\vinfxgb(f)$ agrees with the XGBoost penalty $\vxgb(f)$ 
(defined in \eqref{eq:vxgb-original-def}) whenever $f \in \fst$.
\begin{theorem}\label{thm:complexity-equivalence}
    For every $f \in \fst$, we have $\vinfxgb(f) = \vxgb(f)$.
\end{theorem}

\subsection{Connection Between $\vinfxgb(\cdot)$ and HK Variation}
\label{sec:connection-vinfxgb-hk}
When $d = 1$, we have the following explicit formula—proved in 
Appendix \ref{pf:vinfxgb-1d-formula}—for $\vinfxgb(\cdot)$ in terms of 
total variation (TV) (recall that HK variation coincides with TV 
when $d = 1$). For $f \in \F^{1, s}_{\infty-\text{ST}}$, we have  
\begin{equation}\label{eq:vinfxgb-1d-formula}
    V^{1, s}_{\infty-\text{XGB}}(f) = 
    \begin{cases}
        \tv(f) & \text{if } s = 1, \\
        \frac{1}{2}\left(\tv(f) + |\Delta(f)|\right) 
        & \text{if } s = 2,
    \end{cases}
\end{equation}
where $\Delta(f) := \lim_{x \rightarrow +\infty} f(x) 
- \lim_{x \rightarrow -\infty} f(x)$. Since
$|\Delta(f)| \le \tv(f)$, it follows that
\begin{equation}\label{1dineq}
    \frac{1}{2}\tv(f) \le V^{1, 2}_{\infty-\text{XGB}}(f) 
    \le \tv(f). 
\end{equation}
When $d \ge 2$, it does not seem possible to provide a direct 
formula for $\vinfxgb(\cdot)$ in terms of HK variation, but an 
inequality analogous to \eqref{1dineq} still holds, as shown in the 
next result (proved in Appendix \ref{pf:vinfxgb-hka-rel}).
\begin{proposition}\label{prop:vinfxgb-hka-rel}
    For every $f \in \fstinf$, we have 
    \begin{equation}\label{eq:vinfxgb-hka-rel}
        \frac{1}{\min(2^s - 1, 2^d)} \cdot 
        \Big(\sup_{\abf \in \{-\infty, +\infty\}^d} 
        \hk_{\abf}(f)\Big) \le \vinfxgb(f) 
        \le \inf_{\abf \in \{-\infty, +\infty\}^d} \hk_{\abf}(f).
    \end{equation}
    Both sides of the inequality are tight, in the sense that there 
    exist non-constant functions in $\fstinf$ for which the left and 
    right inequalities hold with equality, respectively.
\end{proposition}

An important distinction between $\vinfxgb(\cdot)$ and HK variation 
is that HK variation is inherently asymmetric, whereas 
$\vinfxgb(\cdot)$ is symmetric. For example, when $d = s = 2$ and 
$\abf = (-\infty, -\infty)$, for any $t_1, t_2 \in \R$, we have
\begin{equation*}
    \hk_{\abf}\big((x_1, x_2) \mapsto 
    \ind(x_1 \ge t_1, x_2 \ge t_2)\big) = 1,
\end{equation*}
while 
\begin{equation*}
    \hk_{\abf}\big((x_1, x_2) \mapsto 
    \ind(x_1 < t_1, x_2 \ge t_2)\big) = 2
    \ \text{ and } \
    \hk_{\abf}\big((x_1, x_2) \mapsto 
    \ind(x_1 < t_1, x_2 < t_2)\big) = 3.
\end{equation*}
Similar asymmetry arises for other choices of $\abf$. In contrast, 
for $\vinfxgb(\cdot)$, we have
\begin{equation*} 
\begin{aligned}
    &\vinfxgb\big((x_1, x_2) \mapsto 
    \ind(x_1 \ge t_1, x_2 \ge t_2)\big) 
    = \vinfxgb\big((x_1, x_2) \mapsto 
    \ind(x_1 < t_1, x_2 \ge t_2)\big) \\
    &\quad = 
    \vinfxgb\big((x_1, x_2) \mapsto 
    \ind(x_1 \ge t_1, x_2 < t_2)\big) 
    = \vinfxgb\big((x_1, x_2) \mapsto 
    \ind(x_1 < t_1, x_2 < t_2)\big) = 1.
\end{aligned}
\end{equation*}

This asymmetry in HK variation arises from the presence of an anchor 
point. HK variation anchors the function at a single corner of the 
domain, thereby inducing asymmetry. Consequently, estimation results 
based on HK variation as a regularization penalty may change if the 
anchor point is moved to another corner or, equivalently, if some 
coordinate axes of the domain are flipped. By contrast, 
$\vinfxgb(\cdot)$ is invariant to axis flipping, as formalized in 
the next proposition (proved in Appendix \ref{pf:vinfxgb-symmetry}), 
which suggests that $\vinfxgb(\cdot)$ provides a more natural 
regularizer than HK variation $\hk_{\abf}(\cdot)$.

\begin{proposition}\label{prop:vinfxgb-symmetry}
    Let $f \in \fstinf$. Fix $j \in [d]$ and $t_j \in \R$, 
    and define $g: \R^d \to \R$ by
    \begin{equation*}
        g(x_1, \dots, x_d) 
        = f(x_1, \dots, x_{j-1}, t_j - x_j, x_{j+1}, \dots, x_d) 
        \quad \text{for } (x_1, \dots, x_d) \in \R^d.
    \end{equation*}
    Then, $\vinfxgb(g) = \vinfxgb(f)$.
\end{proposition}

Additional insight into the asymmetry of HK variation, contrasted 
with the symmetry of $\vinfxgb(\cdot)$, is provided in 
Section \ref{inf_conv}.

\section{Optimization Equivalence Between XGBoost and 
$\eqref{lse_fstinf}$}
\label{sec:optimization}

In this section, we analyze the optimization problems 
\eqref{xgb_opti} and \eqref{lse_fstinf}. Our first result proves 
the existence of solutions to both problems and shows that any 
solution to \eqref{xgb_opti} is simultaneously a solution to 
\eqref{lse_fstinf}. Consequently, XGBoost can be viewed as 
implicitly optimizing over the broader class $\fstinf$, which 
contains smooth functions as well as piecewise constant ones.
\begin{theorem}\label{thm:xgb-opti-lse}
    There exists a solution to \eqref{lse_fstinf} that is also a
    solution to \eqref{xgb_opti}. Moreover, every solution to
    \eqref{xgb_opti} is also a solution to \eqref{lse_fstinf}. 
\end{theorem}

In standard XGBoost implementations, split thresholds for 
regression trees are typically restricted to midpoints between 
observed covariate values. More precisely, for each coordinate 
$j$, split thresholds are chosen from the midpoints between 
consecutive observed values of the $j^{\text{th}}$ covariate. Let 
$v^{(j)}_1 < \dots < v^{(j)}_{n_j}$ denote 
the distinct observed values of the $j^{\text{th}}$ covariate 
$x_j$, sorted in increasing order. Note that
\begin{equation*}
    \big\{v^{(j)}_1, \dots, v^{(j)}_{n_j}\big\} 
    = \big\{x^{(1)}_j, \dots, x^{(n)}_j\big\}
\end{equation*}
where $x^{(i)}_j$ denotes the $j^{\text{th}}$ coordinate of the 
$i^{\text{th}}$ data point $\xbf^{(i)}$. Let $\fstmid$ denote the 
subclass of $\fst$ consisting of finite sums of (right-continuous) 
trees with depth at most $s$, where each individual tree restricts 
split thresholds on the $j^{\text{th}}$ coordinate to the set:
\begin{equation*}
    \Big\{ (v^{(j)}_1 + v^{(j)}_2)/2, \dots, 
    (v^{(j)}_{n_j - 1} + v^{(j)}_{n_j})/2 \Big\}.  
\end{equation*}
Then, the XGBoost algorithm can also be viewed as a greedy 
procedure for solving: 
\begin{equation}\label{xgb_opti_mid}
    \argmin_{f} \Big\{\sum_{i=1}^n 
    \big(y_i - f(\xbf^{(i)})\big)^2 + \alpha \vxgb(f): 
    f \in \fstmid \Big\}.
\end{equation}
The following result (proved in Appendix \ref{pf:xgb-opti-lse-mid}) 
shows that the problem \eqref{lse_fstinf} is also closely related 
to \eqref{xgb_opti_mid}.
\begin{theorem}\label{thm:xgb-opti-lse-mid}
    There exists a solution to \eqref{lse_fstinf} that is also 
    a solution to \eqref{xgb_opti_mid}. Moreover, every solution 
    to \eqref{xgb_opti_mid} is also a solution to 
    \eqref{lse_fstinf}.   
\end{theorem}

The above pair of theorems is a direct consequence of the following
lemma (proved in Appendix \ref{pf:discretization}), 
which asserts that for every $f \in \fstinf$, there exists a 
function in $\fstmid$ that agrees with $f$ at every data point 
$\xbf^{(i)}$ and has no greater complexity.
\begin{lemma}\label{lem:discretization}
    For every $f \in \fstinf$, there exists 
    $\fcnu \in \fstmid$ such that
    \begin{enumerate}[label = (\alph*)]
    \item $\nu_{L, U}$ are discrete signed Borel measures supported 
    on the lattices
    \begin{equation}\label{lattice}
        \prod_{j \in L} \Big\{ (v^{(j)}_1 + v^{(j)}_2)/2, \dots, 
        (v^{(j)}_{n_j - 1} + v^{(j)}_{n_j})/2 \Big\} \times 
        \prod_{j \in U} \Big\{ (v^{(j)}_1 + v^{(j)}_2)/2, \dots, 
        (v^{(j)}_{n_j - 1} + v^{(j)}_{n_j})/2 \Big\}
    \end{equation}
    \item $\fcnu(\xbf^{(i)}) = f(\xbf^{(i)})$ for 
    $i = 1, \dots, n$
    \item 
    \begin{equation*}
        \vinfxgb(\fcnu) 
        = \sum_{0 < |L| + |U| \le s} \|\nu_{L, U}\|_{\text{TV}}
        \le \vinfxgb(f).
    \end{equation*}
    \end{enumerate}
\end{lemma}

Lemma \ref{lem:discretization} continues to hold even if the
midpoint $(v^{(j)}_{m_j} + v^{(j)}_{m_j + 1})/2$ is replaced by 
any other point in the interval $(v^{(j)}_{m_j}, v^{(j)}_{m_j + 1})$. 
We use midpoints because this choice aligns with standard XGBoost 
implementations. By default, XGBoost uses midpoints when the dataset 
is small, although it switches to quantile-based splits for larger 
datasets due to computational limitations 
(see \cite{xgboost_params_docs}).

The equality in the first part of condition (c) deserves special
attention. Since $\vinfxgb(\cdot)$ is defined as an infimum over
all admissible integral representations \eqref{f_fst}, in general, 
only an inequality holds between $\vinfxgb(\fcnu)$ and the sum of 
the total variations of the signed measures $\nu_{L, U}$. However, 
for the functions constructed in Lemma \ref{lem:discretization}, 
equality is attained. This eliminates the need to take an infimum 
and allows the penalty term to be expressed explicitly as a sum 
of the total variations of the associated signed measures.

\section{Minimax Risk}
\label{sec:minimax-rate}

In this section, we study the minimax rate of convergence over 
the function class \eqref{fst_V}. Throughout, we assume 
$(\xbf^{(1)}, y_1), \dots, (\xbf^{(n)}, y_n)$ are generated  
according to the model
\begin{equation*}
    y_i = f^*(\xbf^{(i)}) + \xi_i
\end{equation*}
where $f^*$ is the true regression function. We work in the 
random-design setting, in which the covariates $\xbf^{(i)}$ are 
assumed to be i.i.d. with density $p_0$ supported on a 
compact rectangle and bounded from above:
\begin{equation}\label{density-compact-support}
    p_0(\xbf) = 0 \ \ \text{when } \xbf \notin 
    \prod_{j = 1}^{d} \Big[-\frac{M_j}{2}, \frac{M_j}{2}\Big] 
    \quad \text{ and } \quad  
    B := M_1 \cdots M_d \cdot \sup_{\xbf} p_0(\xbf) < \infty.
\end{equation}
Note that when $p_0$ is the density of the uniform distribution on 
$\prod_{j = 1}^{d} [-M_j/2, M_j/2]$, we have $B = 1$.
We further assume that the error terms $\xi_i$ are i.i.d., 
mean-zero, and independent of the covariates $\xbf^{(i)}$.

The minimax risk over the class \eqref{fst_V} is defined as 
\begin{equation}\label{minimax-risk}
    \minimax := 
    \inf_{\hat{f}_n} \sup_{\substack{f^* \in \fstinf \\ 
    \vinfxgb(f^*) \le V}} 
    \E \|\hat{f}_n - f^*\|_{p_0, 2}^2,
\end{equation}
where the infimum is taken over all estimators $\hat{f}_n$ 
based on the data $(\xbf^{(1)}, y_1), \dots, (\xbf^{(n)}, y_n)$. 
Here, $\|\hat{f}_n - f^*\|_{p_0, 2}$ denotes the $L^2(p_0)$ loss 
between $\hat{f}_n$ and $f^*$: 
\begin{equation*}
    \|\hat{f}_n - f^*\|_{p_0, 2}^2
    := \int_{\R^d} (\hat{f}_n - f^*)^2(\xbf) 
    \cdot p_0(\xbf) \, d\xbf.
\end{equation*}

The first main result of this section establishes an upper bound 
on the minimax risk \eqref{minimax-risk}. We obtain this bound 
by analyzing a specific least squares estimator over the class 
\eqref{fst_V}. Specifically, we consider the least squares 
estimator over \eqref{fst_V} subject to the additional 
restrictions that the associated signed measures $\nu_{L, U}$ 
satisfy condition (a) and the equality in condition (c) of 
Lemma \ref{lem:discretization} in Section \ref{sec:optimization}:
\begin{equation}\label{xgb_opti_mid_const}
\begin{aligned}
    &\lsefstinf \in \argmin_{f} 
    \bigg\{\sum_{i = 1}^n \big(y_i - f(\xbf^{(i)})\big)^2: 
    f \equiv f_{c, \{\nu_{L, U}\}} \in \fstinf, \\
    &\qquad \qquad \qquad \qquad \qquad 
    \sum_{0 < |L| + |U| \le s} \|\nu_{L, U}\|_{\text{TV}} 
    \le V, 
    \text{ and } \nu_{L, U} \text{ satisfy condition (a) of 
    Lemma \ref{lem:discretization}} \bigg\}.
\end{aligned}
\end{equation}
Lemma \ref{lem:discretization} guarantees that $\lsefstinf$
also minimizes the least squares criterion over the original 
class \eqref{fst_V}. In other words, it is a least 
squares estimator over the class \eqref{fst_V}:
\begin{equation*}
    \lsefstinf \in \argmin_{f} 
    \bigg\{\sum_{i = 1}^n \big(y_i - f(\xbf^{(i)})\big)^2: 
    f \in \fstinf \text{ and } \vinfxgb(f) \le V \bigg\}.
\end{equation*}
One can further verify that for each $V$, there exists 
$\alpha$, possibly depending on both $V$ and the data, such that 
$\lsefstinf$ is also a solution to the original penalized 
formulation \eqref{lse_fstinf}. More precisely, if $\alpha$ is 
chosen as the solution to the Lagrange dual problem of 
\eqref{xgb_opti_mid_const}, then $\lsefstinf$ is also a solution 
to the penalized version of \eqref{xgb_opti_mid_const} and hence 
a solution to \eqref{lse_fstinf} (recall 
Lemma \ref{lem:discretization}).

The following theorem (proved in Appendix \ref{pf:risk-upper-bound}) 
provides an upper bound on the risk of $\lsefstinf$. For this 
result, we impose an additional assumption on the error terms 
$\xi_i$. Specifically, we assume that they have finite 
$L^{3, 1}$ norm:
\begin{equation}\label{error-l31-norm}
    \| \xi_i \|_{3, 1} 
    := \int_{0}^{\infty} \P(|\xi_i| > t)^{1/3} \, dt 
    < \infty.
\end{equation}
This norm condition is mild: it is stronger than requiring a 
finite $L^3$ norm but weaker than requiring a finite 
$L^{3 + \epsilon}$ norm for any $\epsilon > 0$ (see, e.g., 
\citet[Chapter 1.4]{loukas2014classical}).

\begin{theorem}\label{thm:risk-upper-bound}
    Fix a true regression function $f^* : \R^d \rightarrow \R$, 
    not necessarily belonging to $\fstinf$. Suppose that the density
    $p_0$ satisfies \eqref{density-compact-support} and that the 
    error terms $\xi_i$ satisfy \eqref{error-l31-norm}. Then, for 
    every $f_0 \in \fstinf$ with $\vinfxgb(f_0) < V$, we have  
    \begin{equation}\label{thm:risk-upper-bound.eq}
        \E \big[\|\lsefstinf - f^*\|_{p_0, 2}^2 \big] 
        \le C \|f_0 - f^*\|_{p_0, 2}^2 + 
        O\big(d^{4\widebar{s}}(1 + \log d)^{4(\widebar{s} - 1)}
        (V + 1)^2 \cdot n^{-2/3} (\log n)^{4(\widebar{s} - 1)/3}\big),
    \end{equation}
    where $C > 0$ is a universal constant, $\widebar{s} := \min(s, d)$, 
    and the constant factor underlying $O(\cdot)$ depends on $B, s$, 
    the moments of $\xi_i$, and 
    \begin{equation*}
        \sup_{\xbf \in \prod_{j = 1}^{d} [-M_j/2, M_j/2]} 
        |f_0(\xbf) - f^*(\xbf)|.
    \end{equation*}
\end{theorem}

Theorem \ref{thm:risk-upper-bound} is stated in a misspecified setting,
allowing the true function $f^*$ to be arbitrary. If $f^* \in
\fstinf$ and $\vinfxgb(f^*) < V$, then we can take $f_0 = f^*$ in
\eqref{thm:risk-upper-bound.eq} to deduce
\begin{equation*}
    \E \big[\|\lsefstinf - f^*\|_{p_0, 2}^2 \big] 
    \le O\big(d^{4\widebar{s}}(1 + \log d)^{4(\widebar{s} - 1)}
    (V + 1)^2 \cdot n^{-2/3} (\log n)^{4(\widebar{s} - 1)/3}\big),
\end{equation*}
where the constant factor underlying $O(\cdot)$ depends on 
$B, s$, and the moments of $\xi_i$. This shows that when $f^* \in
\fstinf$, the least squares estimator $\lsefstinf$ over the class
\eqref{fst_V} converges to $f^*$ at the rate $n^{-2/3}$, up to
multiplicative logarithmic factors.  
The upper bound depends on the complexity bound $V$ on $\vinfxgb(f^*)$ 
through the factor $(V + 1)^2$, indicating that the accuracy of 
$\lsefstinf$ deteriorates as the complexity of the target function 
increases. It is also worth noting that the dependence on $d$ in 
the bound is polynomial.

\begin{remark}
    Given the relationship between $\vinfxgb(\cdot)$ and 
    HK variation discussed in 
    Section \ref{sec:connection-vinfxgb-hk}, 
    it is natural to compare Theorem \ref{thm:risk-upper-bound} 
    with existing results on HK variation denoising, such as 
    those in \citet{fang2021multivariate}. In particular, 
    Theorem 4.5 of \cite{fang2021multivariate} shows that the 
    least squares estimator under a HK variation constraint 
    also achieves an $n^{-2/3}$ rate of convergence (up to a 
    slightly different multiplicative logarithmic factor). 

    This similarity is not surprising in light of the 
    close connection between HK variation and $\vinfxgb(\cdot)$, 
    especially Proposition \ref{prop:vinfxgb-hka-rel}. However, 
    there are important differences. Theorem 4.5 of 
    \cite{fang2021multivariate} is established under a 
    fixed-design setting, where the design points $\xbf^{(i)}$ 
    form a lattice, whereas our result assumes random designs, 
    which are more relevant in many applications. Also, 
    the analysis of \cite{fang2021multivariate} is restricted to the
    case $s = d$ (in their framework, this means all interaction
    orders between covariates are allowed), while our result holds
    for all $1 \le s \le d$. Moreover, the bounds in
    \cite{fang2021multivariate} do not explicitly specify the 
    dependence on $d$. 
\end{remark}

The following upper bound on the bracketing entropy 
(proved in Appendix \ref{pf:bracketing-entropy-bound}) 
is a key ingredient for the proof of 
Theorem \ref{thm:risk-upper-bound}. Let $\F_{\Mbf}(V)$ denote the 
class of all functions $\fcnu \in \fstinf$ of the form \eqref{f_fst} 
satisfying: 
\begin{enumerate}[label = (\alph*)]
    \item $\nu_{L, U}$ are supported on 
    $\prod_{j \in L} (-M_j/2, M_j/2] 
    \times \prod_{j \in U} (-M_j/2, M_j/2]$
    \item $\sum_{0 < |L| + |U| \le s} \|\nu_{L, U}\|_{\text{TV}} 
    \le V$. 
\end{enumerate}
The class $\F_{\Mbf}(V)$ is not totally bounded, since it contains all
constant functions. We therefore restrict attention to the subclass
\begin{equation*}
    B(V, t) = \{f \in \F_{\Mbf}(V) : \|f\|_{p_0, 2} \le t\}. 
\end{equation*}

\begin{lemma}\label{lem:bracketing-entropy-bound}
There exist constants $C_s > 0$, depending only on $s$, and 
$C_{B, s} > 0$, depending only on $B$ and $s$, such that 
for every $\epsilon, t, V > 0$,
\begin{equation*}
    \log N_{[ \ ]}(\epsilon, B(V, t), \| \cdot \|_{p_0, 2}) 
    \le \log\Big(2 + \frac{C_s(V + t)}{\epsilon}\Big) 
    + C_{B, s} d^{2\widebar{s}}(1 + \log d)^{2(\widebar{s} - 1)}
    \Big(2 + \frac{V}{\epsilon}\Big) 
    \Big[\log\Big(2 + \frac{V}{\epsilon}\Big)\Big]^{2(\widebar{s} - 1)}.
\end{equation*}
Here, $N_{[ \ ]}(\epsilon, \F, \| \cdot \|_{p_0, 2})$ denotes the 
$\epsilon$-bracketing number of the class $\F$ with respect to 
the norm $\| \cdot \|_{p_0, 2}$.
\end{lemma}

This result builds on the bracketing entropy bounds of 
\citet{gao2013bracketing} for multivariate cumulative distribution 
functions corresponding to probability measures supported on a fixed 
compact rectangle. The connection between the class $\F_{\Mbf}(V)$ 
and the class of multivariate cumulative distribution 
functions follows from the observation that each term 
$\int b^{L, U}_{\lbf, \ubf} \, d\nu_{L, U}(\lbf, \ubf)$ in 
\eqref{f_fst} resembles a cumulative distribution function, since 
the basis functions $b^{L, U}_{\lbf, \ubf}$ are constructed from 
indicator functions.

An earlier work by \citet{blei2007metric} establishes metric entropy 
bounds (rather than bracketing entropy bounds) for the same class of
cumulative distribution functions, with a sharper logarithmic factor. 
However, for the proof of Theorem \ref{thm:risk-upper-bound}, 
bracketing entropy is essential, and the results of 
\citet{blei2007metric} therefore cannot be directly applied.
  
Theorem \ref{thm:risk-upper-bound} immediately implies the following
corollary (proved in Appendix \ref{pf:minimax-upper-bound}). 

\begin{corollary}\label{cor:minimax-upper-bound}
    The minimax risk $\minimax$ satisfies
    \begin{equation*}
        \minimax
        \le O\big(d^{4\widebar{s}}(1 + \log d)^{4(\widebar{s} - 1)} 
        (V + 1)^2 \cdot n^{-2/3} (\log n)^{4(\widebar{s} - 1)/3}\big),
    \end{equation*}
    where the constant factor underlying $O(\cdot)$ depends on 
    $B, s$, and the moments of $\xi_i$.
\end{corollary}

We now turn to the second main result of this section
(proved in Appendix \ref{pf:minimax-lower-bound}), 
which establishes a lower bound on the minimax risk. For this lower 
bound result, in addition to \eqref{density-compact-support}, we
further assume that the density $p_0$ is bounded away from zero 
on its support, in the sense that
\begin{equation}\label{density-lower-bounded}
    b := M_1 \cdots M_d \cdot 
    \inf_{\xbf \in \prod_{j = 1}^{d} [-M_j/2, M_j/2]} 
    p_0(\xbf) > 0.
\end{equation}
Note that $b = 1$ when $p_0$ is the uniform density on 
$\prod_{j = 1}^{d} [-M_j/2, M_j/2]$. 
For the error terms $\xi_i$, instead of \eqref{error-l31-norm}, 
we assume that they are Gaussian:
\begin{equation}\label{error-Gaussian}
    \xi_i \overset{\text{i.i.d.}}{\sim} N(0, \sigma^2).
\end{equation}

\begin{theorem}\label{thm:minimax-lower-bound}
    Suppose the density $p_0$ satisfies
    \eqref{density-compact-support} and \eqref{density-lower-bounded}, 
    and the error terms $\xi_i$ satisfy \eqref{error-Gaussian}. 
    Then, there exist constants $C_{b, B, \widebar{s}} > 0$, 
    depending only on $b, B$, and $\widebar{s} = \min(s, d)$, and 
    $C_{B, \widebar{s}} > 0$, depending only on $B$ and $\widebar{s}$, 
    such that
    \begin{equation*}
        \minimax \ge C_{b, B, \widebar{s}} 
        \Big(\frac{\sigma^2 V}{n}\Big)^{2/3} 
        \bigg[\log \Big(\frac{n V^2}{\sigma^2}\Big)
        \bigg]^{2(\widebar{s} - 1)/3},
    \end{equation*}
    provided that $n \ge C_{B, \widebar{s}} (\sigma^2/ V^2)$.
\end{theorem}

Combining Corollary \ref{cor:minimax-upper-bound} and 
Theorem \ref{thm:minimax-lower-bound}, we conclude that the minimax 
rate of convergence over the class \eqref{fst_V} is $n^{-2/3}$, up 
to multiplicative logarithmic factors whose exponent lies between 
$2(\widebar{s} - 1)/3$ and $4(\widebar{s} - 1)/3$. This 
nearly dimension-free rate indicates that the class 
\eqref{fst_V} is sufficiently regularized even in high dimensions. 
In other words, the complexity measure $\vinfxgb(\cdot)$ (and hence 
$\vxgb(\cdot)$) provides effective regularization as the dimension 
$d$ increases, adequately controlling model complexity in 
high-dimensional settings. 

This observation offers a possible explanation for the strong 
empirical performance of XGBoost, complementing the fact 
that every solution to the XGBoost optimization problem 
\eqref{xgb_opti} also solves the penalized least squares problem 
\eqref{lse_fstinf} (Theorem \ref{thm:xgb-opti-lse}) over the 
function class $\fstinf$, which contains many functions beyond 
piecewise constant ones (Proposition \ref{prop:fstinf-charac}).

\section{Discussion}\label{sec:discussion}
\subsection{Connection to a Symmetrized HK Variation}
\label{inf_conv}

As discussed in Section \ref{sec:connection-vinfxgb-hk}, HK variation
has a lack of symmetry due to the need to specify an anchor point. 
One can attempt to restore symmetry by combining all $2^d$ versions 
$\hk_{\abf}(\cdot)$ corresponding to $\abf = (a_1, \dots, a_d)$ with 
$a_j \in \{-\infty, +\infty\}$. In the mathematical image processing 
literature, a natural device for combining multiple notions of 
variation into a single quantity is \textit{infimal convolution} 
(see, e.g., \citet{chambolle1997image, bergounioux2016mathematical,
setzer2008variational, setzer2011infimal, bredies2020higher}). 
Following this idea, one may consider the infimal convolution of 
the Hardy–Krause variations $\hk_{\abf}(\cdot)$ over all 
$\abf \in \{-\infty, +\infty\}^d$: 
\begin{equation}\label{infconv}
    \inf \Big\{\sum_{\abf \in \{-\infty, +\infty\}^d} 
    \hk_{\abf}(f_{\abf}): 
    \sum_{\abf \in \{-\infty, +\infty\}^d} f_{\abf} \equiv f,
    \ f_{\abf} \in \fstinf \ \forall \abf\Big\}.
\end{equation}
A natural question is then how this quantity, which also satisfies 
the symmetry condition described in 
Proposition \ref{prop:vinfxgb-symmetry}, relates to $\vinfxgb(\cdot)$.

It can be shown that if the definitions of $\fstinf$ and 
$\vinfxgb(\cdot)$ are modified to forbid repeated use of the 
same variable for splits within each tree, then the resulting 
complexity measure coincides with \eqref{infconv}. To make this 
precise, consider the function class $\oldfstinf$ consisting of all 
functions $f: \R^d \to \R$ of the form \eqref{f_fst}, but with the 
sum ranging only over \textit{disjoint} subsets $L$ and $U$ of $[d]$ 
satisfying $0 < |L| + |U| \le s$. For each $f \in \oldfstinf$, 
define the complexity $\oldvinfxgb(f)$ of $f$ analogously to 
\eqref{vinfxgb_def}, again restricting the sum to disjoint subsets 
$L$ and $U$ with $0 < |L| + |U| \le s$. 

One can verify that $\oldfstinf = \fstinf$ and that 
$\vinfxgb(f) \le \oldvinfxgb(f) \le \hk_{\abf}(f)$ for all 
$f \in \fstinf$ and all $\abf \in \{-\infty, +\infty\}^d$. More 
importantly, $\oldvinfxgb(\cdot)$ coincides with the infimal 
convolution in \eqref{infconv}, as shown in the following proposition
(proved in Appendix \ref{pf:infimal-convolution}).

\begin{proposition}\label{prop:infimal-convolution}
    For every $f \in \fstinf$, we have
    \begin{equation*}
        \oldvinfxgb(f) 
        = \inf \Big\{\sum_{\abf \in \{-\infty, +\infty\}^d} 
        \hk_{\abf}(f_{\abf}): 
        \sum_{\abf \in \{-\infty, +\infty\}^d} f_{\abf} \equiv f,
        \ f_{\abf} \in \fstinf \ \forall \abf\Big\}.
    \end{equation*}
\end{proposition}

Although $\oldvinfxgb(\cdot)$ is symmetric and admits a clean
characterization via infimal convolution of HK variations across
different anchors, it does not fully reflect the behavior of
regression trees as used in practice. An important aspect of
regression trees is the ability to split on the same variable 
multiple times within a single tree, which enables localized 
refinement along a coordinate. Disallowing such repeated splits can 
reduce estimation accuracy in practice. The complexity 
$\oldvinfxgb(\cdot)$ corresponds to this restricted setting, 
whereas $\vinfxgb(\cdot)$ allows repeated splits on the same 
variable within individual trees. As a result, while both notions 
satisfy symmetry properties, $\vinfxgb(\cdot)$ more closely matches 
the structural flexibility inherent in regression trees and is 
therefore the more appropriate notion of variation in this context.

\subsection{Learnability Beyond $\fstinf$}\label{notfstinf}

We have argued that XGBoost is expected to effectively estimate
functions in the class $\fstinf$. In particular, if $f^* \in \fstinf$
and the complexity measure $\vinfxgb(f^*)$ can be treated as a constant, then
the idealized XGBoost estimator—defined as a solution to the XGBoost 
optimization problem—achieves the curse-of-dimensionality-avoiding 
rate $n^{-2/3}$, up to logarithmic factors.

A natural follow-up question is what happens when $f^*$ lies outside
$\fstinf$. A simple example of such a function is
\begin{equation}\label{func_notfst}
    f^*(\xbf) := \ind(x_1 + \dots + x_d \ge 0, \xbf \in [-1, 1]^d).
\end{equation}
It can be shown that this function does not belong to $\fstinf$. One
way to see this is that $f^*$ has infinite Hardy–Krause variation; 
see, e.g., \citet[Proposition 17]{owen2005multidimensional}.

For functions $f^*$ lying outside $\fstinf$, in light of 
Theorems \ref{thm:risk-upper-bound} and \ref{thm:minimax-lower-bound} 
of Section \ref{sec:minimax-rate}, it is natural to conjecture 
that the risk of the idealized XGBoost estimator takes the form
\begin{equation}\label{conj_misspec}
    \inf_{V} \bigg((V + 1)^{\beta} \cdot n^{-2/3} (\log n)^{\gamma}
    + \inf_{\substack{f_0 \in \fstinf \\ \vinfxgb(f_0) \le V}} 
    \|f_0 - f^*\|_{p_0, 2}^2\bigg),
\end{equation}
for some constants $\beta \in [2/3, 2]$ and 
$\gamma \in [2(\widebar{s} - 1)/3, 4(\widebar{s} - 1)/3]$, where 
$\widebar{s} = \min(s, d)$. The upper bounds on $\beta$ and $\gamma$ 
follow from the risk upper bound in Theorem \ref{thm:risk-upper-bound}, 
while the lower bounds are expected from the minimax lower bound in 
Theorem \ref{thm:minimax-lower-bound}. For simplicity, we suppress 
the dependence on other parameters, such as $d$, $s$, and the 
distributions of the covariates and error terms.

If we assume $\beta = 2/3$ and ignore the logarithmic factor 
$(\log n)^{\gamma}$, then \eqref{conj_misspec} reduces to 
\begin{equation*}
    \inf_{V} \bigg((V + 1)^{2/3} \cdot n^{-2/3}
    + \inf_{\substack{f_0 \in \fstinf \\ \vinfxgb(f_0) \le V}} 
    \|f_0 - f^*\|_{p_0, 2}^2\bigg).
\end{equation*}
For additional simplicity, suppose that $p_0$ is the uniform density 
on $[-1, 1]^d$. Then, for the function \eqref{func_notfst}, one can 
show that for sufficiently large $V$, 
\begin{equation*}
    \inf_{\substack{f_0 \in \fstinf \\ \vinfxgb(f_0) \le V}} 
    \|f_0 - f^*\|_{p_0, 2}^2 = \Omega(V^{-1/(d - 1)}). 
\end{equation*}
Consequently, even in this most favorable scenario (with 
$\beta = 2/3$), the convergence rate of the idealized XGBoost 
estimator for this $f^*$ is no faster than 
\begin{equation*}
    \inf_{V} \big((V + 1)^{2/3} \cdot n^{-2/3}
    + V^{-1/(d - 1)}\big) \asymp  n^{-2/(2d + 1)}.  
\end{equation*}
Unlike the curse-of-dimensionality-avoiding rate achieved when 
$f^* \in \fstinf$ with bounded $\vinfxgb(f^*)$, the above rate 
deteriorates rapidly as the dimension $d$ increases. This suggests 
that while XGBoost may still achieve consistency under 
misspecification, it is not well suited for estimating functions 
that lie far outside the class $\fstinf$.

\subsection{$L^1$ vs $L^2$ Regularization}
\label{sec:squared-l2-penalty}

We have focused on the $L^1$ penalty for XGBoost, as defined in
\eqref{eq:vxgb-original-def}. As mentioned in the Introduction,
XGBoost implementations also commonly employ a squared $L^2$ penalty, 
in which $\|\wbf_k\|_1$ is replaced by $\|\wbf_k\|_2^2$. 
More generally, for any $p \ge 1$ and $f \in \fst$, we may define
\begin{equation*}
    \vxgb(f; p) := \inf \Big\{\sum_k \|\wbf_k\|_p^p \Big\},
\end{equation*}
where $\|\cdot\|_p$ denotes the usual $L^p$ norm and, as in
\eqref{eq:vxgb-original-def}, the infimum is taken over all
representations of $f$ as a finite sum of right-continuous regression 
trees of depth at most $s$. However, this variation functional yields 
a meaningful regularization penalty only when $p = 1$. Specifically, 
when $p > 1$, the penalty becomes degenerate, as shown by the 
following result (proved in Appendix \ref{pf:zero-xgb-penalty}). 
This degeneracy explains why we restrict attention to the $L^1$ 
penalty in this paper.
\begin{lemma}\label{lem:zero-xgb-penalty}
    Suppose $p > 1$. Then, $\vxgb(f; p) = 0$ for every $f \in \fst$. 
\end{lemma}
In practice, XGBoost operates on the function class $\fst(K)$
consisting of functions of the form  $\sum_{k=1}^K f_k$, where each
$f_k$ is a regression tree with right-continuous splits and depth 
$\le s$. Here, $K$ is a fixed finite number that is typically 
selected via cross-validation. The distinction between $\fst$ and 
$\fst(K)$ is that the former allows an arbitrary number of trees, 
whereas the latter restricts attention to ensembles of at most $K$ 
trees.

Within this restricted class $\fst(K)$, we can define a truncated
version of the penalty by 
\begin{equation*}
    \vxgb(f; p, K) := \inf \Big\{\sum_{k=1}^K \|\wbf_k\|_p^p \Big\},
\end{equation*}
where the infimum is now taken over all representations of 
$f \in \fst(K)$ as a sum of at most $K$ regression trees of depth 
at most $s$. This modified penalty is likely well defined for all 
$p \ge 1$, including $p = 2$. However, it is theoretically cumbersome 
due to its rigid dependence on the hyperparameter $K$. Specifically, 
this formulation does not admit a meaningful limit as 
$K \rightarrow \infty$. Moreover, because $K$ is data-dependent in 
practice, it is unnatural to treat it as a fixed number.

For these reasons, we focus exclusively on the case $p = 1$, as this
choice provides a stable regularization penalty that generalizes
naturally to continuum tree ensembles and avoids the 
vanishing-penalty issues inherent to norms with $p > 1$ in the 
absence of a fixed tree count.

\subsection{Analysis of the Iterative Algorithm Used by XGBoost}
\label{algo_open} 

Our analysis focuses on the statistical behavior of solutions to the
regularized optimization problem \eqref{xgb_opti} that XGBoost is 
designed to approximate. We do not study whether the greedy 
tree-boosting algorithm employed in practice achieves the same rates 
of convergence over the class $\fstinf$, and establishing such 
guarantees remains an important open problem. Some recent progress 
has been made in the analysis of greedy tree-building algorithms; 
see, for example, \citet{tan2024statistical}.

Despite this limitation, our results remain directly relevant to the
practice of XGBoost. By characterizing the behavior of the target
optimization problem, our theory provides a principled benchmark for
what XGBoost can achieve under favorable optimization. In particular,
the results clarify when dimension-free rates are attainable
and when intrinsic approximation barriers arise due to
misspecification. This perspective helps disentangle statistical
limitations—stemming from the expressiveness of the tree ensemble
and its associated regularization—from algorithmic limitations of
the greedy boosting procedure itself, thereby offering a coherent
framework for interpreting the empirical successes and failures of
XGBoost in practice.

\section*{Acknowledgements}
The authors thank Ryan Tibshirani, Erez Buchweitz, 
and Arnab Mitra for helpful discussions.

\section*{Funding}
The authors gratefully acknowledge support from NSF Grants 
DMS-2210504 and DMS-2515470.

\bibliographystyle{chicago}
\bibliography{main}

\newpage
\appendix

\section{Proofs}\label{sec:proofs}

\subsection{Proofs of Propositions and Theorem in 
Sections \ref{sec:fstinf} and \ref{sec:vinfxgb}}

\subsubsection{Proof of Proposition \ref{prop:alt-charac-fst}}
\label{pf:alt-charac-fst}
\begin{proof}[Proof of Proposition \ref{prop:alt-charac-fst}]
We have already seen that any element of $\fst$ admits a 
representation of the form \eqref{f_fst} with discrete signed 
measures $\nu_{L, U}$ with finite support. It 
therefore suffices to prove the converse inclusion.

Observe that each basis function $b^{L, U}_{\lbf, \ubf}$ can be 
viewed as a regression tree with right-continuous splits and depth 
at most $s$, whose leaf weights are all zero except for a single 
leaf with weight one. Consequently, each basis function 
$b^{L, U}_{\lbf, \ubf}$ belongs to $\fst$. Since $\fst$ is closed 
under addition and scalar multiplication, it follows that any finite 
linear combination of these basis functions—and hence any function 
of the form \eqref{f_fst} with discrete signed measures $\nu_{L, U}$ 
having finite support—also belongs to $\fst$. This completes the proof.
\end{proof}

\subsubsection{Proof of Proposition \ref{prop:vinfxgb-stabilization}}
\label{pf:vinfxgb-stabilization}
\begin{proof}[Proof of Proposition \ref{prop:vinfxgb-stabilization}]
Recall that the sum in \eqref{f_fst} ranges over all
$L, U \subseteq [d]$ with $0 < |L| + |U| \le s$. Hence, for each function 
$f$, the set of admissible representations $\fcnu \equiv f$ enlarges 
as $s$ increases. Since $\vinfxgb(\cdot)$ is defined as an infimum over 
these representations, this gives (a).

Since $|L| + |U|$ is always bounded by $2d$, increasing $s$ beyond $2d$ 
does not enlarge the set of admissible representations. Consequently, 
$\vinfxgb(\cdot)$ stabilizes once $s \ge 2d$, which proves (b).

Lastly, (c) follows from the convexity of total variation 
$\|\cdot\|_{\text{TV}}$ on the space of finite signed Borel measures.
\end{proof}

\subsubsection{Proof of Theorem \ref{thm:complexity-equivalence}}
\label{pf:complexity-equivalence}
Before proving the theorem, we first observe and prove the following 
alternative characterization of $\vxgb(\cdot)$, originally defined 
via \eqref{eq:vxgb-original-def}.

\begin{lemma}\label{lem:alt-charac-vxgb}
    The complexity measure $\vxgb(\cdot)$ can be alternatively 
    characterized as
    \begin{equation*}
    \begin{aligned}
        &\vxgb(f) = \inf \bigg\{
        \sum_{0 < |L| + |U| \le s} \|\nu_{L, U}\|_{\text{TV}}: 
        \fcnu \equiv f \text{ and } \\
        &\qquad \qquad \qquad \qquad \qquad \qquad \quad
        \nu_{L, U} \text{ are discrete signed measures with 
        finitely many support points} \bigg\}.
    \end{aligned}
    \end{equation*}
\end{lemma}

Note that the only difference from the definition \eqref{vinfxgb_def} 
of $\vinfxgb(\cdot)$ is that the signed measures $\nu_{L, U}$ are 
restricted to be discrete with finite support. We will show in the 
proof of Theorem \ref{thm:complexity-equivalence} that this 
additional restriction does not affect the value of the infimum for 
functions in $\fst$.

\begin{proof}[Proof of Lemma \ref{lem:alt-charac-vxgb}]
Suppose first that all $\nu_{L, U}$ are discrete signed measures with 
finitely many support points. Then, $\fcnu$ is a finite linear 
combination of the basis functions $b^{L, U}_{\lbf, \ubf}$ with 
coefficients $\nu_{L, U}(\{(\lbf, \ubf)\})$ (plus a constant). Recall 
that each basis function $b^{L, U}_{\lbf, \ubf}$ can be viewed as a 
regression tree with right-continuous splits and depth at most $s$, 
whose leaf weights are all zero except for a single leaf with weight one. 
Consequently, $\fcnu$ can be represented as a finite sum of regression 
trees of the same type, whose leaf weight vectors each contain a single 
nonzero entry equal to $\nu_{L, U}(\{(\lbf, \ubf)\})$. For this 
representation, the total $\ell^1$ norm of the leaf weight vectors is 
exactly equal to the sum of the total variations of $\nu_{L, U}$. This 
proves that the infimum in the lemma is greater than or equal to the 
infimum in \eqref{eq:vxgb-original-def}. 

Now, suppose that $f \in \fst$ and that it is represented as a finite sum of 
regression trees with right-continuous splits and depth at most $s$. Let 
$\wbf_k$ denote the leaf weight vector of the $k^{\text{th}}$ tree. 
By decomposing each tree into the basis functions $b^{L, U}_{\lbf, \ubf}$ 
corresponding to paths from the root to each leaf, we obtain a 
representation of $f$ as a finite linear combination of these basis 
functions whose coefficient vector has $\ell^1$ norm no larger than 
$\sum_k \|\wbf_k\|_1$. Equivalently, there exists a representation 
$\fcnu \equiv f$ with discrete signed measures $\nu_{L, U}$ having finite 
support such that the sum of the total variations of $\nu_{L, U}$ is no 
larger than $\sum_k \|\wbf_k\|_1$. This shows that the infimum in 
the lemma is no larger than the infimum in \eqref{eq:vxgb-original-def}.
\end{proof}

\begin{proof}[Proof of Theorem \ref{thm:complexity-equivalence}]
We begin by introducing some notation used in the proof. For a function 
$g:[0, 1]^d \to \R$ and a nonempty subset $S \subseteq [d]$, define
\begin{equation*}
    g^S(x_j, j \in S) 
    := \lim_{(x_j, j \in S^c) \to (-\infty, j \in S^c)} 
    g(x_1, \dots, x_d) 
    \qquad \text{for } (x_j, j \in S) \in \R^{|S|}
\end{equation*}
whenever the limit exists, where $S^c = [d] \setminus S$.
    
Fix $f \in \fst$. 
Since $f$ is a finite sum of regression trees, there exists a 
partition $-\infty = v^{(j)}_0 < v^{(j)}_1 < \cdots < v^{(j)}_{n_j} 
< v^{(j)}_{n_j + 1} = +\infty$ of $\R$ for each $j \in [d]$ such that
$f$ is constant on 
\begin{equation}\label{eq:constant-partition}
    \prod_{j \in S} (v^{(j)}_{m_j}, v^{(j)}_{m_j + 1}) 
    \times \prod_{j \in S^c} \{v^{(j)}_{m_j}\}
\end{equation}
for every nonempty $S \subseteq [d]$, 
$m_j \in \{0, \dots, n_j\}$ for $j \in S$, and 
$m_j \in \{1, \dots, n_j\}$ for $j \in S^c$.
    
For each nonempty $S \subseteq [d]$ with $|S| \le s$ and
$\mbf = (m_j, j \in S) \in \prod_{j \in S} [n_j]$, define the 
alternating-sum functional 
\begin{equation*}
    \Delta^S_{\mbf}(g) 
    := \lim_{\epsilon \to 0+} \sum_{\deltabf \in \{0, 1\}^{|S|}}
    (-1)^{\sum_{j \in S} \delta_j} \cdot
    g^S(v^{(j)}_{m_j} - \delta_j \epsilon, j \in S\big)
\end{equation*}
for piecewise constant functions $g$ as in \eqref{eq:constant-partition}. 
    
Suppose $\fcnu \equiv f$. Then, clearly, we have
\begin{equation}\label{eq:alternating-sum-condition}
    \Delta^S_{\mbf}(\fcnu) = \Delta^S_{\mbf}(f)
\end{equation}
for all nonempty $S \subseteq [d]$ with $|S| \le s$ and 
$\mbf \in \prod_{j \in S} [n_j]$.
In fact, condition \eqref{eq:alternating-sum-condition} captures almost 
all of the information contained in the identity $\fcnu \equiv f$. The 
following lemma, whose proof is given after the current proof, makes this 
precise. This lemma will play an important role later.
\begin{lemma}\label{lem:alternating-sum-condition-equivalence}
    If $g, h \in \fst$ are piecewise constant as in 
    \eqref{eq:constant-partition} and satisfy
    \begin{equation*}
        \Delta^S_{\mbf}(g) 
        = \Delta^S_{\mbf}(h)
    \end{equation*}
    for all nonempty $S \subseteq [d]$ with $|S| \le s$ and
    $\mbf \in \prod_{j \in S} [n_j]$, then $g$ and $h$ differ 
    only by an additive constant; that is, there exists $b \in \R$ such that 
    $g(\xbf) = h(\xbf) + b$ for all $\xbf \in \R^d$.
\end{lemma}

We simplify \eqref{eq:alternating-sum-condition} and express it 
in terms of $\nu_{L, U}$ more explicitly. 
Fix a nonempty $S \subseteq [d]$ with $|S| \le s$ and 
$\mbf = (m_j, j \in S) \in \prod_{j \in S} [n_j]$.
Expanding the left-hand side of \eqref{eq:alternating-sum-condition} gives
\begingroup
\allowdisplaybreaks
\begin{align}\label{eq:alternating-sum-expansion}
    \Delta^S_{\mbf}(\fcnu) 
    &= \sum_{\substack{L, U: L \subseteq S \subseteq L \cup U 
    \\ |L| + |U| \le s}}
    \lim_{\epsilon \to 0+}
    \sum_{\deltabf \in \{0, 1\}^{|S|}} (-1)^{\sum_{j \in S} \delta_j} \\
    &\qquad \qquad \qquad \qquad
    \cdot \int_{\R^{|L| + |U|}} \prod_{j \in L} 
    \ind\big(v^{(j)}_{m_j} - \delta_j \epsilon \ge l_j\big)
    \cdot \prod_{j \in U \cap S} 
    \ind\big(v^{(j)}_{m_j} - \delta_j \epsilon < u_j\big) 
    \, d\nu_{L, U}(\lbf, \ubf). \nonumber
\end{align}
\endgroup
The inner limit can be simplified by exchanging the order of summation and 
integration and analyzing each indicator term. Specifically, 
\begin{equation*}
\begin{aligned}
    &\lim_{\epsilon \to 0+}
    \sum_{\deltabf \in \{0, 1\}^{|S|}} (-1)^{\sum_{j \in S} \delta_j} 
    \int_{\R^{|L| + |U|}} \prod_{j \in L} 
    \ind\big(v^{(j)}_{m_j} - \delta_j \epsilon \ge l_j\big)
    \cdot \prod_{j \in U \cap S} 
    \ind\big(v^{(j)}_{m_j} - \delta_j \epsilon < u_j\big) 
    \, d\nu_{L, U}(\lbf, \ubf) \\
    &\ \ = \lim_{\epsilon \to 0+}
    \int_{\R^{|L| + |U|}} \prod_{j \in L \setminus U} \big\{
    \ind\big(v^{(j)}_{m_j} \ge l_j \big)
    - \ind\big(v^{(j)}_{m_j} - \epsilon \ge l_j\big)\big\} \cdot \prod_{j \in L \cap U} \big\{
    \ind\big(l_j \le v^{(j)}_{m_j} < u_j \big)
    - \ind\big(l_j \le v^{(j)}_{m_j} 
    - \epsilon < u_j\big)\big\} \\
    &\qquad \qquad \qquad \qquad 
    \cdot \prod_{j \in (U \setminus L) \cap S} \big\{
    \ind\big(v^{(j)}_{m_j} < u_j \big)
    - \ind\big(v^{(j)}_{m_j} - \epsilon < u_j\big)\big\}
    \, d\nu_{L, U}(\lbf, \ubf) \\
    &\ \ = (-1)^{|(U \setminus L) \cap S|} 
    \sum_{K \subseteq L \cap U} (-1)^{|(L \cap U) \setminus K|} \cdot 
    \nu_{L, U} \Big(\prod_{j \in (L \setminus U) \cup K} 
    \{v^{(j)}_{m_j}\} \times \prod_{j \in (L \cap U) \setminus K} 
    (-\infty, v^{(j)}_{m_j}) \\
    &\qquad \qquad \qquad \qquad \qquad \qquad \qquad \qquad \qquad 
    \qquad \quad
    \times \prod_{j \in K} (v^{(j)}_{m_j}, +\infty) \times 
    \prod_{\substack{j \in ((L \cap U) \setminus K) \\
    \quad \cup ((U \setminus L) \cap S)}} 
    \{v^{(j)}_{m_j}\} \times \R^{|U \setminus S|}\Big).
\end{aligned}
\end{equation*}
Thus, condition \eqref{eq:alternating-sum-condition}, which holds for 
all nonempty $S \subseteq [d]$ with $|S| \le s$ and
$\mbf = (m_j, j \in S) \in \prod_{j \in S} [n_j]$ 
provided that $\fcnu \equiv f$, can be written as
\begin{align}\label{eq:equivalence-necessary-condition}
    &\sum_{\substack{L, U: L \subseteq S \subseteq L \cup U 
    \\ |L| + |U| \le s}} (-1)^{|(U \setminus L) \cap S|} 
    \sum_{K \subseteq L \cap U} (-1)^{|(L \cap U) \setminus K|} 
    \nonumber \\
    &\qquad \quad \cdot \nu_{L, U} \Big(\prod_{j \in (L \setminus U) \cup K} 
    \{v^{(j)}_{m_j}\} \times \prod_{j \in (L \cap U) \setminus K} 
    (-\infty, v^{(j)}_{m_j})
    \times \prod_{j \in K} (v^{(j)}_{m_j}, +\infty) \times
    \prod_{\substack{j \in ((L \cap U) \setminus K) \\
    \quad \cup ((U \setminus L) \cap S)}} 
    \{v^{(j)}_{m_j}\} \times \R^{|U \setminus S|}\Big) \nonumber \\
    &\qquad = \Delta^S_{\mbf}(f).
\end{align}
Consequently, we have
\begin{align}\label{eq:reduction-based-on-alt-sums}
    \vinfxgb(f) &\ge \inf \bigg\{
    \sum_{0 < |L| + |U| \le s} \|\nu_{L, U}\|_{\text{TV}}: 
    \nu_{L, U} 
    \text{ satisfy \eqref{eq:equivalence-necessary-condition}}\bigg\}.
\end{align}
    
Now, we show the infimum in \eqref{eq:reduction-based-on-alt-sums} is
achieved by discrete signed measures supported on 
\begin{equation}\label{eq:lattices-piecewise-constant}
    \prod_{j \in L} \{v^{(j)}_{1}, \dots, v^{(j)}_{n_j}\} 
    \times 
    \prod_{j \in U} \{v^{(j)}_{1}, \dots, v^{(j)}_{n_j}\}.
\end{equation}
Suppose $\nu_{L, U}$ are signed Borel measures satisfying 
\eqref{eq:equivalence-necessary-condition}. 
For each $j \in [d]$, let $V^{(j)} = \{v^{(j)}_{1}, \dots, v^{(j)}_{n_j}\}$, 
$\overline{V}^{(j)}_{m_j} = \{v^{(j)}_{m_j + 1}, \dots, v^{(j)}_{n_j}\}$, and 
$\underline{V}^{(j)}_{m_j} = \{v^{(j)}_{1}, \dots, v^{(j)}_{m_j - 1}\}$.
Define discrete signed measures $\mu_{L, U}$, supported on the 
lattices \eqref{eq:lattices-piecewise-constant}, by
\begin{equation*}
\begin{aligned}
    &\mu_{L, U}\big(\{(v^{(j)}_{p_j}, j \in L; 
    v^{(j)}_{q_j}, j \in U)\}\big)
    = \sum_{\substack{\widetilde{U} \supseteq U \\ 
    |\widetilde{U}| \le s - |L|}} 
    \sum_{T \subseteq (L \cap \widetilde{U}) \setminus U}
    (-1)^{|((L \cap \widetilde{U}) \setminus U) \setminus T|} \\
    &\qquad \qquad \qquad \quad
    \cdot \nu_{L, \widetilde{U}}
    \Big(\prod_{j \in L 
    \setminus (((L \cap \widetilde{U}) \setminus U) \setminus T)} 
    \{v^{(j)}_{p_j}\}  
    \times \prod_{j \in ((L \cap \widetilde{U}) \setminus U) \setminus T} 
    \big((- \infty, v^{(j)}_{p_j}) 
    \setminus \underline{V}^{(j)}_{p_j}\big) \\
    &\qquad \qquad \qquad \qquad \qquad 
    \times \prod_{j \in ((L \cap \widetilde{U})\setminus U) \setminus T} 
    \{v^{(j)}_{p_j}\}
    \times \prod_{j \in T} 
    \big((v^{(j)}_{p_j}, +\infty) \setminus \overline{V}^{(j)}_{p_j}\big)
    \times \prod_{j \in U} \{v^{(j)}_{q_j}\} 
    \times \prod_{j \in (\widetilde{U} \setminus L) \setminus U} 
    (\R \setminus V^{(j)})\Big)
\end{aligned}
\end{equation*}
for $(p_j, j \in L) \in \prod_{j \in L} [n_j]$ and 
$(q_j, j \in U) \in \prod_{j \in U} [n_j]$.
Observe that for $L \subseteq S \subseteq L \cup U$ with 
$|L| + |U| \le s$, 
\begin{equation*}
\begin{aligned}
    &\mu_{L, U} \Big(\prod_{j \in (L \setminus U) \cup K} 
    \{v^{(j)}_{m_j}\} \times \prod_{j \in (L \cap U) \setminus K} 
    (-\infty, v^{(j)}_{m_j})
    \times \prod_{j \in K} (v^{(j)}_{m_j}, +\infty) \times
    \prod_{\substack{j \in ((L \cap U) \setminus K) \\
    \quad \cup ((U \setminus L) \cap S)}} 
    \{v^{(j)}_{m_j}\} \times \R^{|U \setminus S|}\Big) \\ 
    &\quad = \sum_{\rbf \in \prod\limits_{j \in (L \cap U) \setminus K} 
    \underline{V}^{(j)}_{m_j} 
    \times \prod\limits_{j \in K} \overline{V}^{(j)}_{m_j} 
    \times \prod\limits_{U \setminus S} V^{(j)}}  
    \mu_{L, U} \Big(\prod_{j \in (L \setminus U) \cup K} 
    \{v^{(j)}_{m_j}\} \times \prod_{j \in (L \cap U) \setminus K} 
    \{v^{(j)}_{r_j}\} \\
    &\qquad \qquad \qquad \qquad \qquad \qquad \qquad \qquad \qquad \quad
    \times \prod_{j \in K} \{v^{(j)}_{r_j}\} 
    \times \prod_{\substack{j \in ((L \cap U) \setminus K) \\
    \quad \cup ((U \setminus L) \cap S)}} \{v^{(j)}_{m_j}\} 
    \times \prod_{j \in U \setminus S} \{v^{(j)}_{r_j}\}\Big) \\
    &\quad = \sum_{\rbf \in \prod\limits_{j \in (L \cap U) \setminus K} 
    \underline{V}^{(j)}_{m_j} 
    \times \prod\limits_{j \in K} \overline{V}^{(j)}_{m_j} 
    \times \prod\limits_{U \setminus S} V^{(j)}}
    \sum_{\substack{\widetilde{U} \supseteq U \\ 
    |\widetilde{U}| \le s - |L|}} 
    \sum_{T \subseteq (L \cap \widetilde{U}) \setminus U}
    (-1)^{|((L \cap \widetilde{U}) \setminus U) \setminus T|} \\
    &\qquad \qquad \qquad \qquad \quad
    \cdot \nu_{L, \widetilde{U}}
    \Big(\prod_{j \in (L \setminus \widetilde{U}) \cup K \cup T} 
    \{v^{(j)}_{m_j}\}  
    \times \prod_{j \in (L \cap U) \setminus K} 
    \{v^{(j)}_{r_j}\} 
    \times \prod_{j \in ((L \cap \widetilde{U}) \setminus U) \setminus T} 
    \big((- \infty, v^{(j)}_{m_j}) 
    \setminus \underline{V}^{(j)}_{m_j}\big) \\
    &\qquad \qquad \qquad \qquad \qquad \qquad \quad
    \times \prod_{j \in ((L \cap \widetilde{U}) \setminus U) \setminus T} 
    \{v^{(j)}_{m_j}\}
    \times \prod_{j \in T} 
    \big((v^{(j)}_{m_j}, +\infty) \setminus \overline{V}^{(j)}_{m_j}\big) \\
    &\qquad \qquad \qquad \qquad \qquad \qquad \quad
    \times \prod_{j \in K} \{v^{(j)}_{r_j}\} 
    \times \prod_{\substack{j \in ((L \cap U) \setminus K) \\
    \quad \cup ((U \setminus L) \cap S)}} \{v^{(j)}_{m_j}\} 
    \times \prod_{j \in U \setminus S} \{v^{(j)}_{r_j}\}
    \times \prod_{j \in (\widetilde{U} \setminus L) \setminus U} 
    (\R \setminus V^{(j)})\Big) \\
    &\quad = \sum_{\substack{\widetilde{U} \supseteq U \\ 
    |\widetilde{U}| \le s - |L|}} 
    \sum_{T \subseteq (L \cap \widetilde{U}) \setminus U}
    (-1)^{|((L \cap \widetilde{U}) \setminus U) \setminus T|} \\
    &\qquad \qquad \qquad \qquad \quad
    \cdot \nu_{L, \widetilde{U}}
    \Big(\prod_{j \in (L \setminus \widetilde{U}) \cup K \cup T} 
    \{v^{(j)}_{m_j}\}  
    \times \prod_{j \in (L \cap U) \setminus K} 
    \underline{V}^{(j)}_{m_j}
    \times \prod_{j \in ((L \cap \widetilde{U}) \setminus U) \setminus T} 
    \big((- \infty, v^{(j)}_{m_j}) 
    \setminus \underline{V}^{(j)}_{m_j}\big) \\
    &\qquad \qquad \qquad \qquad \qquad \qquad \quad
    \times \prod_{j \in ((L \cap \widetilde{U}) \setminus U) \setminus T} 
    \{v^{(j)}_{m_j}\}
    \times \prod_{j \in T} 
    \big((v^{(j)}_{m_j}, +\infty) \setminus \overline{V}^{(j)}_{m_j}\big) \\
    &\qquad \qquad \qquad \qquad \qquad \qquad \quad
    \times \prod_{j \in K} \overline{V}^{(j)}_{m_j} 
    \times \prod_{\substack{j \in ((L \cap U) \setminus K) \\
    \quad \cup ((U \setminus L) \cap S)}} \{v^{(j)}_{m_j}\} 
    \times \prod_{j \in U \setminus S} V^{(j)}
    \times \prod_{j \in (\widetilde{U} \setminus L) \setminus U} 
    (\R \setminus V^{(j)})\Big).
\end{aligned}
\end{equation*}
Hence, the left-hand side of \eqref{eq:equivalence-necessary-condition} 
for $\mu_{L, U}$ becomes
\begin{equation*}
\begin{aligned}
    &\sum_{\substack{L, U: L \subseteq S \subseteq L \cup U 
    \\ |L| + |U| \le s}} (-1)^{|(U \setminus L) \cap S|} 
    \sum_{K \subseteq L \cap U} (-1)^{|(L \cap U) \setminus K|} \\
    &\qquad \quad \cdot \mu_{L, U} \Big(\prod_{j \in (L \setminus U) \cup K} 
    \{v^{(j)}_{m_j}\} \times \prod_{j \in (L \cap U) \setminus K} 
    (-\infty, v^{(j)}_{m_j})
    \times \prod_{j \in K} (v^{(j)}_{m_j}, +\infty) \times
    \prod_{\substack{j \in ((L \cap U) \setminus K) \\
    \quad \cup ((U \setminus L) \cap S)}} 
    \{v^{(j)}_{m_j}\} \times \R^{|U \setminus S|}\Big) \\ 
    &\quad = \sum_{\substack{L, U: L \subseteq S \subseteq L \cup U 
    \\ |L| + |U| \le s}} (-1)^{|(U \setminus L) \cap S|} 
    \sum_{K \subseteq L \cap U} (-1)^{|(L \cap U) \setminus K|}
    \sum_{\substack{\widetilde{U} \supseteq U \\ 
    |\widetilde{U}| \le s - |L|}} 
    \sum_{T \subseteq (L \cap \widetilde{U}) \setminus U}
    (-1)^{|((L \cap \widetilde{U}) \setminus U) \setminus T|} \\
    &\qquad \qquad \qquad \qquad \qquad \quad
    \cdot \nu_{L, \widetilde{U}}
    \Big(\prod_{j \in (L \setminus \widetilde{U}) \cup K \cup T} 
    \{v^{(j)}_{m_j}\}  
    \times \prod_{j \in (L \cap U) \setminus K} 
    \underline{V}^{(j)}_{m_j}
    \times \prod_{j \in ((L \cap \widetilde{U}) \setminus U) \setminus T} 
    \big((- \infty, v^{(j)}_{m_j}) 
    \setminus \underline{V}^{(j)}_{m_j}\big) \\
    &\qquad \qquad \qquad \qquad \qquad \qquad \qquad \quad
    \times \prod_{j \in ((L \cap \widetilde{U}) \setminus U) \setminus T} 
    \{v^{(j)}_{m_j}\}
    \times \prod_{j \in T} 
    \big((v^{(j)}_{m_j}, +\infty) \setminus \overline{V}^{(j)}_{m_j}\big) \\
    &\qquad \qquad \qquad \qquad \qquad \qquad \qquad \quad
    \times \prod_{j \in K} \overline{V}^{(j)}_{m_j} 
    \times \prod_{\substack{j \in ((L \cap U) \setminus K) \\
    \quad \cup ((U \setminus L) \cap S)}} \{v^{(j)}_{m_j}\} 
    \times \prod_{j \in U \setminus S} V^{(j)}
    \times \prod_{j \in (\widetilde{U} \setminus L) \setminus U} 
    (\R \setminus V^{(j)})\Big).
\end{aligned}
\end{equation*}
By changing the order of summation in $U$ and $\widetilde{U}$ and 
combining the summations in $K$ and $T$ via $\widetilde{K} = K \cup T$,
we can simplify the right-hand side to
\begin{equation*}
\begin{aligned}
    &\sum_{\substack{L, \widetilde{U}: 
    L \subseteq S \subseteq L \cup \widetilde{U} 
    \\ |L| + |\widetilde{U}| \le s}} 
    (-1)^{|(\widetilde{U} \setminus L) \cap S|} 
    \sum_{U: S \setminus L \subseteq U \subseteq \widetilde{U}}
    \sum_{K \subseteq L \cap U} 
    \sum_{T \subseteq (L \cap \widetilde{U}) \setminus U}
    (-1)^{|(L \cap U) \setminus K|} \cdot 
    (-1)^{|((L \cap \widetilde{U}) \setminus U) \setminus T|} \\
    &\qquad \qquad \qquad \qquad \qquad
    \cdot \nu_{L, \widetilde{U}}
    \Big(\prod_{j \in (L \setminus \widetilde{U}) \cup K \cup T} 
    \{v^{(j)}_{m_j}\}  
    \times \prod_{j \in (L \cap U) \setminus K} 
    \underline{V}^{(j)}_{m_j}
    \times \prod_{j \in ((L \cap \widetilde{U}) \setminus U) \setminus T} 
    \big((- \infty, v^{(j)}_{m_j}) 
    \setminus \underline{V}^{(j)}_{m_j}\big) \\
    &\qquad \qquad \qquad \qquad \qquad \qquad \qquad
    \times \prod_{j \in ((L \cap \widetilde{U}) \setminus U) \setminus T} 
    \{v^{(j)}_{m_j}\}
    \times \prod_{j \in T} 
    \big((v^{(j)}_{m_j}, +\infty) \setminus \overline{V}^{(j)}_{m_j}\big) \\
    &\qquad \qquad \qquad \qquad \qquad \qquad \qquad
    \times \prod_{j \in K} \overline{V}^{(j)}_{m_j} 
    \times \prod_{\substack{j \in ((L \cap U) \setminus K) \\
    \quad \cup ((U \setminus L) \cap S)}} \{v^{(j)}_{m_j}\} 
    \times \prod_{j \in U \setminus S} V^{(j)}
    \times \prod_{j \in (\widetilde{U} \setminus L) \setminus U} 
    (\R \setminus V^{(j)})\Big) \\
    &\quad = \sum_{\substack{L, \widetilde{U}: 
    L \subseteq S \subseteq L \cup \widetilde{U} 
    \\ |L| + |\widetilde{U}| \le s}} 
    (-1)^{|(\widetilde{U} \setminus L) \cap S|} 
    \sum_{\widetilde{K} \subseteq L \cap \widetilde{U}} 
    (-1)^{|(L \cap \widetilde{U}) \setminus \widetilde{K}|} 
    \sum_{U: S \setminus L \subseteq U \subseteq \widetilde{U}} \\
    &\qquad \qquad \qquad \qquad \qquad
    \cdot \nu_{L, \widetilde{U}}
    \Big(\prod_{j \in (L \setminus \widetilde{U}) \cup \widetilde{K}} 
    \{v^{(j)}_{m_j}\}  
    \times \prod_{j \in ((L \cap \widetilde{U}) 
    \setminus \widetilde{K}) \cap U} \underline{V}^{(j)}_{m_j}
    \times \prod_{j \in ((L \cap \widetilde{U}) 
    \setminus \widetilde{K}) \setminus U} \big((- \infty, v^{(j)}_{m_j}) 
    \setminus \underline{V}^{(j)}_{m_j}\big) \\
    &\qquad \qquad \qquad \qquad \qquad\qquad \qquad
    \times \prod_{j \in \widetilde{K} \setminus U} 
    \big((v^{(j)}_{m_j}, +\infty) \setminus \overline{V}^{(j)}_{m_j}\big) 
    \times \prod_{j \in \widetilde{K} \cap U} \overline{V}^{(j)}_{m_j} 
    \times \prod_{\substack{j \in 
    ((L \cap \widetilde{U}) \setminus \widetilde{K}) \\
    \quad \cup ((\widetilde{U} \setminus L) \cap S)}} \{v^{(j)}_{m_j}\} \\
    &\qquad \qquad \qquad \qquad \qquad \qquad \qquad
    \times \prod_{j \in U \setminus S} V^{(j)}
    \times \prod_{j \in (\widetilde{U} \setminus L) \setminus U} 
    (\R \setminus V^{(j)})\Big).
\end{aligned}
\end{equation*}
Computing the inner summation over $U$ yields
\begin{equation*}
\begin{aligned}
    &\sum_{\substack{L, U: L \subseteq S \subseteq L \cup U 
    \\ |L| + |U| \le s}} (-1)^{|(U \setminus L) \cap S|} 
    \sum_{K \subseteq L \cap U} (-1)^{|(L \cap U) \setminus K|} \\
    &\qquad \quad \cdot \mu_{L, U} \Big(\prod_{j \in (L \setminus U) \cup K} 
    \{v^{(j)}_{m_j}\} \times \prod_{j \in (L \cap U) \setminus K} 
    (-\infty, v^{(j)}_{m_j})
    \times \prod_{j \in K} (v^{(j)}_{m_j}, +\infty) \times
    \prod_{\substack{j \in ((L \cap U) \setminus K) \\
    \quad \cup ((U \setminus L) \cap S)}} 
    \{v^{(j)}_{m_j}\} \times \R^{|U \setminus S|}\Big) \\
    &\quad= \sum_{\substack{L, \widetilde{U}: 
    L \subseteq S \subseteq L \cup \widetilde{U} 
    \\ |L| + |\widetilde{U}| \le s}} 
    (-1)^{|(\widetilde{U} \setminus L) \cap S|} 
    \sum_{\widetilde{K} \subseteq L \cap \widetilde{U}} 
    (-1)^{|(L \cap \widetilde{U}) \setminus \widetilde{K}|} \\ 
    &\qquad \qquad
    \cdot \nu_{L, \widetilde{U}}
    \Big(\prod_{j \in (L \setminus \widetilde{U}) \cup \widetilde{K}} 
    \{v^{(j)}_{m_j}\}  
    \times \prod_{j \in ((L \cap \widetilde{U}) 
    \setminus \widetilde{K})} (- \infty, v^{(j)}_{m_j})
    \times \prod_{j \in \widetilde{K}} (v^{(j)}_{m_j}, +\infty) 
    \times \prod_{\substack{j \in 
    ((L \cap \widetilde{U}) \setminus \widetilde{K}) \\
    \quad \cup ((\widetilde{U} \setminus L) \cap S)}} \{v^{(j)}_{m_j}\}
    \times \R^{|\widetilde{U} \setminus S|}\Big) \\
    &\quad= \Delta^S_{\mbf}(f), 
\end{aligned}
\end{equation*}
which shows that \eqref{eq:equivalence-necessary-condition} also holds 
for $\mu_{L, U}$.
    
Moreover, by the definition of $\mu_{L, U}$, we have
\begingroup
\allowdisplaybreaks
\begin{align*}
    &\sum_{0 < |L| + |U| \le s} |\mu_{L, U}| (\R^{|L| + |U|}) 
    = \sum_{0 < |L| + |U| \le s} 
    \sum_{\pbf \in \prod_{j \in L} [n_j]} 
    \sum_{\qbf \in \prod_{j \in U} [n_j]}
    \big|\mu_{L, U}\big(\{(v^{(j)}_{p_j}, j \in L)
    \times (v^{(j)}_{q_j}, j \in U)\}\big)\big| \\
    &\qquad \le \sum_{0 < |L| + |U| \le s} 
    \sum_{\pbf \in \prod_{j \in L} [n_j]} 
    \sum_{\qbf \in \prod_{j \in U} [n_j]} 
    \sum_{\substack{\widetilde{U} \supseteq U \\ 
    |\widetilde{U}| \le s - |L|}} 
    \sum_{T \subseteq (L \cap \widetilde{U}) \setminus U} \\
    &\qquad \qquad \qquad \quad
    |\nu_{L, \widetilde{U}}|
    \Big(\prod_{j \in L 
    \setminus (((L \cap \widetilde{U}) \setminus U) \setminus T)} 
    \{v^{(j)}_{p_j}\}  
    \times \prod_{j \in ((L \cap \widetilde{U}) \setminus U) \setminus T} 
    \big((- \infty, v^{(j)}_{p_j}) 
    \setminus \underline{V}^{(j)}_{p_j}\big) \\
    &\qquad \qquad \qquad \qquad \qquad 
    \times \prod_{j \in ((L \cap \widetilde{U})\setminus U) \setminus T} 
    \{v^{(j)}_{p_j}\}
    \times \prod_{j \in T} 
    \big((v^{(j)}_{p_j}, +\infty) \setminus \overline{V}^{(j)}_{p_j}\big)
    \times \prod_{j \in U} \{v^{(j)}_{q_j}\} 
    \times \prod_{j \in (\widetilde{U} \setminus L) \setminus U} 
    (\R \setminus V^{(j)})\Big) \\
    &\qquad = \sum_{0 < |L| + |\widetilde{U}| \le s} 
    \sum_{U \subseteq \widetilde{U}} 
    \sum_{T \subseteq (L \cap \widetilde{U}) \setminus U} 
    \sum_{\pbf \in \prod_{j \in L} [n_j]} 
    \sum_{\qbf \in \prod_{j \in U} [n_j]} \\
    &\qquad \qquad \qquad \quad
    |\nu_{L, \widetilde{U}}|
    \Big(\prod_{j \in L 
    \setminus (((L \cap \widetilde{U}) \setminus U) \setminus T)} 
    \{v^{(j)}_{p_j}\}  
    \times \prod_{j \in ((L \cap \widetilde{U}) \setminus U) \setminus T} 
    \big((- \infty, v^{(j)}_{p_j}) 
    \setminus \underline{V}^{(j)}_{p_j}\big) \\
    &\qquad \qquad \qquad \qquad \qquad
    \times \prod_{j \in ((L \cap \widetilde{U})\setminus U) \setminus T} 
    \{v^{(j)}_{p_j}\}
    \times \prod_{j \in T} 
    \big((v^{(j)}_{p_j}, +\infty) \setminus \overline{V}^{(j)}_{p_j}\big)
    \times \prod_{j \in U} \{v^{(j)}_{q_j}\} 
    \times \prod_{j \in (\widetilde{U} \setminus L) \setminus U} 
    (\R \setminus V^{(j)})\Big) \\
    &\qquad \le \sum_{0 < |L| + |\widetilde{U}| \le s} 
    \sum_{U \subseteq \widetilde{U}} 
    \sum_{T \subseteq (L \cap \widetilde{U}) \setminus U} 
    \sum_{\pbf \in \prod_{j \in L} [n_j]} 
    \sum_{\qbf \in \prod_{j \in U} [n_j]} \\
    &\qquad \qquad \qquad \quad
    |\nu_{L, \widetilde{U}}|
    \Big(\prod_{j \in L 
    \setminus (((L \cap \widetilde{U}) \setminus U) \setminus T)} 
    \{v^{(j)}_{p_j}\}  
    \times \prod_{j \in ((L \cap \widetilde{U}) \setminus U) \setminus T} 
    (\R \setminus V^{(j)}) \\
    &\qquad \qquad \qquad \qquad \qquad
    \times \prod_{j \in ((L \cap \widetilde{U})\setminus U) \setminus T} 
    \{v^{(j)}_{p_j}\}
    \times \prod_{j \in T} (\R \setminus V^{(j)})
    \times \prod_{j \in U} \{v^{(j)}_{q_j}\} 
    \times \prod_{j \in (\widetilde{U} \setminus L) \setminus U} 
    (\R \setminus V^{(j)})\Big) \\
    &\qquad = \sum_{0 < |L| + |\widetilde{U}| \le s} 
    \sum_{U \subseteq \widetilde{U}} 
    \sum_{T \subseteq (L \cap \widetilde{U}) \setminus U} \\
    &\qquad \qquad \qquad \quad
    |\nu_{L, \widetilde{U}}|
    \Big(\prod_{j \in L 
    \setminus (((L \cap \widetilde{U}) \setminus U) \setminus T)} V^{(j)}  
    \times \prod_{j \in ((L \cap \widetilde{U}) \setminus U) \setminus T} 
    (\R \setminus V^{(j)}) \\
    &\qquad \qquad \qquad \qquad \qquad
    \times \prod_{j \in ((L \cap \widetilde{U})\setminus U) \setminus T} V^{(j)} 
    \times \prod_{j \in T} (\R \setminus V^{(j)})
    \times \prod_{j \in U} V^{(j)} 
    \times \prod_{j \in (\widetilde{U} \setminus L) \setminus U} 
    (\R \setminus V^{(j)})\Big) \\
    &\qquad \le \sum_{0 < |L| + |\widetilde{U}| \le s} 
    |\nu_{L, \widetilde{U}}|(\R^{|L| + |\widetilde{U}|}).
\end{align*}
\endgroup
Here, we change the order of summation in $U$ and $\widetilde{U}$ for 
the second equality, and for the second inequality, we use the fact that 
\begin{equation*}
    (-\infty, v^{(j)}_{p_j}) \setminus \underline{V}^{(j)}_{p_j}  
    \subseteq \R \setminus V^{(j)} 
    \ \text{ and } \
    (v^{(j)}_{p_j}, +\infty) \setminus \overline{V}^{(j)}_{p_j}  
    \subseteq \R \setminus V^{(j)}
    \text{ for all } j.
\end{equation*}
The last inequality follows from the observation that for each $L$ 
and $\widetilde{U}$, the sets
\begin{equation*}
\begin{aligned}
    &\prod_{j \in L 
    \setminus (((L \cap \widetilde{U}) \setminus U) \setminus T)} V^{(j)}  
    \times \prod_{j \in ((L \cap \widetilde{U}) \setminus U) \setminus T} 
    (\R \setminus V^{(j)}) \\
    &\qquad \quad  
    \times \prod_{j \in ((L \cap \widetilde{U})\setminus U) \setminus T} V^{(j)} 
    \times \prod_{j \in T} (\R \setminus V^{(j)})
    \times \prod_{j \in U} V^{(j)} 
    \times \prod_{j \in (\widetilde{U} \setminus L) \setminus U} 
    (\R \setminus V^{(j)}) 
\end{aligned}
\end{equation*}
are pairwise disjoint as $U$ ranges over subsets of $\widetilde{U}$ and 
$T$ ranges over subsets of $(L \cap \widetilde{U}) \setminus U$.
This shows that the objective function of 
\eqref{eq:reduction-based-on-alt-sums} for 
$\mu_{L, U}$ is no larger than that for $\nu_{L, U}$ and 
ensures that the infimum of \eqref{eq:reduction-based-on-alt-sums} is 
attained by some discrete signed measures 
$\mu_{L, U}$ supported on \eqref{eq:lattices-piecewise-constant}.
    
Suppose $\mu_{L, U}$ are discrete signed measures 
supported on the lattices \eqref{eq:lattices-piecewise-constant} that 
satisfy \eqref{eq:equivalence-necessary-condition}. 
We can parametrize these measures by 
\begin{equation}\label{eq:measures-discrete-parametrization-piecewise-constant}
    \mu_{L, U} \big(\{(v^{(j)}_{p_j}, j \in L)
    \times (v^{(j)}_{q_j}, j \in U)\}\big)
    = \beta^{L, U}_{\pbf, \qbf}
\end{equation}
for $L, U \subseteq [d]$ with $0 < |L| + |U| \le s$, 
$\pbf = (p_j, j \in L) \in \prod_{j \in L} [n_j]$, and
$\qbf = (q_j, j \in U) \in \prod_{j \in U} [n_j]$. 
Under this parametrization, condition 
\eqref{eq:equivalence-necessary-condition} can be written entirely in terms 
of $\beta^{L, U}_{\pbf, \qbf}$ as
\begingroup
\allowdisplaybreaks
\begin{align}\label{eq:equivalence-necessary-condition-discrete}
    &\sum_{\substack{L, U: L \subseteq S \subseteq L \cup U 
    \\ |L| + |U| \le s}} (-1)^{|(U \setminus L) \cap S|} 
    \sum_{K \subseteq L \cap U} (-1)^{|(L \cap U) \setminus K|} 
    \sum_{\rbf \in \prod\limits_{j \in (L \cap U) \setminus K} 
    \underline{V}^{(j)}_{m_j} 
    \times \prod\limits_{j \in K} \overline{V}^{(j)}_{m_j} 
    \times \prod\limits_{U \setminus S} V^{(j)}}  
    \nonumber \\
    &\qquad \qquad \qquad \quad \beta^{L, U}_{(m_j, j \in (L \setminus U) \cup K; 
    r_j, j \in (L \cap U) \setminus K) \times (r_j, j \in K; 
    m_j, j \in ((L \cap U) \setminus K) 
    \cup ((U \setminus L) \cap S); r_j, j \in U \setminus S)} 
    = \Delta^S_{\mbf}(f).
\end{align}
\endgroup
Consequently, \eqref{eq:reduction-based-on-alt-sums} becomes
\begin{equation*}
    \vinfxgb(f) \ge \inf \big\{\|\beta\|_1: 
    \beta^{L, U}_{\pbf, \qbf}
    \text{ satisfy 
    \eqref{eq:equivalence-necessary-condition-discrete}} \big\}.
\end{equation*}
Since \eqref{eq:equivalence-necessary-condition-discrete} is a system of 
linear equations, there clearly exist minimizers
$\beta^{L, U}_{\pbf, \qbf}$ of the right-hand side. Let 
$\hat{\beta}^{L, U}_{\pbf, \qbf}$ 
denote one such minimizer and let $\hat{\mu}_{L, U}$ be
the corresponding discrete signed measures defined via 
\eqref{eq:measures-discrete-parametrization-piecewise-constant}.
    
Choose any constant $c_0 \in \R$. By construction, the function
$f^{d, s}_{c_0, \{\hat{\mu}_{L, U}\}}$ is piecewise constant as 
in \eqref{eq:constant-partition} and satisfies 
\eqref{eq:alternating-sum-condition}. By Lemma 
\ref{lem:alternating-sum-condition-equivalence}, there exists a constant 
$b \in \R$ such that $f(\xbf) = b + f^{d, s}_{c_0, \{\hat{\mu}_{L, U}\}}(\xbf)$ 
for all $\xbf \in \R^d$. Defining $c = b + c_0$, 
we then have $f^{d, s}_{c, \{\hat{\mu}_{L, U}\}} \equiv f$. 
It follows that
\begin{equation*}
    \|\hat{\beta}\|_1 = \sum_{0 < |L| + |U| \le s} 
    \|\hat{\mu}_{L, U}\|_{\text{TV}} 
    \ge \vinfxgb(f) 
    \ge \|\hat{\beta}\|_1,
\end{equation*}
so that 
\begin{equation*}
    \vinfxgb(f) = \sum_{0 < |L| + |U| \le s} 
    \|\hat{\mu}_{L, U}\|_{\text{TV}}. 
\end{equation*}
Therefore, the minimum in the definition of $\vinfxgb(f)$ is attained by discrete 
signed measures. This proves $\vinfxgb(f) = \vxgb(f)$.
\end{proof}
    
\begin{proof}[Proof of Lemma \ref{lem:alternating-sum-condition-equivalence}]
Since the alternating-sum functional is linear, it suffices to show that
if $f \in \fst$ is piecewise constant as in \eqref{eq:constant-partition} and 
satisfies
\begin{equation}\label{eq:alternating-sum-zero-condition}
    \Delta^S_{\mbf}(f) = 0
\end{equation}
for all $\emptyset \neq S \subseteq [d]$ with $|S| \le s$ and
$\mbf \in \prod_{j \in S} [n_j]$, then $f$ is a constant function.
    
Fix such a $f \in \fst$.
As seen in \eqref{eq:alternating-sum-expansion}, the expansion of 
$\Delta^S_{\mbf}(f)$ involves the summation over 
$L \subseteq S \subseteq L \cup U$ with $|L| + |U| \le s$. When $|S| > s$, this 
summation is vacuous, 
so that $\Delta^S_{\mbf}(f) = 0$ holds automatically. Thus, 
\eqref{eq:alternating-sum-zero-condition} in fact holds for all nonempty
$S \subseteq [d]$.
    
For each $j \in [d]$, define 
\begin{equation*}
    w^{(j)}_0 = v^{(j)}_1 - 1 \ \text{ and } \
    w^{(j)}_{m_j} = v^{(j)}_{m_j} \text{ for } m_j \in [n_j],
\end{equation*}
and let $\phibf$ denote the vector of evaluations of $f$ at 
$(w^{(1)}_{m_1}, \dots, w^{(d)}_{m_d})$; that is,
\begin{equation*}
    \phi(\mbf) = f(w^{(1)}_{m_1}, \dots, w^{(d)}_{m_d})
    \quad \text{for } 
    \mbf = (m_1, \dots, m_d) \in \prod_{j = 1}^d \{0, \dots, n_j\}.
\end{equation*}
Because $f$ is piecewise constant as in \eqref{eq:constant-partition} 
and right-continuous, the vector $\phibf$ completely 
determines $f$. Moreover, for the same reason, 
\eqref{eq:alternating-sum-zero-condition} is equivalent to
\begin{equation}\label{eq:alternating-sum-zero-condition-vector}
    \Delta^S_{\mbf} \phibf 
    := \sum_{\deltabf \in \{0, 1\}^{|S|}} (-1)^{\sum_{j \in S} \delta_j} \cdot 
    \phi\big(m_j - \delta_j, j \in S; 0, j \in S^c\big)
    = 0
    \quad \text{for } 
    \mbf = (m_j, j \in S) \in \prod_{j \in S} [n_j]
\end{equation}
for all nonempty $S$.
Thus, it suffices to show that if $\phibf$ satisfies 
\eqref{eq:alternating-sum-zero-condition-vector},
then $\phibf$ is a constant vector.
    
We prove this claim by induction on $d$. The case $d = 1$ is straightforward. 
When $d = 1$, \eqref{eq:alternating-sum-zero-condition-vector} reduces to
\begin{equation*}
    \Delta^{\{1\}}_{m}\phibf = \phi(m) - \phi(m - 1) = 0
    \quad \text{for } m \in [n_1].
\end{equation*}
Hence, in this case, it is clear that $\phibf$ is a constant vector.
Suppose the claim holds for $d - 1$, and let us prove it for $d$.
For each $m_d \in \{0, \dots, n_d\}$, let $\phibf^{(m_d)}$ denote the 
subvector of $\phibf$ with last index $m_d$; that is,
\begin{equation*}
    \phi^{(m_d)}(m_1, \dots, m_{d - 1}) 
    = \phi(m_1, \dots, m_{d - 1}, m_d)
    \quad \text{for } (m_1, \dots, m_{d - 1}) \in 
    \prod_{j = 1}^{d - 1} \{0, \dots, n_j\}.
\end{equation*}
Clearly, $\phibf^{(0)}$ satisfies 
\eqref{eq:alternating-sum-zero-condition-vector} for $d - 1$. Note that
for $m_d \in [n_d]$, 
\begin{equation*}
    \Delta^{S}_{(m_1, \dots, m_{d - 1})}\phibf^{(m_d)}
    - \Delta^{S}_{(m_1, \dots, m_{d - 1})}\phibf^{(m_d - 1)}
    =  \Delta^{S \cup \{d\}}_{(m_1, \dots, m_{d - 1}, m_d)}\phibf
    = 0
\end{equation*}
for every nonempty $S \subseteq [d - 1]$ and $(m_1, \dots, m_{d - 1}) \in 
\prod_{j = 1}^{d - 1} [n_j]$. Thus, all $\phibf^{(m_d)}$ satisfy 
\eqref{eq:alternating-sum-zero-condition-vector} for $d - 1$.
By the induction hypothesis, it follows that each $\phibf^{(m_d)}$ is a 
constant vector. Lastly, taking $S = \{d\}$ in 
\eqref{eq:alternating-sum-zero-condition-vector} gives
\begin{equation*}
    \Delta^{\{d\}}_{m_d}\phibf
    = \phi(0, \dots, 0, m_d) - \phi(0, \dots, 0, m_d - 1) = 0
    \quad \text{for } m_d \in [n_d].
\end{equation*}
Thus, the constants in $\phibf^{(m_d)}$ are the same for all $m_d$,
which means that $\phibf$ is a constant vector.
\end{proof}

\subsubsection{Proof of \eqref{eq:vinfxgb-1d-formula}}
\label{pf:vinfxgb-1d-formula}

We use the following standard result from real analysis in the proof.
\begin{theorem}[Theorem 3.29 of \citet{folland1999real}]
\label{thm:folland-representation}
    Suppose $f: \R \to \R$ has finite total variation and is 
    right-continuous. Then, there exists a unique constant $c \in \R$ 
    and a unique finite signed Borel measure $\lambda$ on $\R$ such that
    \begin{equation}\label{eq:folland-representation}
        f(x) = c + \int \ind(x \ge t) \, d\lambda(t) 
        \quad \text{ for } x \in \R.
    \end{equation}
    Conversely, if $f: \R \to \R$ is of the form 
    \eqref{eq:folland-representation}, then $f$ has finite total 
    variation, is right-continuous, and 
    $\tv(f) = \|\lambda\|_{\text{TV}}$.
\end{theorem}

\begin{proof}[Proof of \eqref{eq:vinfxgb-1d-formula}]
We first show that 
\begin{equation*}
    \mathcal{F}^{1, 1}_{\infty-\text{ST}}  
    = \big\{f: \tv(f) < \infty \text{ and } 
    f \text{ is right-continuous}\big\} 
\end{equation*}
and that 
\begin{equation*}
    V^{1, 1}_{\infty-\text{XGB}}(f) = \tv(f)
\end{equation*}
for $f \in \mathcal{F}^{1, 1}_{\infty-\text{ST}}$.

Suppose 
$f = f^{1, 1}_{c, \{\nu_{L, U}\}} \in \mathcal{F}^{1, 1}_{\infty-\text{ST}}$.
Since $d = s = 1$, the only admissible pairs of $(L, U)$ with 
$0 < |L| + |U| \le s$ are $(\{1\}, \emptyset)$ and 
$(\emptyset, \{1\})$.
Thus, $f$ can be expressed as
\begin{equation*}
    f(x) = c + \int \ind(x \ge l) \, d\nu_{\{1\}, \emptyset}(l)
    + \int \ind(x < u) \, d\nu_{\emptyset, \{1\}}(u)
\end{equation*}
for $x \in \R$. Define a signed Borel measure $\lambda$ by 
$\lambda = \nu_{\{1\}, \emptyset} - \nu_{\emptyset, \{1\}}$. 
Then,
\begin{equation*}
    f(x) = c + \nu_{\emptyset, \{1\}}(\R)
    + \int \ind(x \ge l) \, d\lambda(l)
\end{equation*}
for $x \in \R$, and  
\begin{equation*}
    \|\lambda\|_{\text{TV}} = |\lambda|(\R)
    \le |\nu_{\{1\}, \emptyset}|(\R) 
    + |\nu_{\emptyset, \{1\}}|(\R) 
    = \|\nu_{\{1\}, \emptyset}\|_{\text{TV}} 
    + \|\nu_{\emptyset, \{1\}}\|_{\text{TV}}. 
\end{equation*}
Hence, every $f \in \mathcal{F}^{1, 1}_{\infty-\text{ST}}$ admits 
the simpler representation
\begin{equation*}
    f_{c, \lambda} (x) :=  c + \int \ind(x \ge l) \, d\lambda(l),
\end{equation*}
and its complexity $V^{1, 1}_{\infty-\text{XGB}}(f)$ can be computed by
\begin{equation*}
    V^{1, 1}_{\infty-\text{XGB}}(f) 
    = \inf\{\|\lambda\|_{\text{TV}}: f_{c, \lambda} \equiv f \}.
\end{equation*}
By Theorem \ref{thm:folland-representation}, the collection of such
functions $f_{c, \lambda}$ is precisely the collection of all
right-continuous functions with finite total variation. Moreover, 
the pair $(c, \lambda)$ with $f_{c, \lambda} \equiv f$ is unique and 
satisfies $\|\lambda\|_{\text{TV}} = \tv(f)$. Consequently, 
\begin{equation*}
    \mathcal{F}^{1, 1}_{\infty-\text{ST}}  
    = \big\{f: \tv(f) < \infty \text{ and } 
    f \text{ is right-continuous}\big\},
\end{equation*}
and for every $f \in \mathcal{F}^{1, 1}_{\infty-\text{ST}}$, 
\begin{equation*}
    V^{1, 1}_{\infty-\text{XGB}}(f) = \tv(f).
\end{equation*}

We next prove that 
\begin{equation*}
    V^{1, 2}_{\infty-\text{XGB}}(f) 
    = \frac{1}{2} \cdot \big(\tv(f) + |\Delta(f)|\big)
\end{equation*}
for all $f \in \mathcal{F}^{1, 2}_{\infty-\text{ST}} 
(= \mathcal{F}^{1, 1}_{\infty-\text{ST}})$.
By the same argument as above, we can show that every 
$f \in \mathcal{F}^{1, 2}_{\infty-\text{ST}}$ admits the representation 
\begin{equation*}
    f_{c, \lambda, \mu} (x) :=  c + \int \ind(x \ge l) \, d\lambda(l) 
    + \int \ind(l \le x < u) \, d\mu(l, u),
\end{equation*}
where $\lambda$ and $\mu$ are finite signed Borel measures on $\R$ and 
$\R^2$, respectively, and its complexity $V^{1, 2}_{\infty-\text{XGB}}(f)$ 
is given by 
\begin{equation*}
    V^{1, 2}_{\infty-\text{XGB}}(f) 
    = \inf\big\{\|\lambda\|_{\text{TV}} + \|\mu\|_{\text{TV}}: 
    f_{c, \lambda, \mu} \equiv f \big\}.
\end{equation*}

First, suppose 
$f \equiv f_{c, \lambda, \mu} \in \mathcal{F}^{1, 2}_{\infty-\text{ST}}$.
For every $x < y$, we have
\begin{equation*}
\begin{aligned}
    |f(x) - f(y)| 
    &= \big|-\lambda((x, y]) + \mu\big((-\infty, x] \times (x, y]\big) 
    - \mu\big((x, y] \times (y, +\infty)\big)\big| \\
    &\le |\lambda|((x, y]) + |\mu|\big(\R \times (x, y]\big) 
    + |\mu|\big((x, y] \times \R\big),
\end{aligned}
\end{equation*}
and it thus follows that 
\begin{equation*}
    \tv(f) \le \|\lambda\|_{\text{TV}} + 2\|\mu\|_{\text{TV}}.    
\end{equation*}
Moreover, we have
\begin{equation*}
    |\Delta(f)| 
    = \big|\lim_{x \to +\infty} f(x) - \lim_{x \to -\infty} f(x)\big|
    = |\lambda(\R)| \le \|\lambda\|_{\text{TV}}.
\end{equation*}
Combining these two inequalities, we obtain
\begin{equation*}
    \frac{1}{2} \cdot \big(\tv(f) + |\Delta(f)|\big) 
    \le \|\lambda\|_{\text{TV}} + \|\mu\|_{\text{TV}}.
\end{equation*}
Taking the infimum over all $\lambda$ and $\mu$ with
$f_{c, \lambda, \mu} \equiv f$ gives
\begin{equation*}
    \frac{1}{2} \cdot \big(\tv(f) + |\Delta(f)|\big)
    \le V^{1, 2}_{\infty-\text{XGB}}(f),
\end{equation*}
which proves one direction of the desired identity.

We now prove the reverse inequality. 
Suppose $f \in \mathcal{F}^{1, 2}_{\infty-\text{ST}} 
(= \mathcal{F}^{1, 1}_{\infty-\text{ST}})$. Since $f$ has finite total 
variation and is right-continuous, 
Theorem \ref{thm:folland-representation} guarantees the existence of a 
constant $c \in \R$ and a finite signed Borel measure $\lambda$ on 
$\R$ such that 
\begin{equation*}
    f(x) = c + \int \ind(x \ge l) \, d\lambda(l)
    \quad \text{for } x \in \R
\end{equation*} 
and $\|\lambda\|_{\text{TV}} = \tv(f)$. 
Let $\lambda = \lambda^+ - \lambda^-$ be the Jordan decomposition of 
$\lambda$, and let $(P, N)$ be a Hahn decomposition. 
Without loss of generality, we assume 
\begin{equation*}
    \lambda^+(\R) \ge \lambda^-(\R).
\end{equation*}
Then, 
\begin{equation*}
    |\Delta(f)| 
    = \big|\lim_{x \to +\infty} f(x) - \lim_{x \to -\infty} f(x)\big|
    = |\lambda(\R)|
    = \lambda^+(\R) - \lambda^-(\R),
\end{equation*}
and thus, 
\begin{equation*}
    \frac{1}{2} \cdot \big(\tv(f) + |\Delta(f)|\big) 
    = \frac{1}{2} \cdot 
    \Big(\lambda^+(\R) + \lambda^-(\R) + \lambda^+(\R) - \lambda^-(\R)\Big) 
    = \lambda^+(\R).
\end{equation*}

Define a Borel measure $\widetilde{\lambda}$ on $\R$ by 
\begin{equation*}
    \widetilde{\lambda}(E) 
    = \Big(1 - \frac{\lambda^-(\R)}{\lambda^+(\R)}\Big) \cdot \lambda^+(E)
\end{equation*}
for Borel sets $E \subseteq \R$. The assumption 
$\lambda^+(\R) \ge \lambda^-(\R)$ ensures that $\widetilde{\lambda}(E)$ 
is nonnegative for all $E$. 
Next, define a signed Borel measure $\mu$ on $\R^2$ by 
\begin{equation*}
    \mu(E_1 \times E_2) 
    = \frac{1}{\lambda^+(\R)} \cdot 
    \Big(\lambda^+(E_1) \cdot \lambda^-(E_2) 
    - \lambda^-(E_1) \cdot \lambda^+(E_2)\Big) 
\end{equation*}
for Borel sets $E_1, E_2 \subseteq \R$.
By construction,
\begin{equation*}
    \mu(E_1 \times E_2) = 
    \begin{cases}
    \lambda^+(E_1) \cdot \lambda^-(E_2) / \lambda^+(\R) \ge 0
    & \text{if } E_1 \subseteq P, E_2 \subseteq N, \\
    -\lambda^-(E_1) \cdot \lambda^+(E_2) / \lambda^+(\R) \le 0 
    & \text{if } E_1 \subseteq N, E_2 \subseteq P, \\
    0 & \text{otherwise},
    \end{cases}
\end{equation*}
and therefore, 
\begin{equation*}
    \|\mu\|_{\text{TV}} = \mu(P \times N) - \mu(N \times P) 
    = \frac{2\lambda^+(P) \cdot \lambda^-(N)}{\lambda^+(\R)} 
    = 2 \lambda^-(\R).
\end{equation*}
Define a signed Borel measure $\widetilde{\mu}$ on $\R^2$ by 
\begin{equation*}
    d\widetilde{\mu}(l, u) = 1(l \le u) \cdot d\mu(l, u).
\end{equation*}
Since $\mu$ is anti-symmetric, i.e., 
$\mu(E_1 \times E_2) = -\mu(E_2 \times E_1)$ for all Borel sets 
$E_1, E_2 \subseteq \R$, we have
\begin{equation*}
    \|\widetilde{\mu}\|_{\text{TV}} 
    = \frac{1}{2} \cdot \|\mu\|_{\text{TV}} = \lambda^-(\R).
\end{equation*}
Hence,
\begin{equation*}
    \|\widetilde{\lambda}\|_{\text{TV}} 
    + \|\widetilde{\mu}\|_{\text{TV}}
    = \widetilde{\lambda}(\R) + \|\widetilde{\mu}\|_{\text{TV}}
    = \lambda^+(\R) - \lambda^-(\R) + \lambda^-(\R)
    = \lambda^+(\R).
\end{equation*}

Now, observe that 
\begin{equation*}
    \int \ind(x \ge l) \, d\widetilde{\lambda}(l)
    = \int \ind(x \ge l) \, d\lambda^+(l) 
    - \frac{\lambda^-(\R)}{\lambda^+(\R)} \cdot \lambda^+((-\infty, x])
\end{equation*}
and that
\begin{equation*}
\begin{aligned}
    &\int \ind(l \le x < u) \, d\widetilde{\mu}(l, u) 
    = \int \ind(l \le x < u) \, d\mu(l, u) 
    = \mu\big((-\infty, x] \times (x, +\infty)\big) \\
    &\quad= \frac{1}{\lambda^+(\R)} \cdot 
    \Big(\lambda^+((-\infty, x]) \cdot \lambda^-((x, +\infty)) 
    - \lambda^-((-\infty, x]) \cdot \lambda^+((x, +\infty))\Big) \\
    &\quad= \frac{\lambda^-(\R)}{\lambda^+(\R)} \cdot \lambda^+((-\infty, x])
    - \lambda^-((-\infty, x]) 
    = \frac{\lambda^-(\R)}{\lambda^+(\R)} \cdot \lambda^+((-\infty, x])
    - \int \ind(x \ge l) \, d\lambda^-(l)
\end{aligned}
\end{equation*}
for every $x \in \R$.
Combining these two equations gives 
\begin{equation*}
    f_{c, \widetilde{\lambda}, \widetilde{\mu}}(x) 
    = c + \int \ind(x \ge l) \, d\widetilde{\lambda}(l)
    + \int \ind(l \le x < u) \, d\widetilde{\mu}(l, u)
    = c + \int \ind(x \ge l) \, d\lambda(l)
    = f(x)
\end{equation*}
for every $x \in \R$. As a result,
\begin{equation*}
    V^{1, 2}_{\infty-\text{XGB}}(f) 
    \le \|\widetilde{\lambda}\|_{\text{TV}} 
    + \|\widetilde{\mu}\|_{\text{TV}} 
    = \lambda^+(\R)
    = \frac{1}{2} \cdot \big(\tv(f) + |\Delta(f)|\big),
\end{equation*}
which proves the reverse inequality.
\end{proof}

\subsubsection{Proof of Proposition \ref{prop:vinfxgb-hka-rel}}
\label{pf:vinfxgb-hka-rel}

In the proof of Proposition \ref{prop:vinfxgb-hka-rel}, we use the 
following theorem, which connects functions on a compact domain with 
finite Hardy–Krause variation to the cumulative distribution functions 
of finite signed Borel measures on the same domain. This result will 
also play a central role in the proofs of 
Propositions \ref{prop:fstinf-charac} and \ref{prop:infimal-convolution}. 

To state the result, we first 
recall the definition of Hardy–Krause variation on compact domains.
Let $f: \prod_{j = 1}^{m} [u_j, v_j] \to \R$ and 
$\abf = (a_1, \dots, a_m) \in \prod_{j = 1}^{m} \{u_j, v_j\}$. 
For each $S \subseteq [m]$, define
\begin{equation*}
    f^S_{(a_j, j \in S^c)}(x_j, j \in S) 
    = f(x_j, j \in S; a_j, j \in S^c)
    \quad \text{for } (x_j, j \in S) \in \R^{|S|}.
\end{equation*}
Since the domain is compact, there is no need to take limits as 
in \eqref{eq:f_section}.
The Hardy–Krause variation of $f$ anchored at $\abf$ on
$\prod_{j = 1}^{m} [u_j, v_j]$ is then defined by
\begin{equation*}
    \hk_{\abf}\Big(f; \prod_{j = 1}^{m} [u_j, v_j]\Big)
    = \sum_{0 < |S| \le m} \vit\Big(f^S_{(a_j, j \in S^c)}; 
    \prod_{j \in S} [u_j, v_j]\Big).
\end{equation*}

\begin{theorem}[Theorem 3 of \citet{aistleitner2015functions}]
\label{thm:aistleitner}
Suppose $f: \prod_{j = 1}^{m} [u_j, v_j] \to \R$ is  
right-continuous and has finite Hardy–Krause variation anchored at 
$\abf = (u_1, \dots, u_m)$. Then, there exists a unique finite 
signed Borel measure $\nu$ on $\prod_{j = 1}^{m} [u_j, v_j]$ such that
\begin{equation}\label{eq:hk-representation}
    f(x_1, \dots, x_m) = \nu\Big(\prod_{j = 1}^{m} [u_j, x_j]\Big) 
    \quad \text{for } (x_1, \dots, x_m) \in \prod_{j = 1}^{m} [u_j, v_j].
\end{equation}
Conversely, if $f$ is of the form \eqref{eq:hk-representation}, then $f$ has 
finite Hardy–Krause variation anchored at $\abf$ and 
\begin{equation*}
    \|\nu\|_{\text{TV}} 
    = \hk_{\abf}\Big(f; \prod_{j = 1}^{m} [u_j, v_j]\Big) 
    + |f(\abf)|.
\end{equation*}
\end{theorem}

\begin{proof}[Proof of Proposition \ref{prop:vinfxgb-hka-rel}]
To prove \eqref{eq:vinfxgb-hka-rel}, it suffices to show that 
\begin{equation*}
    \hk_{\abf}(f) / \min(2^s - 1, 2^d) \le \vinfxgb(f) 
    \le \hk_{\abf}(f)
\end{equation*}
for each anchor point $\abf \in \{-\infty, +\infty\}^d$.
Here, we prove this inequality only for the case 
$\abf = (-\infty, \dots, -\infty)$. The same argument applies to 
the other anchor point choices.

Recall that $\fstinf$ is the collection of all functions 
$\fcnu$ of the form \eqref{f_fst}. 
For each $\fcnu \in \fstinf$, by modifying 
each basis function $b^{L, U}_{\lbf, \ubf}$ as 
\begin{align}\label{eq:modification-of-indicator-function}
    &b^{L, U}_{\lbf, \ubf}(x_1, \dots, x_d) 
    = \prod_{j \in L} \ind(x_j \ge l_j) 
    \cdot \prod_{j \in U} \ind(x_j < u_j) \\
    &\quad= \prod_{j \in L \setminus U} \ind(x_j \ge l_j) 
    \cdot \prod_{j \in U \setminus L} \big(1 - \ind(x_j \ge u_j)\big)
    \cdot \prod_{j \in L \cap U} 
    \big(\ind(x_j \ge l_j) - \ind(x_j \ge u_j)\big)
    \cdot \prod_{j \in L \cap U} \ind(l_j \le u_j), \nonumber
\end{align}
we can represent $\fcnu$ as 
\begin{equation}\label{eq:alternative-representation-of-fstinf}
    \fcnu(x_1, \dots, x_d) 
    = f^{d, s}_{b, \{\mu_S\}}(x_1, \dots, x_d) 
    := b + \sum_{0 < |S| \le s} 
    \int_{\R^{|S|}} \prod_{j \in S} \ind(x_j \ge t_j)
    \, d\mu_S(t_j, j \in S) 
\end{equation}
for some $b \in \R$ and finite signed Borel measures $\mu_S$ 
on $\R^{|S|}$, where the summation runs over all nonempty 
$S \subseteq [d]$ with $|S| \le s$.
Specifically, each $\mu_S$ is related to 
the original measures $\nu_{L, U}$ by
\begin{align}\label{eq:relation-between-mu-and-nu}
    &\mu_S\Big(\prod_{j \in S} E_j\Big) 
    = \sum_{\substack{L, U: L \subseteq S \subseteq L \cup U \\
    |L| + |U| \le s}}
    (-1)^{|(U \setminus L) \cap S|}
    \sum_{K \subseteq L \cap U} (-1)^{|(L \cap U) \setminus K|} \\
    &\qquad \qquad \qquad \qquad \qquad \qquad
    \cdot \widebar{\nu}_{L, U}
    \Big(\prod_{j \in (L \setminus U) \cup K} E_j 
    \times \R^{|(L \cap U) \setminus K|} 
    \times \R^{|(U \setminus S) \cup K|} 
    \times \prod_{\substack{j \in ((L \cap U) \setminus K) \\
    \quad \cup ((U \setminus L) \cap S)}} E_j\Big) \nonumber 
\end{align}
for Borel sets $E_j \subseteq \R$ for $j \in S$, 
where $\widebar{\nu}_{L, U}$ are the signed Borel measures on 
$\R^{|L| + |U|}$ defined by
\begin{equation*}
    d\widebar{\nu}_{L, U}(\lbf, \ubf) 
    = \prod_{j \in L \cap U} \ind(l_j \le u_j) 
    \cdot d\nu_{L, U}(\lbf, \ubf).
\end{equation*}
This relationship between $\mu_S$ and $\nu_{L, U}$ implies
\begin{align}\label{eq:inequality-between-complexities}
    &\sum_{0 < |S| \le s} |\mu_S|(\R^{|S|}) 
    \le \sum_{0 < |S| \le s} 
    \sum_{\substack{L, U: L \subseteq S \subseteq L \cup U \\
    |L| + |U| \le s}}
    \sum_{K \subseteq L \cap U} 
    |\widebar{\nu}_{L, U}|(\R^{|L| + |U|}) \nonumber \\
    &\qquad= \sum_{0 < |L| + |U| \le s}
    \sum_{S: L \subseteq S \subseteq L \cup U, S \neq \emptyset}
    \sum_{K \subseteq L \cap U} 
    |\widebar{\nu}_{L, U}|(\R^{|L| + |U|}) \nonumber \\
    &\qquad= \sum_{0 < |L| + |U| \le s}
    \big(\ind(L \neq \emptyset) \cdot 2^{|U \setminus L|} 
    + \ind(L = \emptyset) \cdot (2^{|U \setminus L|} - 1)\big) 
    \cdot 2^{|L \cap U|} \cdot
    |\widebar{\nu}_{L, U}|(\R^{|L| + |U|}) \nonumber \\
    &\qquad\le \min(2^s - 1, 2^d) \cdot
    \sum_{0 < |L| + |U| \le s}
    |\nu_{L, U}|(\R^{|L| + |U|}).
\end{align}

Define 
\begin{equation}\label{eq:va-def}
    V_{\abf}(f) = \inf \Big\{ 
    \sum_{0 < |S| \le s} \|\mu_S\|_{\text{TV}}: 
    f^{d, s}_{b, \{\mu_S\}} \equiv f \Big\} 
    \quad \text{for } f \in \fstinf.
\end{equation}
By \eqref{eq:inequality-between-complexities}, we have
\begin{equation*}
    \vinfxgb(f) \le V_{\abf}(f) 
    \le \min(2^s - 1, 2^d) \cdot \vinfxgb(f)
    \quad \text{for every } f \in \fstinf.
\end{equation*}
Thus, it suffices to prove that  
\begin{equation*}
    V_{\abf}(f) = \hk_{\abf}(f)
    \quad \text{for every } f \in \fstinf.
\end{equation*}

Fix $f \equiv f^{d, s}_{b, \{\mu_S\}} \in \fstinf$.
For each nonempty $S \subseteq [d]$, we have
\begin{equation*}
    f^S_{(a_j, j \in S^c)}(x_j, j \in S) 
    = b + 
    \sum_{\substack{R: \emptyset \neq R \subseteq S \\ |R| \le s}} 
    \int_{\R^{|R|}} \prod_{j \in R} \ind(x_j \ge t_j) 
    \, d\mu_R(t_j, j \in R).
\end{equation*}
Hence, for each nonempty $T \subseteq [d]$,
\begin{equation*}
    \sum_{S \subseteq T} (-1)^{|T| - |S|} \cdot 
    f^S_{(a_j, j \in S^c)}(x_j, j \in S) 
    = \sum_{\substack{R: \emptyset \neq R \subseteq T \\ |R| \le s}} 
    \Big(\sum_{S: R \subseteq S \subseteq T} (-1)^{|T| - |S|}\Big) 
    \int_{\R^{|R|}} \prod_{j \in R} \ind(x_j \ge t_j) 
    \, d\mu_R(t_j, j \in R).
\end{equation*}
The inner sum vanishes unless $R = T$, in which case it equals $1$.
Therefore, if $|T| \le s$,
\begin{equation*}
    \sum_{S \subseteq T} (-1)^{|T| - |S|} \cdot 
    f^S_{(a_j, j \in S^c)}(x_j, j \in S)
    = \int_{\R^{|T|}} \prod_{j \in T} \ind(x_j \ge t_j) 
    \, d\mu_T(t_j, j \in T),
\end{equation*}
while if $|T| > s$, the expression vanishes.

Now, fix a nonempty $T \subseteq [d]$ with $|T| \le s$.
Since 
\begin{equation*}
    \vit(f^T_{(a_j, j \in T^c)}) 
    = \vit\Big((x_j, j \in T) \mapsto \sum_{S \subseteq T} (-1)^{|T| - |S|} 
    \cdot f^S_{(a_j, j \in S^c)}(x_j, j \in S)\Big),
\end{equation*}
we have
\begin{equation*}
\begin{aligned}
    \vit(f^T_{(a_j, j \in T^c)}) 
    &= \vit\Big((x_j, j \in T) \mapsto \int_{\R^{|T|}} 
    \prod_{j \in T} \ind(x_j \ge t_j) \, d\mu_T(t_j, j \in T)\Big) \\
    &= \sup_{u_j < v_j, j \in T} 
    \vit\Big((x_j, j \in T) \mapsto \int_{\R^{|T|}} 
    \prod_{j \in T} \ind(x_j \ge t_j) 
    \, d\mu_T(t_j, j \in T); 
    \prod_{j \in T} [u_j, v_j] \Big) \\
    &= \sup_{u_j < v_j, j \in T} 
    \vit\bigg((x_j, j \in T) \mapsto 
    \mu_T\Big(\prod_{j \in T} (u_j, x_j]\Big); 
    \prod_{j \in T} [u_j, v_j] \bigg).
\end{aligned}
\end{equation*}
Moreover, by Theorem \ref{thm:aistleitner}, 
\begin{equation*}
\begin{aligned}
    &\vit\bigg((x_j, j \in T) \mapsto 
    \mu_T\Big(\prod_{j \in T} (u_j, x_j]\Big); 
    \prod_{j \in T} [u_j, v_j] \bigg) \\
    &\quad= \hk_{(u_j, j \in T)}
    \bigg((x_j, j \in T) \mapsto 
    \mu_T\Big(\prod_{j \in T} (u_j, x_j]\Big); 
    \prod_{j \in T} [u_j, v_j] \bigg) 
    = |\mu_T|\Big(\prod_{j \in T} (u_j, v_j]\Big).
\end{aligned}
\end{equation*}
Here, the first equality holds because the map vanishes on every section 
containing the anchor point 
$(u_j, j \in T)$; that is, 
it becomes zero whenever $x_j = u_j$ for some $j \in T$. 
Consequently,
\begin{equation*}
    \vit(f^T_{(a_j, j \in T^c)}) = |\mu_T|(\R^{|T|}) 
    = \|\mu_T\|_{\text{TV}}.
\end{equation*}

Since 
\begin{equation*}
    \vit(f^T_{(a_j, j \in T^c)}) = 0
\end{equation*}
for all $|T| > s$, it follows that
\begin{equation*}
    \hk_{\abf}(f) 
    = \sum_{0 < |T| \le d} \vit(f^T_{(a_j, j \in T^c)})
    = \sum_{0 < |T| \le s} \vit(f^T_{(a_j, j \in T^c)})
    = \sum_{0 < |T| \le s} \|\mu_T\|_{\text{TV}}.
\end{equation*}
Thus, there is in fact no need to take the infimum in 
\eqref{eq:va-def}, and 
\begin{equation*}
    \hk_{\abf}(f) = V_{\abf}(f),
\end{equation*}
which completes the proof of \eqref{eq:vinfxgb-hka-rel}.

We now investigate the tightness of the inequalities in 
\eqref{eq:vinfxgb-hka-rel}. Fix $\abf = (-\infty, \dots, -\infty)$.
First, observe that for the function 
$f(x_1, \dots, x_d) = \ind(x_1 \ge 0)$, we have
\begin{equation*}
    \vinfxgb(f) = 1 = \hk_{\abf}(f).
\end{equation*}
This shows that the right inequality in \eqref{eq:vinfxgb-hka-rel} 
is tight. 

To show that the left inequality in \eqref{eq:vinfxgb-hka-rel} is 
also tight, we consider two cases, depending on whether $s \le d$ 
or $s > d$. In the case $s \le d$, consider the function 
$f(x_1, \dots, x_d) = \ind(x_1, \dots, x_s < 0)$. It is clear that 
$\vinfxgb(f) = 1$. Moreover, since
\begin{equation*}
    f(x_1, \dots, x_d) = \prod_{j = 1}^{s} \big(1 - \ind(x_j \ge 0)\big)
    = 1 + \sum_{l = 1}^{s} (-1)^{l}
    \sum_{1 \le j_1 < \cdots < j_l \le s} 
    \ind(x_{j_1} \ge 0, \dots, x_{j_l} \ge 0),
\end{equation*}
it follows that
\begin{equation*}
    \hk_{\abf}(f) = 2^s - 1.
\end{equation*}
This shows that the left inequality in \eqref{eq:vinfxgb-hka-rel} 
is tight when $s \le d$. In the case $s > d$, consider the function 
$f(x_1, \dots, x_d) = \ind(-1 \le x_1, \dots, x_{s - d} < 0, \ 
x_{s - d + 1}, \dots, x_d < 0)$. Again, it is clear that 
$\vinfxgb(f) = 1$. Furthermore, we can write
\begin{equation*}
    f(x_1, \dots, x_d) 
    = \prod_{j = 1}^{s - d} \big(\ind(x_j \ge -1) - \ind(x_j \ge 0)\big)
    \cdot \prod_{j = s - d + 1}^{d} \big(1 - \ind(x_j \ge 0)\big),
\end{equation*}
from which we obtain
\begin{equation*}
    \hk_{\abf}(f) = 2^d.
\end{equation*}
This shows that the left inequality in \eqref{eq:vinfxgb-hka-rel} 
is tight when $s > d$.
\end{proof}

\subsubsection{Proof of Proposition \ref{prop:fstinf-stabilization}}
\label{pf:fstinf-stabilization}
\begin{proof}[Proof of Proposition \ref{prop:fstinf-stabilization}]
First, (a) follows from the right-continuity of each basis function 
$b^{L, U}_{\lbf, \ubf}$ and the dominated 
convergence theorem, together with the fact that each signed measure 
$\nu_{L, U}$ is finite.

For (b) and (c), recall the alternative representation 
$f^{d, s}_{b, \{\mu_S\}}$ of $\fcnu$, given in 
\eqref{eq:alternative-representation-of-fstinf} and 
introduced in the proof of Proposition \ref{prop:vinfxgb-hka-rel}. 
Since the sum in this representation ranges over all 
$S \subseteq [d]$ with $0 < |S| \le s$, the function 
class $\fstinf$ enlarges as $s$ increases. This yields (b). Since 
$|S|$ is always bounded by $d$, the class $\fstinf$ remains unchanged 
once $s \ge d$, which proves (c).

Lastly, (d) follows immediately from the definition.
\end{proof}

\subsubsection{Proof of Proposition \ref{prop:fstinf-charac}}
\label{pf:fstinf-charac}

\begin{proof}[Proof of Proposition \ref{prop:fstinf-charac}] 
\textbf{Step 1: $f \in \fstinfd$ $\Rightarrow$ $f$ is right-continuous, 
and $\hk_{\abf}(f) < \infty$ for all $\abf \in \{-\infty, +\infty\}^d$.} 

This follows directly from Propositions \ref{prop:fstinf-stabilization} 
and \ref{prop:vinfxgb-hka-rel}.

\textbf{Step 2: $f \in \fstinf$ $\Rightarrow$ 
\eqref{eq:interaction-restriction-cond} holds for all 
$S \subseteq [d]$ with $|S| > s$.}

We only consider the case $\abf = (-\infty, \cdots, -\infty)$; 
the argument for other choices of anchor points is entirely analogous.
Suppose that $f \in \fstinf$. Recall from the proof of 
Proposition \ref{prop:vinfxgb-hka-rel} that $f$ admits the 
alternative representation $f \equiv f^{d, s}_{b, \{\mu_S\}}$ of the 
form \eqref{eq:alternative-representation-of-fstinf} for some 
$b \in \R$ and finite signed Borel measures $\mu_S$ on $\R^{|S|}$. 

Fix $T \subseteq [d]$ with $|T| > s$. Using this 
representation, we can express $f^T_{(a_j, j \in T^c)}$ as
\begin{equation*}
    f^T_{(a_j, j \in T^c)}(x_j, j \in T) 
    = b + 
    \sum_{\substack{S: \emptyset \neq S \subseteq T \\ |S| \le s}} 
    \int_{\R^{|S|}} \prod_{j \in S} \ind(x_j \ge t_j) 
    \, d\mu_S(t_j, j \in S).
\end{equation*}
It follows that
\begin{equation*}
\begin{aligned}
    &\sum_{\deltabf \in \{0, 1\}^{|T|}} (-1)^{\sum_{j \in T} 
    \delta_j} \cdot f^T_{(a_j, j \in T^c)}\big((1 - \delta_j) 
    w_j + \delta_j v_j, j \in T\big) \\
    &\quad= 
    \sum_{\substack{S: \emptyset \neq S \subseteq T \\ |S| \le s}}
    \sum_{\deltabf_S \in \{0, 1\}^{|S|}}
    \bigg(\sum_{\deltabf_{T \setminus S} \in \{0, 1\}^{|T \setminus S|}} 
    (-1)^{\sum_{j \in T \setminus S} \delta_j} \bigg) \cdot 
    (-1)^{\sum_{j \in S} \delta_j}
    \int_{\R^{|S|}} \prod_{j \in S} \ind(x_j \ge t_j) 
    \, d\mu_S(t_j, j \in S),
\end{aligned}
\end{equation*}
where $\deltabf_S = (\delta_j, j \in S)$ and 
$\deltabf_{T \setminus S} = (\delta_j, j \in T \setminus S)$. 
In the last expression, since $|S| \le s < |T|$, the innermost 
sum always vanishes. This proves that 
\eqref{eq:interaction-restriction-cond} holds for $T$.

\textbf{Step 3: $f$ is right-continuous, and 
$\hk_{\abf}(f) < \infty$ for some $\abf \in \{-\infty, +\infty\}^d$ 
$\Rightarrow$ $f \in \fstinfd$.}

Assume that $f: \R^d \to \R$ is right-continuous and that 
$\hk_{\abf}(f) < \infty$ for some $\abf \in \{-\infty, +\infty\}^d$. 
Fix a nonempty $S \subseteq [d]$, and for each integer $N \ge 1$, 
define $g^S_N: \R^{|S|} \to \R$ by
\begin{equation*}
    g^S_N(x_j, j \in S) = \sum_{\deltabf \in \{0, 1\}^{|S|}}
    (-1)^{\sum_{j \in S} \delta_j} \cdot 
    f^S_{(a_j, j \in S^c)}\big((1 - \delta_j) x_j + \delta_j (-N), 
    j \in S\big).
\end{equation*}
Clearly, $g^S_N$ inherits the right-continuity of $f$ on the 
coordinates $j \in S$. Moreover,
\begin{equation*}
    \hk_{(-N, j \in S)}
    \big(g^S_N; [-N, N]^{|S|}\big) 
    = \vit\big(g^S_N; [-N, N]^{|S|}\big) 
    = \vit\big(f^S_{(a_j, j \in S^c)}; [-N, N]^{|S|}\big) < \infty.
\end{equation*}
Here, the first equality follows from the fact that $g^S_N$ vanishes
whenever $x_j = -N$ for some $j \in S$. Hence, by 
Theorem \ref{thm:aistleitner}, there exists a unique finite signed 
Borel measure $\nu^S_N$ on $[-N, N]^{|S|}$ such that
\begin{equation*}
    g^S_N(x_j, j \in S) = \nu^S_N\Big(\prod_{j \in S} (-N, x_j]\Big)
    \quad \text{for } (x_j, j \in S) \in [-N, N]^{|S|}.
\end{equation*}
Here, the endpoint $-N$ could be excluded from the intervals because 
$g^S_N$ becomes zero if $x_j = -N$ for some $j \in S$.
Furthermore, Theorem \ref{thm:aistleitner} also gives
\begin{equation*}
    |\nu^S_N|([-N, N]^{|S|}) 
    =  \vit\big(f^S_{(a_j, j \in S^c)}; [-N, N]^{|S|}\big)
    \le \vit\big(f^S_{(a_j, j \in S^c)}\big) < \infty.
\end{equation*}

Now, fix integers $N_2 > N_1 \ge 1$, and observe that
\begin{equation*}
    g^S_{N_1}(x_j, j \in S) 
    = \sum_{\deltabf \in \{0, 1\}^{|S|}}
    (-1)^{\sum_{j \in S} \delta_j} \cdot 
    g^S_{N_2}\big((1 - \delta_j) x_j + \delta_j (-N_1), j \in S\big)
    = \nu^S_{N_2}\Big(\prod_{j \in S} (-N_1, x_j]\Big)
\end{equation*}
for every $(x_j, j \in S) \in [-N_1, N_1]^{|S|}$.
By uniqueness of $\nu^S_{N_1}$, this means that the restriction of 
$\nu^S_{N_2}$ to $[-N_1, N_1]^{|S|}$ coincides with $\nu^S_{N_1}$. 
Hence, $\{\nu^S_{N}\}_{N \ge 1}$ forms a sequence of finite signed Borel 
measures, where $\nu^S_{N_2}$ is an extension of $\nu^S_{N_1}$ whenever 
$N_2 > N_1$.

Using this sequence of signed measures, we define a finite signed Borel 
measure $\nu_S$ on $\R^{|S|}$ extending $\nu^S_{N}$ for all $N \ge 1$.
Specifically, we define $\nu_S$ by 
\begin{equation*}
    \nu_S(E) = \lim_{N \to \infty} \nu^S_{N}(E \cap [-N, N]^{|S|})
    \quad \text{for Borel sets } E \subseteq \R^{|S|}.
\end{equation*}
We first verify that $\nu_S(E)$ is well-defined for each Borel set $E$.
For integers $M > N \ge 1$, 
\begin{equation*}
\begin{aligned}
    \big|\nu^S_{M}(E \cap [-M, M]^{|S|}) 
    - \nu^S_{N}(E \cap [-N, N]^{|S|})\big| 
    &= \big|\nu^S_{M}\big(E \cap 
    ([-M, M]^{|S|} \setminus [-N, N]^{|S|})\big)\big| \\
    &\le |\nu^S_{M}|\big([-M, M]^{|S|} \setminus [-N, N]^{|S|}\big).
\end{aligned}
\end{equation*}
Since 
\begin{equation*}
    \sup_{N \ge 1} |\nu^S_{N}|([-N, N]^{|S|}) 
    \le \vit\big(f^S_{(a_j, j \in S^c)}\big) < \infty,
\end{equation*}
for every $\epsilon > 0$, there exists an integer $N_0 \ge 1$ such that
\begin{equation*}
    |\nu^S_{M}|\big([-M, M]^{|S|} \setminus [-N, N]^{|S|}\big) < \epsilon
    \quad \text{for all } M > N \ge N_0.
\end{equation*}
Thus, $\{\nu^S_{N}(E \cap [-N, N]^{|S|})\}_{N \ge 1}$ 
forms a Cauchy sequence, and hence, $\nu_S(E)$ is well-defined.

Next, we show that $\nu_S$ is countably additive. Suppose 
$E = \cup_{k \ge 1} E_k$ for disjoint Borel sets $E_k$.
For integers $N_2 > N_1 \ge 1$, since $\nu^S_{N_2}$ extends $\nu^S_{N_1}$,
the restriction of $|\nu^S_{N_2}|$ (the variation of $\nu^S_{N_2}$) to 
$[-N_1, N_1]^{|S|}$ also coincides with $|\nu^S_{N_1}|$. Thus, for each 
$k \ge 1$, 
\begin{equation*}
    \big\{|\nu^S_{N}|(E_k \cap [-N, N]^{|S|})\big\}_{N \ge 1}
\end{equation*}
is an increasing sequence of nonnegative numbers. By the monotone 
convergence theorem, we have
\begin{equation*}
\begin{aligned}
    &\sum_{k \ge 1} \lim_{N \to \infty} 
    |\nu^S_{N}|(E_k \cap [-N, N]^{|S|}) 
    = \lim_{N \to \infty} \sum_{k \ge 1} 
    |\nu^S_{N}|(E_k \cap [-N, N]^{|S|}) \\
    &\qquad= \lim_{N \to \infty} |\nu^S_{N}|(E \cap [-N, N]^{|S|})
    \le \sup_{N \ge 1} |\nu^S_{N}|([-N, N]^{|S|}) 
    \le \vit\big(f^S_{(a_j, j \in S^c)}\big) < \infty.
\end{aligned}
\end{equation*}
Moreover, since 
\begin{equation*}
    \big|\nu^S_{N}\big(E_k \cap [-N, N]^{|S|}\big)\big|
    \le |\nu^S_{N}|\big(E_k \cap [-N, N]^{|S|}\big)
    \le \lim_{N \to \infty} |\nu^S_{N}|\big(E_k \cap [-N, N]^{|S|}\big)
\end{equation*}
for each $k$ and $N$, the dominated convergence theorem yields
\begin{equation*}
\begin{aligned}
    \nu_S(E) &= \lim_{N \to \infty} \nu^S_{N}(E \cap [-N, N]^{|S|}) 
    = \lim_{N \to \infty} \sum_{k \ge 1} 
    \nu^S_{N}(E_k \cap [-N, N]^{|S|}) \\
    &= \sum_{k \ge 1} \lim_{N \to \infty} 
    \nu^S_{N}(E_k \cap [-N, N]^{|S|}) 
    = \sum_{k \ge 1} \nu_S(E_k).
\end{aligned}
\end{equation*}
This establishes that $\nu_S$ is countably additive.

For each $N \ge 1$, it is clear from the definition of $\nu_S$ that for any 
Borel set $E \subseteq [-N, N]^{|S|}$, we have $\nu_S(E) = \nu^S_{N}(E)$.
Furthermore,
\begin{equation*}
    |\nu_S|(\R^{|S|}) = \lim_{N \to \infty} |\nu_S|([-N, N]^{|S|}) 
    = \lim_{N \to \infty} |\nu^S_{N}|([-N, N]^{|S|}) 
    \le \vit\big(f^S_{(a_j, j \in S^c)}\big) < \infty.
\end{equation*}
Hence, $\nu_S$ is a finite signed Borel measure on $\R^{|S|}$ 
extending $\nu^S_{N}$ for all $N \ge 1$, as desired.

Define $N_{\abf} = \{j \in [d]: a_j = -\infty\}$. For each integer 
$N \ge 1$, we have
\begin{equation*}
\begin{aligned}
    &\sum_{\deltabf \in \{0, 1\}^{|S|}}
    (-1)^{\sum_{j \in S} \delta_j} \cdot 
    f^S_{(a_j, j \in S^c)}\big((1 - \delta_j) x_j + \delta_j (-N), 
    j \in S \cap N_{\abf};
    (1 - \delta_j) x_j + \delta_j N, j \in S \setminus N_{\abf} \big) \\
    &\quad= \sum_{\deltabf \in \{0, 1\}^{|S|}}
    (-1)^{\sum_{j \in S} \delta_j} \cdot 
    g^S_N\big((1 - \delta_j) x_j + \delta_j (-N), j \in S \cap N_{\abf};
    (1 - \delta_j) x_j + \delta_j N, j \in S \setminus N_{\abf} \big) \\
    &\quad= (-1)^{|S \setminus N_{\abf}|} \cdot
    \nu_S\Big(\prod_{j \in S \cap N_{\abf}} (-N, x_j] 
    \times \prod_{j \in S \setminus N_{\abf}} (x_j, N]\Big).
\end{aligned}
\end{equation*}
Taking the limit as $N \to \infty$ yields
\begin{equation*}
    \sum_{R \subseteq S} (-1)^{|S| - |R|} \cdot 
    f^R_{(a_j, j \in R^c)}(x_j, j \in R) 
    = \widetilde{\nu}_S
    \Big(\prod_{j \in S \cap N_{\abf}} (-\infty, x_j] 
    \times \prod_{j \in S \setminus N_{\abf}} (x_j, +\infty)\Big) 
\end{equation*}
where $\widetilde{\nu}_S$ is the signed Borel measure on $\R^{|S|}$ defined 
by 
$\widetilde{\nu}_S = (-1)^{|S \setminus N_{\abf}|} \cdot \nu_S$.
Moreover, since
\begin{equation*}
    f(x_1, \dots, x_d) 
    = \lim_{\zbf \to \abf} f(\zbf)
    + \sum_{S: \emptyset \neq S \subseteq [d]} \sum_{R \subseteq S} 
    (-1)^{|S| - |R|} \cdot f^R_{(a_j, j \in R^c)}(x_j, j \in R)
    \quad \text{for } (x_1, \dots, x_d) \in \R^d,
\end{equation*}
it follows that
\begin{equation*}
\begin{aligned}
    f(x_1, \dots, x_d) 
    &= \lim_{\zbf \to \abf} f(\zbf) 
    + \sum_{S: \emptyset \neq S \subseteq [d]} \widetilde{\nu}_S
    \Big(\prod_{j \in S \cap N_{\abf}} (-\infty, x_j] 
    \times \prod_{j \in S \setminus N_{\abf}} (x_j, +\infty)\Big) \\
    &= \lim_{\zbf \to \abf} f(\zbf) 
    + \sum_{S: \emptyset \neq S \subseteq [d]} \int_{\R^{|S|}} 
    \Big(\prod_{j \in S \cap N_{\abf}} \ind(x_j \ge t_j)\Big)
    \cdot \Big(\prod_{j \in S \setminus N_{\abf}} \ind(x_j < t_j)\Big)
    \, d\widetilde{\nu}_S(t_j, j \in S)
\end{aligned}
\end{equation*}
for all $(x_1, \dots, x_d) \in \R^d$.
This proves that $f \in \fstinfd$.

\textbf{Step 4: $f \in \fstinfd$ and \eqref{eq:interaction-restriction-cond} 
holds for all $S \subseteq [d]$ with $|S| > s$ $\Rightarrow$ $f \in \fstinf$.}

Now, we assume that $f$ additionally satisfies condition 
\eqref{eq:interaction-restriction-cond} for all $S \subseteq [d]$ with 
$|S| > s$. Since the additional condition is vacuous when $s \ge d$, we 
assume that $s < d$. Here, we present the argument only for the case 
$\abf = (-\infty, \cdots, -\infty)$, but the proof for other anchor 
points is entirely analogous.

Since $f \in \fstinfd$, $f$ admits the alternative representation 
\eqref{eq:alternative-representation-of-fstinf} for some $b \in \R$ and 
finite signed Borel measures $\mu_S$ on $\R^{|S|}$.
For each $S \subseteq [d]$ with $|S| > s$, we have
\begin{equation*}
    \sum_{\deltabf \in \{0, 1\}^{|S|}} (-1)^{\sum_{j \in S} \delta_j} 
    \cdot f^S_{(a_j, j \in S^c)}\big((1 - \delta_j) v_j 
    + \delta_j u_j, j \in S\big) 
    = \mu_S\Big(\prod_{j \in S} (u_j, v_j]\Big)
\end{equation*}
for all $u_j < v_j, j \in S$. Therefore, condition 
\eqref{eq:interaction-restriction-cond} implies that for all such $S$ 
and all $u_j < v_j, j \in S$,
\begin{equation*}
    \mu_S\Big(\prod_{j \in S} (u_j, v_j]\Big) = 0.
\end{equation*}
By Dynkin's $\pi$-$\lambda$ theorem, this forces $\mu_S = 0$.
Hence, all integrals over $\mu_S$ with $|S| > s$ can be 
dropped from \eqref{eq:alternative-representation-of-fstinf}, and we can 
conclude that $f \in \fstinf$.
\end{proof}

\subsubsection{Proof of Proposition \ref{prop:vinfxgb-symmetry}}
\label{pf:vinfxgb-symmetry}

\begin{proof}[Proof of Proposition \ref{prop:vinfxgb-symmetry}]
Fix $j_0 \in [d]$ and $t_{j_0} \in \R$. Define $g: \R^d \to \R$ 
as in the statement of the proposition with $j = j_0$.
By symmetry, it suffices to show that $\vinfxgb(g) \le \vinfxgb(f)$.

Suppose that $f \equiv \fcnu$. For each $L,U \subseteq [d]$ with 
$0 < |L| + |U| \le s$, we define a signed Borel measure $\mu_{L, U}$ 
on $\R^{|L| + |U|}$ as follows. If $j_0 \notin L \cup U$,
set $\mu_{L, U} = \nu_{L, U}$. If $j_0 \in L \setminus U$, define 
$\mu_{L, U}$ as the pushforward of 
$\nu_{L \setminus \{j_0\}, U \cup \{j_0\}}$ under the map
\begin{equation*}
    \big((l_j, j \in L\setminus\{j_0\}), 
    (u_j, j \in U \cup \{j_0\})\big) 
    \mapsto 
    \big((l_j, j \in L\setminus\{j_0\}; t_{j_0} - u_{j_0}), 
    (u_j, j \in U)\big),
\end{equation*}
if $j_0 \in U \setminus L$, define $\mu_{L, U}$ as the 
pushforward of $\nu_{L \cup \{j_0\}, U \setminus \{j_0\}}$ under 
the map
\begin{equation*}
    \big((l_j, j \in L \cup \{j_0\}), 
    (u_j, j \in U \setminus \{j_0\})\big)
    \mapsto 
    \big((l_j, j \in L), 
    (u_j, j \in U \setminus \{j_0\}; t_{j_0} - l_{j_0})\big),
\end{equation*}
and if $j_0 \in L \cap U$, define $\mu_{L, U}$ as the pushforward 
of $\nu_{L, U}$ under the map
\begin{equation*}
    \big((l_j, j \in L), (u_j, j \in U)\big)
    \mapsto 
    \big((l_j, j \in L \setminus \{j_0\}; t_{j_0} - u_{j_0}), 
    (u_j, j \in U \setminus \{j_0\}; t_{j_0} - l_{j_0})\big).   
\end{equation*}

With these definitions, one readily checks that
$f^{d, s}_{c, \{\mu_{L, U}\}} \equiv g$. 
Moreover, there is a one-to-one correspondence between 
the signed measures $\mu_{L, U}$ and the signed measures 
$\nu_{L, U}$, under which the corresponding signed measures have 
the same total variation.
Therefore,  
\begingroup
\allowdisplaybreaks
\begin{align*}
    &\vinfxgb(g) \le \sum_{0 < |L| + |U| \le s} 
    \|\mu_{L, U}\|_{\text{TV}} \\
    &\qquad= 
    \sum_{\substack{0 < |L| + |U| \le s \\ j_0 \not\in L \cup U}} 
    \|\mu_{L, U}\|_{\text{TV}}
    + \sum_{\substack{0 < |L| + |U| \le s \\ j_0 \in L \setminus U}} 
    \|\mu_{L, U}\|_{\text{TV}} 
    + \sum_{\substack{0 < |L| + |U| \le s \\ j_0 \in U \setminus L}} 
    \|\mu_{L, U}\|_{\text{TV}}
    + \sum_{\substack{0 < |L| + |U| \le s \\ j_0 \in L \cap U}} 
    \|\mu_{L, U}\|_{\text{TV}} \\
    &\qquad= \sum_{\substack{0 < |L| + |U| \le s 
    \\ j_0 \not\in L \cup U}} \|\nu_{L, U}\|_{\text{TV}} 
    + \sum_{\substack{0 < |L| + |U| \le s 
    \\ j_0 \in L \setminus U}} 
    \|\nu_{L \setminus \{j_0\}, U \cup \{j_0\}}\|_{\text{TV}} \\
    &\qquad \qquad \quad 
    + \sum_{\substack{0 < |L| + |U| \le s 
    \\ j_0 \in U \setminus L}}
    \|\nu_{L \cup \{j_0\}, U \setminus \{j_0\}}\|_{\text{TV}}
    + \sum_{\substack{0 < |L| + |U| \le s 
    \\ j_0 \in L \cap U}} \|\nu_{L, U}\|_{\text{TV}}  
    = \sum_{0 < |L| + |U| \le s} \|\nu_{L, U}\|_{\text{TV}}.
\end{align*}
\endgroup
Taking the infimum over all representations $\fcnu$ of $f$ 
yields $\vinfxgb(g) \le \vinfxgb(f)$.
\end{proof}

\subsection{Proofs of Theorem and Lemma in Section \ref{sec:optimization}}
\label{pf:optimization}

\subsubsection{Proof of Theorem \ref{thm:xgb-opti-lse-mid}}
\label{pf:xgb-opti-lse-mid}

Since the latter part of the theorem is a direct consequence of 
Lemma \ref{lem:discretization}, we only prove the former part 
concerning existence here. The proof of 
Theorem \ref{thm:xgb-opti-lse} is entirely analogous.

\begin{proof}[Proof of Theorem \ref{thm:xgb-opti-lse-mid}]
When the signed measures $\nu_{L, U}$ satisfy condition (a) of
Lemma \ref{lem:discretization}, the corresponding function 
$\fcnu$ can be written as
\begin{equation*}
    \fcnu(x_1, \dots, x_d) = c + \sum_{(L, U, \pbf, \qbf) \in J} 
    \beta^{L, U}_{\pbf, \qbf} \cdot \prod_{j \in L}
    \ind\big(x_j \ge (v^{(j)}_{p_j} + v^{(j)}_{p_j + 1})/2 \big)
    \cdot \prod_{j \in U} 
    \ind\big(x_j < (v^{(j)}_{q_j} + v^{(j)}_{q_j + 1})/2 \big)
\end{equation*}
where
\begin{equation}\label{eq:index-set}
    J = \bigg\{ (L, U, \pbf, \qbf):  
    L, U \subseteq [d], 0 < |L| + |U| \le s,
    \pbf \in \prod_{j \in L} [n_j - 1], \text{ and } 
    \qbf \in \prod_{j \in U} [n_j - 1] \bigg\}
\end{equation}
and 
\begin{equation*}
    \beta^{L, U}_{\pbf, \qbf} 
    =  \nu_{L, U} \Big(\big\{\big((v^{(j)}_{p_j} + v^{(j)}_{p_j + 1})/2, 
    j \in L; 
    (v^{(j)}_{q_j} + v^{(j)}_{q_j + 1})/2, j \in U \big)\big\}\Big)
\end{equation*}
for each $(L, U, \pbf, \qbf) \in J$, with 
$\pbf = (p_j, j \in L)$ and $\qbf = (q_j, j \in U)$.

Let $(\hat{c}, (\hat{\beta}^{L, U}_{\pbf, \qbf}, 
(L, U, \pbf, \qbf) \in J))$ be a solution to the 
finite-dimensional optimization problem
\begin{equation*}
\begin{aligned}
    &\text{argmin} \sum_{i = 1}^{n} \bigg(y_i - c 
    - \sum_{(L, U, \pbf, \qbf) \in J} 
    \beta^{L, U}_{\pbf, \qbf} \cdot \prod_{j \in L}
    \ind\big(x^{(i)}_j \ge (v^{(j)}_{p_j} + v^{(j)}_{p_j + 1})/2 \big) 
    \cdot \prod_{j \in U} \ind\big(x^{(i)}_j 
    < (v^{(j)}_{q_j} + v^{(j)}_{q_j + 1})/2 \big)\bigg)^2 \\
    &\qquad \text{s.t.} 
    \sum_{(L, U, \pbf, \qbf) \in J} 
    |\beta^{L, U}_{\pbf, \qbf}| \le V.
\end{aligned}
\end{equation*}
The existence of such a solution is immediate.
Define $\lsefstinf: \R^d \to \R$ by 
\begin{equation*}
    \lsefstinf(x_1, \dots, x_d) = \hat{c} 
    + \sum_{(L, U, \pbf, \qbf) \in J} 
    \hat{\beta}^{L, U}_{\pbf, \qbf} \cdot \prod_{j \in L}
    \ind\big(x_j \ge (v^{(j)}_{p_j} + v^{(j)}_{p_j + 1})/2 \big)
    \cdot \prod_{j \in U} 
    \ind\big(x_j < (v^{(j)}_{q_j} + v^{(j)}_{q_j + 1})/2 \big).
\end{equation*}
By construction, $\lsefstinf \in \fstmid$, and it is a solution to 
the problem \eqref{xgb_opti_mid}. Moreover, 
Lemma \ref{lem:discretization} implies that it is also a solution to 
\eqref{lse_fstinf}.
\end{proof}

\subsubsection{Proof of Lemma \ref{lem:discretization}}
\label{pf:discretization}
We use the following lemma for the proof. 
This lemma is proved right after the proof of 
Lemma~\ref{lem:discretization}.
\begin{lemma}\label{lem:reduction-to-discrete-measures}
    For every $\fcnu \in \fstinf$, there exists 
    $f^{d, s}_{b, \{\mu_{L, U}\}} \in \fst$ with discrete signed 
    measures $\mu_{L, U}$ supported on the lattices \eqref{lattice} 
    such that 
    \begin{enumerate}[label = (\alph*)]
    \item $f^{d, s}_{b, \{\mu_{L, U}\}}(\xbf^{(i)})
    = \fcnu(\xbf^{(i)})$ for $i = 1, \dots, n$
    \item 
    \begin{equation*}
        \sum_{0 < |L| + |U| \le s} \|\mu_{L, U}\|_{\text{TV}} 
        \le \sum_{0 < |L| + |U| \le s} \|\nu_{L, U}\|_{\text{TV}}.
    \end{equation*}
    \end{enumerate}
\end{lemma}

\begin{proof}[Proof of Lemma \ref{lem:discretization}]
For $z_1, \dots, z_n \in \R$, define 
\begin{equation*}
    \vinfxgb(z_1, \dots, z_n) = \inf \bigg\{ 
    \sum_{0 < |L| + |U| \le s} \|\nu_{L, U}\|_{\text{TV}}: 
    \fcnu(\xbf^{(i)}) 
    = z_i \text{ for } i = 1, \dots, n \bigg\}.
\end{equation*}
For simplicity, we suppress the dependence on the design points 
$\xbf^{(1)}, \dots, \xbf^{(n)}$.
By definition,
\begin{equation*}
    \vinfxgb(f(\xbf^{(1)}), \dots, f(\xbf^{(n)})) 
    \le \vinfxgb(f) \quad \text{for every } f \in \fstinf.
\end{equation*}

By Lemma \ref{lem:reduction-to-discrete-measures}, for each 
$z_1, \dots, z_n$, we have
\begin{equation*}
\begin{aligned}
    \vinfxgb(z_1, \dots, z_n) &= \inf \bigg\{ 
    \sum_{0 < |L| + |U| \le s} \|\nu_{L, U}\|_{\text{TV}}: 
    \fcnu(\xbf^{(i)}) 
    = z_i \text{ for } i = 1, \dots, n, \\
    &\qquad \qquad \qquad \qquad \qquad \qquad \qquad 
    \text{ with } \nu_{L, U} 
    \text{ supported on the lattices \eqref{lattice}}\bigg\}.
\end{aligned}
\end{equation*}
Recall the index set $J$ from \eqref{eq:index-set}, and let $\xbf$ 
denote the $n \times |J|$ matrix with entries
\begin{equation*}
    X_{i, (L, U, \pbf, \qbf)} 
    = \prod_{j \in L} 
    \ind\big(x^{(i)}_j \ge (v^{(j)}_{p_j} + v^{(j)}_{p_j + 1})/2\big)
    \cdot \prod_{j \in U} 
    \ind\big(x^{(i)}_j < (v^{(j)}_{q_j} + v^{(j)}_{q_j + 1})/2\big)
\end{equation*}
for $i = 1, \dots, n$ and $(L, U, \pbf, \qbf) \in J$.
We parametrize discrete signed measures $\nu_{L, U}$ 
supported on the lattices \eqref{lattice} by 
\begin{equation}\label{eq:measures-discrete-parametrization}
    \nu_{L, U} 
    \big(\{((v^{(j)}_{p_j} + v^{(j)}_{p_j + 1})/2, j \in L) 
    \times ((v^{(j)}_{q_j} + v^{(j)}_{q_j + 1})/2, j \in U)\}\big)
    = \beta^{L, U}_{\pbf, \qbf}
\end{equation}
for $L, U \subseteq [d]$ with $0 < |L| + |U| \le s$, 
$\pbf = (p_j, j \in L) \in \prod_{j \in L} [n_j - 1]$, 
and $\qbf = (q_j, j \in U) \in \prod_{j \in U} [n_j - 1]$.
With this parametrization, we can express $\vinfxgb(z_1, \dots, z_n)$ as 
\begin{equation}\label{eq:discrete-complexity}
    \vinfxgb(z_1, \dots, z_n) = \inf \big\{\|\betabf\|_1: \xbf \betabf 
    = \zbf - c \onevec \text{ for some } c \in \R\big\},
\end{equation}
where $\zbf = (z_1, \dots, z_n)$, and $\onevec$ is the all-ones 
vector.

Next, we show that if the set 
\begin{equation}\label{eq:discrete-feasible-set}
    \mathcal{D}_{\zbf} := \big\{\betabf \in \R^{|J|}: 
    \xbf \betabf = \zbf - c \onevec 
    \text{ for some } c \in \R\big\}
\end{equation}
is nonempty, which is clearly the case when
$\zbf = (f(\xbf^{(1)}), \dots, f(\xbf^{(n)}))$ 
for some $f \in \fstinf$, then there exists $\hat{\betabf}$ that achieves 
the minimum in \eqref{eq:discrete-complexity}.
Fix $\zbf \in \R^n$ such that $\mathcal{D}_{\zbf}$ is 
nonempty, and suppose $\xbf \betabf_0 = \zbf - c_0 \onevec$ 
for some $\betabf_0$ and $c_0$.
Clearly, the infimum on the right-hand side of 
\eqref{eq:discrete-complexity} remains unchanged if we further 
constrain $\betabf$ to satisfy $\|\betabf\|_1 \le \|\betabf_0\|_1$:
\begin{equation*}
    \vinfxgb(z_1, \dots, z_n) = \inf \big\{\|\betabf\|_1: \xbf \betabf 
    = \zbf - c \onevec \text{ for some } c \in \R 
    \text{ and } \|\betabf\|_1 \le \|\betabf_0\|_1\big\}.
\end{equation*}
It is also straightforward to verify that the set
\begin{equation*}
    \mathcal{D}_{\zbf} \cap \{\betabf \in \R^{|J|}: 
    \|\betabf\|_1 \le \|\betabf_0\|_1\}
\end{equation*}
is nonempty, closed, and bounded.
Since the map $\betabf \mapsto \|\betabf\|_1$ is continuous, it follows 
that there exists $\hat{\betabf}$ that attains the minimum in
\eqref{eq:discrete-complexity}.

Using the results established above, we now prove the lemma. 
Fix $f \in \fstinf$, and let $\hat{\betabf}$ be the minimizer of 
\eqref{eq:discrete-complexity} for $\zbf = (f(\xbf^{(1)}), \dots, 
f(\xbf^{(n)}))$.
Let $\hat{c}$ be the corresponding constant from 
\eqref{eq:discrete-feasible-set}, and let 
$\nu_{L, U}$ denote the signed Borel measures associated with 
$\hat{\betabf}$ via \eqref{eq:measures-discrete-parametrization}.
By construction, $\fcnu \in \fst$ 
satisfies the first two conditions of the lemma. 
Moreover, since
\begin{equation*}
\begin{aligned}
    &\vinfxgb(f(\xbf^{(1)}), \dots, f(\xbf^{(n)})) 
    = \|\hat{\betabf}\|_1 
    = \sum_{0 < |L| + |U| \le s} \|\nu_{L, U}\|_{\text{TV}} \\
    &\qquad \ge \vinfxgb(\fcnu)
    \ge \vinfxgb(f(\xbf^{(1)}), \dots, f(\xbf^{(n)})),
\end{aligned}
\end{equation*} 
we have
\begin{equation*}
    \vinfxgb(\fcnu) 
    = \sum_{0 < |L| + |U| \le s} \|\nu_{L, U}\|_{\text{TV}}
    = \vinfxgb(f(\xbf^{(1)}), \dots, f(\xbf^{(n)})) 
    \le \vinfxgb(f).
\end{equation*}
Hence, $\fcnu \in \fstmid$ is a function that satisfies 
all the desired properties.
\end{proof}

\begin{proof}[Proof of Lemma \ref{lem:reduction-to-discrete-measures}]
For each $j \in [d]$, define
\begin{equation*}
    I^{(j)}_{m_j} = (v^{(j)}_{m_j}, v^{(j)}_{m_j + 1}] 
    \qquad \text{for } m_j \in [n_j - 1],
\end{equation*}
and
\begin{equation*}
    \underline{I}^{(j)}_{m_j} = (v^{(j)}_{1}, v^{(j)}_{m_j}]
    \ \text{ and } \
    \overline{I}^{(j)}_{m_j} = (v^{(j)}_{m_j}, v^{(j)}_{n_j}]
    \qquad \text{for } m_j \in [n_j].
\end{equation*}
Also, let
\begin{equation*}
    \underline{O}^{(j)} = (-\infty, v^{(j)}_{1}]
    \ \text{ and } \
    \overline{O}^{(j)} = (v^{(j)}_{n_j}, +\infty).
\end{equation*}

With these notations, we define discrete signed measures 
$\mu_{L, U}$ and a constant $b$ as follows.
For $L, U \subseteq [d]$ with $0 < |L| + |U| \le s$, let 
$\mu_{L, U}$ be the discrete signed measure 
supported on the lattice \eqref{lattice}, defined by
\begin{equation*}
\begin{aligned}
    &\mu_{L, U} \big(\{(
    (v^{(j)}_{p_j} + v^{(j)}_{p_j + 1})/2, j \in L) 
    \times ((v^{(j)}_{q_j} + v^{(j)}_{q_j + 1})/2, j \in U)\}\big) \\
    &\quad= \sum_{\substack{\widetilde{L}, \widetilde{U}: 
    \widetilde{L} \supseteq L, \widetilde{U} \supseteq U \\
    |\widetilde{L}| + |\widetilde{U}| \le s}}
    \nu_{\widetilde{L}, \widetilde{U}} 
    \Big(\prod_{j \in L} I^{(j)}_{p_j} 
    \times \prod_{j \in \widetilde{L} \setminus L} \underline{O}^{(j)} 
    \times \prod_{j \in U} I^{(j)}_{q_j}
    \times \prod_{j \in \widetilde{U} \setminus U} 
    \overline{O}^{(j)}\Big)
\end{aligned}
\end{equation*}
for $(p_j, j \in L) \in \prod_{j \in L} [n_j - 1]$ and 
$(q_j, j \in U) \in \prod_{j \in U} [n_j - 1]$.
Also, let
\begin{equation*}
    b = c + \sum_{0 < |L| + |U| \le s} \nu_{L, U} 
    \Big(\prod_{j \in L} \underline{O}^{(j)} 
    \times \prod_{j \in U} \overline{O}^{(j)}\Big).
\end{equation*}
By construction, for each $(m_1, \dots, m_d) 
\in \prod_{j = 1}^{d} [n_j]$, we have
\begin{equation*}
\begin{aligned}
    &\int_{\R^{|L| + |U|}} \prod_{j \in L} 
    \ind\big(v^{(j)}_{m_j} \ge l_j\big)  
    \cdot \prod_{j \in U} \ind\big(v^{(j)}_{m_j} < u_j\big)
    \, d\mu_{L, U}(\lbf, \ubf) \\
    &\quad = 
    \sum_{\rbf \in \prod\limits_{j \in L} \{1, \dots, m_j - 1\}
    \times \prod\limits_{j \in U} \{m_j, \dots, n_j - 1\}} 
    \mu_{L, U} \big(\{((v^{(j)}_{r_j} + v^{(j)}_{r_j + 1})/2, j \in L) 
    \times ((v^{(j)}_{r_j} + v^{(j)}_{r_j + 1})/2, j \in U)\}\big) \\
    &\quad= \sum_{\substack{\widetilde{L}, \widetilde{U}: 
    \widetilde{L} \supseteq L, \widetilde{U} \supseteq U \\
    |\widetilde{L}| + |\widetilde{U}| \le s}}
    \nu_{\widetilde{L}, \widetilde{U}} 
    \Big(\prod_{j \in L} \underline{I}^{(j)}_{m_j} 
    \times \prod_{j \in \widetilde{L} \setminus L} \underline{O}^{(j)} 
    \times \prod_{j \in U} \overline{I}^{(j)}_{m_j}
    \times \prod_{j \in \widetilde{U} \setminus U} 
    \overline{O}^{(j)}\Big).
\end{aligned}
\end{equation*}
It follows that for each 
$(m_1, \dots, m_d) \in \prod_{j = 1}^{d} [n_j]$, 
\begin{equation*}
\begin{aligned}
    &f^{d, s}_{b, \{\mu_{L, U}\}}(v^{(1)}_{m_1}, 
    \dots, v^{(d)}_{m_d}) = b + \sum_{0 < |L| + |U| \le s} 
    \sum_{\substack{\widetilde{L}, \widetilde{U}: 
    \widetilde{L} \supseteq L, \widetilde{U} \supseteq U \\
    |\widetilde{L}| + |\widetilde{U}| \le s}}
    \nu_{\widetilde{L}, \widetilde{U}} 
    \Big(\prod_{j \in L} \underline{I}^{(j)}_{m_j} 
    \times \prod_{j \in \widetilde{L} \setminus L} \underline{O}^{(j)} 
    \times \prod_{j \in U} \overline{I}^{(j)}_{m_j}
    \times \prod_{j \in \widetilde{U} \setminus U} 
    \overline{O}^{(j)}\Big) \\
    &\quad = c + \sum_{0 < |\widetilde{L}| + |\widetilde{U}| \le s} 
    \sum_{L \subseteq \widetilde{L}}
    \sum_{U \subseteq \widetilde{U}}
    \nu_{\widetilde{L}, \widetilde{U}} 
    \Big(\prod_{j \in L} \underline{I}^{(j)}_{m_j} 
    \times \prod_{j \in \widetilde{L} \setminus L} \underline{O}^{(j)} 
    \times \prod_{j \in U} \overline{I}^{(j)}_{m_j}
    \times \prod_{j \in \widetilde{U} \setminus U} 
    \overline{O}^{(j)}\Big) \\
    &\quad = c + \sum_{0 < |\widetilde{L}| + |\widetilde{U}| \le s} 
    \nu_{\widetilde{L}, \widetilde{U}} 
    \Big(\prod_{j \in \widetilde{L}} 
    \big(\underline{I}^{(j)}_{m_j} \cup \underline{O}^{(j)}\big)
    \times \prod_{j \in \widetilde{U}} 
    \big(\overline{I}^{(j)}_{m_j} \cup \overline{O}^{(j)}\big)\Big) 
    = \fcnu(v^{(1)}_{m_1}, \dots, v^{(d)}_{m_d}),
\end{aligned}
\end{equation*}
which implies that $f^{d, s}_{b, \{\mu_{L, U}\}}$ agrees with 
$\fcnu$ at all design points 
$\xbf^{(1)}, \dots, \xbf^{(n)}$. Moreover, 
\begin{equation*}
\begin{aligned}
    &\sum_{0 < |L| + |U| \le s} |\mu_{L, U}|(\R^{|L| + |U|}) 
    \le \sum_{0 < |L| + |U| \le s} 
    \sum_{\pbf \in \prod_{j \in L} [n_j - 1]}
    \sum_{\qbf \in \prod_{j \in U} [n_j - 1]}
    \sum_{\substack{\widetilde{L}, \widetilde{U}: 
    \widetilde{L} \supseteq L, \widetilde{U} \supseteq U \\
    |\widetilde{L}| + |\widetilde{U}| \le s}} \\
    &\qquad \qquad \qquad \qquad \qquad \qquad \qquad \qquad \qquad \qquad 
    \quad
    |\nu_{\widetilde{L}, \widetilde{U}}|
    \Big(\prod_{j \in L} I^{(j)}_{p_j} 
    \times \prod_{j \in \widetilde{L} \setminus L} \underline{O}^{(j)} 
    \times \prod_{j \in U} I^{(j)}_{q_j}
    \times \prod_{j \in \widetilde{U} \setminus U} 
    \overline{O}^{(j)}\Big) \\
    &\qquad = \sum_{0 < |L| + |U| \le s} 
    \sum_{\substack{\widetilde{L}, \widetilde{U}: 
    \widetilde{L} \supseteq L, \widetilde{U} \supseteq U \\
    |\widetilde{L}| + |\widetilde{U}| \le s}}  
    |\nu_{\widetilde{L}, \widetilde{U}}|
    \Big(\prod_{j \in L} \underline{I}^{(j)}_{n_j} 
    \times \prod_{j \in \widetilde{L} \setminus L} \underline{O}^{(j)} 
    \times \prod_{j \in U} \overline{I}^{(j)}_{1}
    \times \prod_{j \in \widetilde{U} \setminus U}
    \overline{O}^{(j)}\Big) \\
    &\qquad \le \sum_{0 < |\widetilde{L}| + |\widetilde{U}| \le s} 
    \sum_{L \subseteq \widetilde{L}}
    \sum_{U \subseteq \widetilde{U}}
    |\nu_{\widetilde{L}, \widetilde{U}}|
    \Big(\prod_{j \in L} \underline{I}^{(j)}_{n_j} 
    \times \prod_{j \in \widetilde{L} \setminus L} \underline{O}^{(j)} 
    \times \prod_{j \in U} \overline{I}^{(j)}_{1}
    \times \prod_{j \in \widetilde{U} \setminus U} 
    \overline{O}^{(j)}\Big) \\
    &\qquad = \sum_{0 < |\widetilde{L}| + |\widetilde{U}| \le s} 
    |\nu_{\widetilde{L}, \widetilde{U}}|
    \Big(\prod_{j \in \widetilde{L}} 
    \big(\underline{I}^{(j)}_{n_j} \cup \underline{O}^{(j)}\big)
    \times \prod_{j \in \widetilde{U}} 
    \big(\overline{I}^{(j)}_{1} \cup \overline{O}^{(j)}\big)\Big) 
    \le \sum_{0 < |\widetilde{L}| + |\widetilde{U}| \le s} 
    |\nu_{\widetilde{L}, \widetilde{U}}|(\R^{|\widetilde{L}| 
    + |\widetilde{U}|}).
\end{aligned}
\end{equation*}
This proves that $f^{d, s}_{b, \{\mu_{L, U}\}}$ is the desired function 
satisfying the conditions of the lemma.
\end{proof}

\subsection{Proofs of Theorems, Lemma, and Corollary in Section \ref{sec:minimax-rate}}
\label{pf:minimax-rate}
\subsubsection{Proof of Theorem \ref{thm:risk-upper-bound}}
\label{pf:risk-upper-bound}
We will use the following three results from empirical process theory to prove 
the theorem. Theorem \ref{thm:moment-inequality} provides a moment inequality 
for the expected supremum of multiplier empirical processes. 
Lemma \ref{lem:bracketing-integral-bound} bounds the expected supremum of 
empirical processes with Rademacher multipliers in terms of bracketing entropy 
integrals. Theorem \ref{thm:han2019convergence} reduces the problem of 
controlling the expected supremum of general multiplier empirical processes to 
the case with Rademacher multipliers. While Theorems \ref{thm:moment-inequality} and
\ref{thm:han2019convergence} are general results,
Lemma \ref{lem:bracketing-integral-bound} is more specific to our setting. 
We provide the proof of Lemma \ref{lem:bracketing-integral-bound} in Appendix 
\ref{pf:bracketing-integral-bound}.

\begin{theorem}[Proposition 3.1 of \citet{gine2000exponential}]
\label{thm:moment-inequality}
Suppose $\F$ is a countable collection of functions from $\mathcal{X}$ to $\R$.
Assume that $\xbf^{(1)}, \dots, \xbf^{(n)}$ are i.i.d. with law $P$ 
on $\mathcal{X}$ and that $\xi_1, \dots, \xi_n$ are independent mean-zero random 
variables, independent of $\xbf^{(1)}, \dots, \xbf^{(n)}$. Then, 
there exists a constant $C > 0$ such that
\begin{equation*}
\begin{aligned}
    \E\Big[\sup_{f \in \F} \Big|\sum_{i = 1}^{n} 
    \xi_i f(\xbf^{(i)})\Big|^p \Big] 
    &\le C^p \bigg[ \E \Big[\sup_{f \in \F} \Big|\sum_{i = 1}^{n} 
    \xi_i f(\xbf^{(i)})\Big|\Big]^p
    + p^{p/2} n^{p/2} \Big(\sup_{f \in \F} \|f\|_{P, 2}\Big)^p 
    \cdot \max_i \|\xi_i\|_2^p \\
    &\qquad \quad + p^p \E \Big[\max_i \Big(|\xi_i|^p \cdot 
    \sup_{f \in \F} |f(\xbf^{(i)})|^p \Big)\Big]\bigg]
\end{aligned}
\end{equation*}
for every $p \ge 1$. Here, $\|\cdot\|_{P, 2}$ is defined by
\begin{equation*}
    \|f\|_{P, 2} = \big(\E_{X \sim P} [f^2(X)]\big)^{1/2}.
\end{equation*}
\end{theorem}

\begin{lemma}\label{lem:bracketing-integral-bound}
    Suppose $\xbf^{(1)}, \dots, \xbf^{(k)}$ are i.i.d. random 
    variables on $\R^d$ with density $p_0$, and let 
    $\epsilon_1, \dots, \epsilon_k$ be independent Rademacher random variables, 
    independent of $\xbf^{(1)}, \dots, \xbf^{(k)}$. Let $\F$ be a 
    countable collection of functions from $\R^d$ to $\R$, and suppose there 
    exist $t, D > 0$ such that $\|f\|_{p_0, 2} \le t$ and 
    $\|f\|_{\infty} \le D$ for all $f \in \F$. Then, 
    \begin{equation*}
        \E\Big[\sup_{f \in \F} \Big|\frac{1}{\sqrt{k}} 
        \sum_{i = 1}^{k} \epsilon_i f(\xbf^{(i)})\Big|\Big] 
        \le C J_{[ \ ]}(t, \F, \| \cdot \|_{p_0, 2}) 
        \cdot \Big(1 + D \cdot 
        \frac{J_{[ \ ]}(t, \F, \| \cdot \|_{p_0, 2})}{t^2 \sqrt{k}}\Big),
    \end{equation*}
    where $C$ is a universal constant, and 
    $J_{[ \ ]}(t, \F, \| \cdot \|_{p_0, 2})$ is the bracketing entropy integral defined by
    \begin{equation*}
        J_{[ \ ]}(t, \F, \| \cdot \|_{p_0, 2}) = \int_{0}^{t} 
        \sqrt{1 + \log N_{[ \ ]}(\epsilon, \F, \| \cdot \|_{p_0, 2})} \, d\epsilon,
    \end{equation*}
    with $N_{[ \ ]}(\epsilon, \F, \| \cdot \|_{p_0, 2})$ denoting the 
    $\epsilon$-bracketing number of $\F$ with respect to $\| \cdot \|_{p_0, 2}$.
\end{lemma}

\begin{theorem}[Corollary 1 of \citet{han2019convergence}]
\label{thm:han2019convergence}
    Let $\F_1, \dots, \F_n$ be countable collections of functions from 
    $\mathcal{X}$ to $\R$ such that $\F_k \supseteq \F_n$ for every 
    $1 \le k \le n$. Suppose $\xbf^{(1)}, \dots, \xbf^{(n)}$
    are permutation invariant random variables on $\mathcal{X}$, and let 
    $\xi_1, \dots, \xi_n$ be i.i.d. mean-zero random variables, 
    independent of $\xbf^{(1)}, \dots, \xbf^{(n)}$. Assume that
    there exist $p \ge 1$ and $C > 0$ such that
    \begin{equation*}
        \E\Big[\sup_{f \in \F_k} 
        \Big|\sum_{i = 1}^{k} \epsilon_i f(\xbf^{(i)})\Big| \Big]
        \le C k^{1/p}
    \end{equation*}
    for every $1 \le k \le n$, where $\epsilon_1, \dots, \epsilon_n$ are 
    independent Rademacher random variables, independent of 
    $\xbf^{(1)}, \dots, \xbf^{(n)}$. 
    Then, for every $q \ge 1$, 
    \begin{equation*}
        \E\Big[\sup_{f \in \F_n} 
        \Big|\sum_{i = 1}^{n} \xi_i f(\xbf^{(i)})\Big| \Big]
        \le 4C \|\xi_1\|_{\min(p, q), 1} \cdot k^{1/\min(p, q)},
    \end{equation*}
    where for each $r \ge 1$,
    \begin{equation*}
        \|\xi_1\|_{r, 1} := \int_{0}^{\infty} \P(|\xi_1| > t)^{1/r} \, dt.
    \end{equation*}
\end{theorem}

\begin{remark}\label{rmk:measurability}
    The function classes in the above results are assumed to be countable 
    to ensure measurability of the suprema inside the expectations. For an 
    uncountable function class $\F$ and a stochastic process 
    $(\Phi(f): f \in \F)$ indexed by $\F$, the supremum 
    $\sup_{f \in \F} \Phi(f)$ may not be measurable.
    
    In the proof of Theorem \ref{thm:risk-upper-bound}, to avoid such a
    measurability issue, we define the expected supremum of 
    $\Phi$ over $\F$ as 
    \begin{equation*}
        \E\Big[\sup_{f \in \F} \Phi(f)\Big] 
        := \sup \Big\{\E\Big[\sup_{f \in \G} \Phi(f)\Big]: 
        \G \subseteq \F \text{ is countable} \Big\},
    \end{equation*}
    following \citet{talagrand2022upper}.
    Similarly, for any $c \in \R$, we define 
    \begin{equation*}
        \P\Big(\sup_{f \in \F} \Phi(f) \ge c\Big) 
        := \sup \Big\{\P\Big(\sup_{f \in \G} \Phi(f) > c\Big): 
        \G \subseteq \F \text{ is countable} \Big\}.
    \end{equation*}
    With these definitions, we can avoid measurability concerns, and 
    the above theorems and lemma also extend to uncountable function classes.
\end{remark}

\begin{proof}[Proof of Theorem \ref{thm:risk-upper-bound}]
Let $\F_{\Mbf}(V)$ denote the collection of all functions 
$\fcnu \in \fstinf$ of the form 
\eqref{f_fst} satisfying the following conditions:
\begin{enumerate}[label = (\alph*)]
\item $\nu_{L, U}$ are supported on 
$\prod_{j \in L} (-M_j/2, M_j/2] \times \prod_{j \in U} (-M_j/2, M_j/2]$
\item
\begin{equation*}
    \sum_{L, U: 0 < |L| + |U| \le s} \|\nu_{L, U}\|_{\text{TV}} 
    \le V.
\end{equation*}
\end{enumerate}
It is clear from the definition of $\lsefstinf$ that
\begin{equation*}
    \lsefstinf \in \F_{\Mbf}(V) \subseteq \{f \in \fstinf: \vinfxgb(f) \le V\}.
\end{equation*}
Also, the following lemma, proved in 
Appendix \ref{pf:reduction-to-compact-domain}, guarantees the existence of 
$f_{0, \Mbf} \in \F_{\Mbf}(V)$ such that $f_{0, \Mbf}(\cdot) = f_0(\cdot)$ on 
$\prod_{j = 1}^{d} [-M_j/2, M_j/2]$. 
    
\begin{lemma}\label{lem:reduction-to-compact-domain}
    For every $f \in \fstinf$ with $\vinfxgb(f) < V$, there exists 
    $f_{\Mbf} \in \F_{\Mbf}(V)$ such that $f_{\Mbf}(\cdot) = f(\cdot)$ on 
    $\prod_{j = 1}^{d} [-M_j/2, M_j/2]$.
\end{lemma}
    
For each $t > 0$, define 
\begin{equation*}
    B(V, t) = \{f \in \F_{\Mbf}(V): \|f\|_{p_0, 2} \le t\}.
\end{equation*}
We suppress the dependence of $B(V, t)$ on 
${\Mbf} = (M_1, \dots, M_d)$ for brevity. The following lemma, 
proved in Appendix \ref{pf:BVt-bracketing-integral-bound}, 
provides a bracketing entropy integral bound for $B(V, t)$, 
which will play a crucial role throughout the proof.
\begin{lemma}\label{lem:BVt-bracketing-integral-bound}
    There exists a constant $C_{B, s} > 0$, 
    depending on $B$ and $s$, such that for all $t > 0$,
    \begin{equation*}
        J_{[ \ ]}(t, B(V, t), \| \cdot \|_{p_0, 2}) 
        \le C_{B, s} d^{\widebar{s}}(1 + \log d)^{\widebar{s} - 1} 
        \bigg(t \log\Big(2 + \frac{V}{t}\Big) + V^{1/2}t^{1/2} 
        \Big[\log\Big(2 + \frac{V}{t}\Big)\Big]^{\widebar{s} - 1}\bigg).
    \end{equation*}
\end{lemma}
    
Now, suppose we have $t_n > 4 \|f_0 - f^*\|_{p_0, 2}$ such that for every $r \ge 1$,  
\begin{equation}\label{eq:expected-suprema-bound}
\begin{aligned}
    &\E\Big[ \sup_{f \in B(V, rt_n)} \Big| \frac{1}{\sqrt{n}} 
    \sum_{i = 1}^{n} \epsilon_i f(\xbf^{(i)})\Big|\Big] 
    \le r\sqrt{n}t_n^2 / (V + 1), \\
    &\E\Big[ \sup_{f \in B(V, rt_n)} \Big| \frac{1}{\sqrt{n}} 
    \sum_{i = 1}^{n} \xi_i f(\xbf^{(i)})\Big|\Big] 
    \le r\sqrt{n}t_n^2 / (V + 1), \text{ and } \\
    &\E\Big[ \sup_{f \in B(V, rt_n)} \Big| \frac{1}{\sqrt{n}} 
    \sum_{i = 1}^{n} \epsilon_i f(\xbf^{(i)})\cdot 
    (f_0 - f^*)(\xbf^{(i)})\Big|\Big] 
    \le r\sqrt{n}t_n^2 / (V + 1),
\end{aligned}
\end{equation}
where $\epsilon_i$ are Rademacher random variables 
independent of $\xbf^{(i)}$, and the expectations are taken over 
$\xbf^{(i)}$, $\epsilon_i$, and $\xi_i$.
In what follows, we will first see how these bounds on the expected suprema can 
be used to obtain a risk bound for $\lsefstinf$. The value of $t_n$ satisfying the 
above inequalities will be specified in the next step, after which 
more precise risk bounds will be derived. The subscript $n$ 
emphasizes that $t_n$ depends on $n$, while its dependence on other parameters 
is suppressed for notational simplicity.
    
We first aim to bound $\P(\|\lsefstinf - f_0\|_{p_0, 2} > t)$ for $t \ge t_n$. 
Fix $r \ge 1$, and for each integer $j \ge 2$, define 
\begin{equation*}
    \F_j = \big\{ f \in \F_{\Mbf}(V): 
    2^{j - 2} r t_n < \|f - f_{0, \Mbf}\|_{p_0, 2} < 2^j r t_n \big\}.
\end{equation*}
By construction, 
\begin{equation*}
    \P\big(\|\lsefstinf - f_{0}\|_{p_0, 2} > rt_n \big)
    = \P\big(\|\lsefstinf - f_{0, \Mbf}\|_{p_0, 2} > rt_n \big) 
    \le \sum_{j = 2}^{\infty} \P\big(\lsefstinf \in \F_j\big).
\end{equation*}
Next, let $(M_n(f): f \in \F_{\Mbf}(V))$ denote the stochastic processes defined by
\begin{equation*}
    M_n(f) = \frac{2}{n} \sum_{i = 1}^{n} \xi_i (f - f^*)(\xbf^{(i)}) 
    - \frac{1}{n} \sum_{i = 1}^{n} (f - f^*)^2(\xbf^{(i)})
\end{equation*}
and define $(M(f): f \in \F_{\Mbf}(V))$ by 
\begin{equation*}
    M(f) = - \|f - f^*\|_{p_0, 2}^2.
\end{equation*}
Since 
\begin{equation*}
    M_n(f) = \frac{2}{n} \sum_{i = 1}^{n} \xi_i (f - f^*)(\xbf^{(i)}) 
    - \frac{1}{n} \sum_{i = 1}^{n} (f - f^*)^2(\xbf^{(i)})
    = - \frac{1}{n} \sum_{i = 1}^{n} \big(y_i - f(\xbf^{(i)})\big)^2 
    + \frac{1}{n} \sum_{i = 1}^{n} \xi_i^2, 
\end{equation*}
we have 
\begin{equation*}
    M_n(\lsefstinf) - M_n(f_{0, \Mbf}) \ge 0,
\end{equation*}
as $\lsefstinf$ minimizes the least squares over $\F_{\Mbf}(V)$.
Moreover, 
\begin{equation*}
\begin{aligned}
    M_n(f) - M_n(f_{0, \Mbf}) 
    &= \frac{2}{n} \sum_{i = 1}^{n} \xi_i (f - f_{0, \Mbf})(\xbf^{(i)})
    - \frac{1}{n} \sum_{i = 1}^{n} (f - f_{0, \Mbf})^2(\xbf^{(i)}) \\
    &\quad- \frac{2}{n} \sum_{i = 1}^{n} (f - f_{0, \Mbf})(\xbf^{(i)})
    \cdot (f_{0, \Mbf} - f^*)(\xbf^{(i)})
\end{aligned}
\end{equation*}
and
\begin{equation*}
    M(f) - M(f_{0, \Mbf}) 
    = - \|f - f_{0, \Mbf}\|_{p_0, 2}^2 
    - 2 \E_{\xbf \sim p_0} \big[(f - f_{0, \Mbf})(\xbf) 
    \cdot (f_{0, \Mbf} - f^*)(\xbf)\big].
\end{equation*}
The assumption $t_n > 4 \|f_0 - f^*\|_{p_0, 2} 
= 4 \|f_{0, \Mbf} - f^*\|_{p_0, 2}$ implies that for every $f \in \F_j$,
\begin{equation*}
\begin{aligned}
    -(M(f) - M(f_{0, \Mbf})) 
    &\ge \|f - f_{0, \Mbf}\|_{p_0, 2} \cdot 
    \big(\|f - f_{0, \Mbf}\|_{p_0, 2} - 2 \|f_{0, \Mbf} - f^*\|_{p_0, 2}\big) \\ 
    &\ge 2^{j - 2} r t_n \cdot 
    \Big(2^{j - 2} r t_n - \frac{t_n}{2}\Big) 
    \ge 2^{2j - 5} r^2 t_n^2.
\end{aligned}
\end{equation*}
Therefore,\footnote{Because of our definitions in Remark \ref{rmk:measurability},
introduced to avoid measurability issues, some additional care is required in 
justifying the first inequality. A more detailed argument is provided in 
the remark following the proof.}
\begin{align}\label{eq:peeling-bound}
    &\P\big(\lsefstinf \in \F_j\big) 
    \le
    \P\Big(\sup_{f \in \F_j} \big(M_n(f) - M_n(f_{0, \Mbf})\big) \ge 0 \Big) \nonumber \\
    &\quad\le 
    \P\Big(\sup_{f \in \F_j} \big((M_n(f) - M_n(f_{0, \Mbf})) 
    - (M(f) - M(f_{0, \Mbf}))\big) \ge 2^{2j - 5} r^2 t_n^2 \Big) \nonumber \\
    &\quad\le \P\Big(\sup_{f \in \F_j} \Big|\frac{1}{n} 
    \sum_{i = 1}^{n} \xi_i (f - f_{0, \Mbf})(\xbf^{(i)})\Big| 
    \ge 2^{2j - 7} r^2 t_n^2 \Big) \nonumber \\
    &\quad \qquad + \P\Big(\sup_{f \in \F_j} \Big|\frac{1}{n} 
    \sum_{i = 1}^{n} (f - f_{0, \Mbf})^2(\xbf^{(i)}) 
    -\|f - f_{0, \Mbf}\|_{p_0, 2}^2\Big| \ge 2^{2j - 7} r^2 t_n^2 \Big) \nonumber \\
    &\quad \qquad + \P\Big(\sup_{f \in \F_j} \Big|\frac{1}{n} 
    \sum_{i = 1}^{n} (f - f_{0, \Mbf})(\xbf^{(i)})
    \cdot (f_{0, \Mbf} - f^*)(\xbf^{(i)}) \nonumber \\
    &\quad \qquad \qquad \qquad \quad
    - \E_{\xbf \sim p_0} \big[(f - f_{0, \Mbf})(\xbf) \cdot
    (f_{0, \Mbf} - f^*)(\xbf)\big]\Big|
    \ge 2^{2j - 8} r^2 t_n^2 \Big) \nonumber \\
    &\quad \le \P\Big(\sup_{f \in B(2V, 2^j r t_n)} \Big|\frac{1}{\sqrt{n}} 
    \sum_{i = 1}^{n} \xi_i f(\xbf^{(i)})\Big| 
    \ge 2^{2j - 7} r^2 \sqrt{n}t_n^2 \Big) \nonumber \\
    &\quad \qquad + \P\Big(\sup_{f \in B(2V, 2^j r t_n)} 
    \Big|\frac{1}{\sqrt{n}} \sum_{i = 1}^{n} \big( f^2(\xbf^{(i)}) 
    -\|f\|_{p_0, 2}^2\big) \Big| \ge 2^{2j - 7} r^2 \sqrt{n}t_n^2 \Big) \\
    &\quad \qquad 
    + \P\Big(\sup_{f \in B(2V, 2^j r t_n)} \Big|\frac{1}{\sqrt{n}} 
    \sum_{i = 1}^{n} \Big(f(\xbf^{(i)})
    \cdot (f_{0, \Mbf} - f^*)(\xbf^{(i)}) \nonumber \\
    &\qquad \qquad \qquad \qquad \qquad \qquad \qquad \qquad
    - \E_{\xbf \sim p_0} \big[f(\xbf) \cdot
    (f_{0, \Mbf} - f^*)(\xbf)\big]\Big)\Big|
    \ge 2^{2j - 8} r^2 \sqrt{n} t_n^2 \Big), \nonumber
\end{align}
where the second inequality uses that 
$-(M(f) - M(f_{0, \Mbf})) \ge 2^{2j - 5} r^2 t_n^2$ for all $f \in \F_j$, 
and the last inequality follows because 
$f - f_{0, \Mbf} \in B(2V, 2^j r t_n)$ for all $f \in \F_j$.
    
We next bound each term on the right-hand side of \eqref{eq:peeling-bound}.
As a preliminary step, we show that there exists a constant $C > 0$ such that 
$\|f\|_{\infty} \le C(V + t)$ for every $f \in B(V, t)$. 
Suppose $f \in B(V, t)$ is of the form 
\begin{equation*}
    f(x_1, \dots, x_d) = c + \sum_{0 < |L| + |U| \le s} 
    \int_{\R^{|L| + |U|}} \prod_{j \in L} \ind(x_j \ge l_j) 
    \cdot \prod_{j \in U} \ind(x_j < u_j) 
    \, d\nu_{L, U}(\lbf, \ubf)
\end{equation*}
where 
\begin{equation*}
    \sum_{0 < |L| + |U| \le s} \|\nu_{L, U}\|_{\text{TV}}
    \le V.
\end{equation*}
Since the sum of the total variations of the signed measures is bounded by $V$, 
the second term in the above representation of $f$ is uniformly bounded in 
absolute value by $V$. 
Hence, by Cauchy inequality, 
\begin{equation*}
    t^2 \ge \|f\|_{p_0, 2}^2 = \int_{\prod_{j = 1}^{d} [-M_j/2, M_j/2]} 
    f^2(\xbf) \cdot p_0(\xbf) \, d\xbf 
    \ge \int_{\prod_{j = 1}^{d} [-M_j/2, M_j/2]} \Big(\frac{c^2}{2} - V^2\Big) 
    \cdot p_0(\xbf) \, d\xbf 
    = \frac{c^2}{2} - V^2.
\end{equation*}
It follows that 
\begin{equation*}
    \|f\|_{\infty} \le |c| + V \le C(V + t)
\end{equation*}
for some universal constant $C > 0$.
    
We now bound the first term on the right-hand side of \eqref{eq:peeling-bound}. 
By Markov's inequality, 
\begin{equation*}
    \P\Big(\sup_{f \in B(2V, 2^j r t_n)} \Big|\frac{1}{\sqrt{n}} 
    \sum_{i = 1}^{n} \xi_i f(\xbf^{(i)})\Big| 
    \ge 2^{2j - 7} r^2 \sqrt{n}t_n^2 \Big) 
    \le \frac{1}{2^{6j - 21} r^6 n^{3/2} t_n^6} \cdot
    \E\Big[\sup_{f \in B(2V, 2^j r t_n)} \Big|\frac{1}{\sqrt{n}} 
    \sum_{i = 1}^{n} \xi_i f(\xbf^{(i)})\Big|^3 \Big]. 
\end{equation*}
To bound the expectation on the right, we apply Theorem \ref{thm:moment-inequality},
which gives
\begin{align}\label{eq:moment-expansion}
    \E\Big[\sup_{f \in B(2V, 2^j r t_n)} \Big|\frac{1}{\sqrt{n}} 
    \sum_{i = 1}^{n} \xi_i f(\xbf^{(i)})\Big|^3 \Big] 
    &\le C \cdot \E\Big[\sup_{f \in B(2V, 2^j r t_n)} \Big|\frac{1}{\sqrt{n}} 
    \sum_{i = 1}^{n} \xi_i f(\xbf^{(i)})\Big| \Big]^3
    + C \cdot 2^{3j} r^3 t_n^3 \|\xi_1\|_2^3 \nonumber \\
    &\quad + C n^{-3/2} \cdot \E \Big[\max_i \Big(|\xi_i|^3 \cdot 
    \sup_{f \in B(2V, 2^j r t_n)} |f(\xbf^{(i)})|^3 \Big)\Big].
\end{align}
Using \eqref{eq:expected-suprema-bound}, the inequality
\begin{equation*}
    \max_i |\xi_i|^3 \le \sum_{i = 1}^{n} |\xi_i|^3,
\end{equation*}
and the preliminary result
\begin{equation}\label{eq:preliminary-bound}
    \|f\|_{\infty} \le C(V + 2^j r t_n) 
    \ \text{ for all } f \in B(2V, 2^j r t_n),
\end{equation}
we deduce from \eqref{eq:moment-expansion} that
\begin{equation*}
    \E\Big[\sup_{f \in B(2V, 2^j r t_n)} \Big|\frac{1}{\sqrt{n}} 
    \sum_{i = 1}^{n} \xi_i f(\xbf^{(i)})\Big|^3 \Big]
    \le C \cdot 2^{3j} r^3 n^{3/2} t_n^6 
    + C \cdot 2^{3j} r^3 t_n^3 \|\xi_1\|_2^3 
    + C n^{-1/2} \|\xi_1\|_3^3 (V^3 + 2^{3j} r^3 t_n^3).
\end{equation*}
Substituting this back into the Markov inequality bound, we obtain
\begin{equation*}
    \P\Big(\sup_{f \in B(2V, 2^j r t_n)} \Big|\frac{1}{\sqrt{n}} 
    \sum_{i = 1}^{n} \xi_i f(\xbf^{(i)})\Big| 
    \ge 2^{2j - 7} r^2 \sqrt{n}t_n^2 \Big) 
    \le \frac{C}{2^{3j} r^3} 
    + \frac{C\|\xi_1\|_2^3}{2^{3j} r^3 n^{3/2} t_n^3} 
    + \frac{C\|\xi_1\|_3^3 V^3}{2^{6j} r^6 n^2 t_n^6} 
    + \frac{C\|\xi_1\|_3^3}{2^{3j}r^3 n^2 t_n^3}.
\end{equation*}
    
We next bound the second term on the right-hand side of \eqref{eq:peeling-bound}.
We divide into two cases depending on whether $2^j r t_n \le V$ or not.
First, suppose $2^j r t_n \le V$.
By Markov's inequality, 
\begin{equation}\label{eq:markov-inequality-first-case-second-term}
\begin{aligned}
    &\P\Big(\sup_{f \in B(2V, 2^j r t_n)} 
    \Big|\frac{1}{\sqrt{n}} \sum_{i = 1}^{n} \big(f^2(\xbf^{(i)}) 
    -\|f\|_{p_0, 2}^2\big)\Big| \ge 2^{2j - 7} r^2 \sqrt{n}t_n^2 \Big) \\
    &\qquad \le \frac{1}{2^{6j - 21} r^6 n^{3/2} t_n^6} \cdot 
    \E\Big[\sup_{f \in B(2V, 2^j r t_n)} 
    \Big|\frac{1}{\sqrt{n}} \sum_{i = 1}^{n} \big(f^2(\xbf^{(i)}) 
    -\|f\|_{p_0, 2}^2\big) \Big|^3 \Big].
\end{aligned}
\end{equation}
Also, by the standard argument of symmetrization (see, e.g., 
\citet[Lemma 2.3.1]{vaartwellner96book} and 
\citet[Theorem 16.1]{van2016estimation}),
\begin{equation}\label{eq:symmetrization-first-case-second-term}
    \E\Big[\sup_{f \in B(2V, 2^j r t_n)} 
    \Big|\frac{1}{\sqrt{n}} \sum_{i = 1}^{n} \big(f^2(\xbf^{(i)}) 
    -\|f\|_{p_0, 2}^2\big) \Big|^3 \Big] 
    \le 8 \E\Big[\sup_{f \in B(2V, 2^j r t_n)} 
    \Big|\frac{1}{\sqrt{n}} \sum_{i = 1}^{n} \epsilon_i f^2(\xbf^{(i)}) 
    \Big|^3 \Big], 
\end{equation}
where $\epsilon_i$ are Rademacher random variables independent of 
$\xbf^{(i)}$.
Applying Theorem \ref{thm:moment-inequality}, we can bound the expectation on 
the right-hand side as
\begin{equation*}
\begin{aligned}
    &\E\Big[\sup_{f \in B(2V, 2^j r t_n)} 
    \Big|\frac{1}{\sqrt{n}} \sum_{i = 1}^{n} \epsilon_i f^2(\xbf^{(i)}) 
    \Big|^3 \Big] 
    \le C \cdot \E\Big[\sup_{f \in B(2V, 2^j r t_n)} 
    \Big|\frac{1}{\sqrt{n}} \sum_{i = 1}^{n} \epsilon_i f^2(\xbf^{(i)}) 
    \Big| \Big]^3 \\ 
    &\qquad \qquad \qquad \qquad \qquad \qquad
    + C \Big(\sup_{f \in B(2V, 2^j r t_n)} \|f^2\|_{p_0, 2}\Big)^3
    + Cn^{-3/2} \cdot \E \Big[\max_i \sup_{f \in B(2V, 2^j r t_n)} 
    |f(\xbf^{(i)})|^6\Big].
\end{aligned}
\end{equation*}
The contraction principle (see, e.g., 
\citet[Proposition A.3.2]{vaartwellner96book} and 
\citet[Theorem 4.12]{ledoux1991probability}), together with 
\eqref{eq:expected-suprema-bound} and
\eqref{eq:preliminary-bound}, gives
\begin{equation*}
\begin{aligned}
    \E\Big[\sup_{f \in B(2V, 2^j r t_n)} 
    \Big|\frac{1}{\sqrt{n}} \sum_{i = 1}^{n} \epsilon_i f^2(\xbf^{(i)}) 
    \Big| \Big] 
    &\le C (V + 2^j r t_n) \cdot 
    \E\Big[\sup_{f \in B(2V, 2^j r t_n)} 
    \Big|\frac{1}{\sqrt{n}} \sum_{i = 1}^{n} \epsilon_i f(\xbf^{(i)}) 
    \Big| \Big] \\
    &\le C(V + 2^j r t_n) \cdot 2^j r \sqrt{n} t_n^2 / V.
\end{aligned}
\end{equation*}
Moreover, we have 
\begin{equation*}
    \sup_{f \in B(2V, 2^j r t_n)} \|f^2\|_{p_0, 2} 
    \le \sup_{f \in B(2V, 2^j r t_n)} \|f\|_{\infty} \cdot
    \sup_{f \in B(2V, 2^j r t_n)} \|f\|_{p_0, 2}  
    \le C(V + 2^j r t_n) \cdot 2^j r t_n.
\end{equation*}
Therefore,
\begin{equation*}
\begin{aligned}
    &\E\Big[\sup_{f \in B(2V, 2^j r t_n)} 
    \Big|\frac{1}{\sqrt{n}} \sum_{i = 1}^{n} \epsilon_i f^2(\xbf^{(i)}) 
    \Big|^3 \Big] \\ 
    &\qquad \le 
    C(V + 2^j r t_n)^3 \cdot 2^{3j} r^3 n^{3/2} t_n^6 / V^3
    + C(V + 2^j r t_n)^3 \cdot 2^{3j} r^3 t_n^3 
    + Cn^{-3/2}(V + 2^j r t_n)^6.
\end{aligned}
\end{equation*}
Combining this with \eqref{eq:markov-inequality-first-case-second-term} and 
\eqref{eq:symmetrization-first-case-second-term} yields
\begin{equation*}
\begin{aligned}
    &\P\Big(\sup_{f \in B(2V, 2^j r t_n)} 
    \Big|\frac{1}{\sqrt{n}} \sum_{i = 1}^{n} \big(f^2(\xbf^{(i)}) 
    -\|f\|_{p_0, 2}^2\big) \Big| \ge 2^{2j - 7} r^2 \sqrt{n}t_n^2 \Big) \\
    &\qquad \le \frac{C(V + 2^j r t_n)^3}{2^{3j} r^3 V^3} 
    + \frac{C(V + 2^j r t_n)^3}{2^{3j} r^3 n^{3/2} t_n^3} 
    + \frac{C(V + 2^j r t_n)^6}{2^{6j} r^6 n^3 t_n^6} 
    \le \frac{C}{2^{3j} r^3} 
    + \frac{CV^3}{2^{3j} r^3 n^{3/2} t_n^3} 
    + \frac{CV^6}{2^{6j} r^6 n^3 t_n^6}.
\end{aligned}
\end{equation*}
    
Next, assume that $2^j r t_n > V$. For each $f \in B(2V, 2^j r t_n)$, as seen 
in the preliminary step, we can decompose $f$ as $f = c + g$ where $c$ is a 
constant with $|c| \le C(V + 2^j r t_n)$ and $g$ is a function uniformly 
bounded in absolute value by $2V$. Using this decomposition, we can write
\begin{equation*}
    \frac{1}{\sqrt{n}} \sum_{i = 1}^{n} \big( f^2(\xbf^{(i)}) 
    -\|f\|_{p_0, 2}^2\big) 
    = 2c \cdot \frac{1}{\sqrt{n}} \sum_{i = 1}^{n} \big(g(\xbf^{(i)}) 
    - \E_{X \sim p_0}[g(X)]\big) 
    + \frac{1}{\sqrt{n}} \sum_{i = 1}^{n} \big(g^2(\xbf^{(i)}) 
    -\|g\|_{p_0, 2}^2\big).
\end{equation*}
It follows that
\begin{align}\label{eq:triangle-inequality-second-case}
    &\P\Big(\sup_{f \in B(2V, 2^j r t_n)} 
    \Big|\frac{1}{\sqrt{n}} \sum_{i = 1}^{n} \big(f^2(\xbf^{(i)}) 
    -\|f\|_{p_0, 2}^2\big)\Big| \ge 2^{2j - 7} r^2 \sqrt{n}t_n^2 \Big) 
    \nonumber \\
    \begin{split}
    &\qquad \le 
    \P\Big(\sup_{\substack{g \in F_{\Mbf}(2V) \\ \|g\|_{\infty} \le 2V}} 
    \Big|\frac{1}{\sqrt{n}} \sum_{i = 1}^{n} \big(g(\xbf^{(i)}) 
    - \E_{X \sim p_0}[g(X)]\big)\Big| 
    \ge \frac{2^{2j - 8} r^2 \sqrt{n}t_n^2}{C(V + 2^j r t_n)} \Big) \\
    & \qquad \qquad  
    + \P\Big(\sup_{\substack{g \in F_{\Mbf}(2V) \\ \|g\|_{\infty} \le 2V}} 
    \Big|\frac{1}{\sqrt{n}} \sum_{i = 1}^{n} \big(g^2(\xbf^{(i)}) 
    -\|g\|_{p_0, 2}^2\big)\Big| \ge 2^{2j - 8} r^2 \sqrt{n}t_n^2 \Big).
    \end{split}
\end{align}
By Markov's inequality, the first term of 
\eqref{eq:triangle-inequality-second-case} is bounded as
\begin{equation}\label{eq:markov-inequality-second-case-first-term}
\begin{aligned}
    &\P\Big(\sup_{\substack{g \in F_{\Mbf}(2V) \\ \|g\|_{\infty} \le 2V}} 
    \Big|\frac{1}{\sqrt{n}} \sum_{i = 1}^{n} \big(g(\xbf^{(i)}) 
    - \E_{X \sim p_0}[g(X)]\big)\Big| 
    \ge \frac{2^{2j - 8} r^2 \sqrt{n}t_n^2}{C(V + 2^j r t_n)} \Big) \\
    &\qquad \le \frac{C(V + 2^j r t_n)^3}{2^{6j - 24} r^6 n^{3/2} t_n^6} \cdot 
    \E\Big[\sup_{\substack{g \in F_{\Mbf}(2V) \\ \|g\|_{\infty} \le 2V}} 
    \Big|\frac{1}{\sqrt{n}} \sum_{i = 1}^{n} \big(g(\xbf^{(i)}) 
    - \E_{X \sim p_0}[g(X)]\big)\Big|^3 \Big]. 
\end{aligned}
\end{equation}
Also, by the standard argument of symmetrization,
\begin{equation}\label{eq:symmetrization-second-case-first-term}
    \E\Big[\sup_{\substack{g \in F_{\Mbf}(2V) \\ \|g\|_{\infty} \le 2V}} 
    \Big|\frac{1}{\sqrt{n}} \sum_{i = 1}^{n} \big(g(\xbf^{(i)}) 
    - \E_{X \sim p_0}[g(X)]\big)\Big|^3 \Big] 
    \le 8 \E\Big[\sup_{\substack{g \in F_{\Mbf}(2V) \\ \|g\|_{\infty} \le 2V}} 
    \Big|\frac{1}{\sqrt{n}} \sum_{i = 1}^{n} 
    \epsilon_i g(\xbf^{(i)}) \Big|^3 \Big].
\end{equation}
Using Theorem \ref{thm:moment-inequality}, the expectation on the right-hand 
side can be bounded as
\begin{equation}\label{eq:moment-inequality-second-case-first-term}
    \E\Big[\sup_{\substack{g \in F_{\Mbf}(2V) \\ \|g\|_{\infty} \le 2V}} 
    \Big|\frac{1}{\sqrt{n}} \sum_{i = 1}^{n} 
    \epsilon_i g(\xbf^{(i)}) \Big|^3 \Big] 
    \le C \cdot \E\Big[\sup_{\substack{g \in F_{\Mbf}(2V) \\ \|g\|_{\infty} \le 2V}} 
    \Big|\frac{1}{\sqrt{n}} \sum_{i = 1}^{n} 
    \epsilon_i g(\xbf^{(i)}) \Big|\Big]^3
    + C V^3. 
\end{equation}
Applying Lemma \ref{lem:bracketing-integral-bound} and 
Lemma \ref{lem:BVt-bracketing-integral-bound}, we obtain
\begin{equation*}
\begin{aligned}
    &\E\Big[\sup_{\substack{g \in F_{\Mbf}(2V) \\ \|g\|_{\infty} \le 2V}} 
    \Big|\frac{1}{\sqrt{n}} \sum_{i = 1}^{n} 
    \epsilon_i g(\xbf^{(i)}) \Big|\Big]
    = \E\Big[\sup_{\substack{g \in B(2V, 2V) \\ \|g\|_{\infty} \le 2V}} 
    \Big|\frac{1}{\sqrt{n}} \sum_{i = 1}^{n} 
    \epsilon_i g(\xbf^{(i)}) \Big|\Big] \\
    &\qquad \le C J_{[ \ ]}(2V, B(2V, 2V), \| \cdot \|_{p_0, 2}) 
    \cdot \Big(1 + 2V \cdot 
    \frac{J_{[ \ ]}(2V, B(2V, 2V), \| \cdot \|_{p_0, 2})}{4V^2 \sqrt{n}}\Big)
    \le C_{B, s} a_{d, s}^2 V,
\end{aligned}
\end{equation*}
where $a_{d, s} := d^{\widebar{s}}(1 + \log d)^{\widebar{s} - 1}$.
Combining this with \eqref{eq:markov-inequality-second-case-first-term}, 
\eqref{eq:symmetrization-second-case-first-term}, and
\eqref{eq:moment-inequality-second-case-first-term} yields
\begin{equation*}
    \P\Big(\sup_{\substack{g \in F_{\Mbf}(2V) \\ \|g\|_{\infty} \le 2V}} 
    \Big|\frac{1}{\sqrt{n}} \sum_{i = 1}^{n} \big(g(\xbf^{(i)}) 
    - \E_{X \sim p_0}[g(X)]\big)\Big| 
    \ge \frac{2^{2j - 8} r^2 \sqrt{n}t_n^2}{C(V + 2^j r t_n)} \Big) \le \frac{C_{B, s} 
    a_{d, s}^6 V^3(V + 2^j r t_n)^3}{2^{6j - 24} r^6 n^{3/2} t_n^6}
    \le \frac{C_{B, s} a_{d, s}^6 V^3}{2^{3j} r^3 n^{3/2} t_n^3},
\end{equation*}
where the last inequality follows from the assumption that $2^j r t_n > V$.
By a similar argument, the second term in
\eqref{eq:triangle-inequality-second-case} is bounded by
\begin{equation*}
\begin{aligned}
    &\P\Big(\sup_{\substack{g \in F_{\Mbf}(2V) \\ \|g\|_{\infty} \le 2V}} 
    \Big|\frac{1}{\sqrt{n}} \sum_{i = 1}^{n} \big(g^2(\xbf^{(i)}) 
    -\|g\|_{p_0, 2}^2\big)\Big| \ge 2^{2j - 8} r^2 \sqrt{n}t_n^2 \Big) 
    \le \frac{C_{B, s} a_{d, s}^6 V^6}{2^{6j} r^6 n^{3/2} t_n^6} 
    \le \frac{C_{B, s} a_{d, s}^6 V^3}{2^{3j} r^3 n^{3/2} t_n^3}.
\end{aligned}
\end{equation*}
Substituting these bounds back into \eqref{eq:triangle-inequality-second-case}
gives
\begin{equation*}
    \P\Big(\sup_{f \in B(2V, 2^j r t_n)} 
    \Big|\frac{1}{\sqrt{n}} \sum_{i = 1}^{n} \big(f^2(\xbf^{(i)}) 
    -\|f\|_{p_0, 2}^2\big)\Big| \ge 2^{2j - 7} r^2 \sqrt{n}t_n^2 \Big)
    \le \frac{C_{B, s} a_{d, s}^6 V^3}{2^{3j} r^3 n^{3/2} t_n^3}.
\end{equation*}
Thus, whether or not $2^j r t_n \le V$, we have
\begin{equation*}
    \P\Big(\sup_{f \in B(2V, 2^j r t_n)} 
    \Big|\frac{1}{\sqrt{n}} \sum_{i = 1}^{n} \big(f^2(\xbf^{(i)}) 
    -\|f\|_{p_0, 2}^2\big) \Big| \ge 2^{2j - 7} r^2 \sqrt{n}t_n^2 \Big)
    \le \frac{C}{2^{3j} r^3} 
    + \frac{C_{B, s} a_{d, s}^6 V^3}{2^{3j} r^3 n^{3/2} t_n^3} 
    + \frac{CV^6}{2^{6j} r^6 n^3 t_n^6}.
\end{equation*}
    
The third term on the right-hand side of \eqref{eq:peeling-bound} can 
be bounded similarly to the second term in the case $2^j r t_n \le V$.
Applying Markov's inequality, the symmetrization argument, and 
Theorem \ref{thm:moment-inequality} in turn yields
\begin{equation*}
\begin{aligned}
    &\P\Big(\sup_{f \in B(2V, 2^j r t_n)} \Big|\frac{1}{\sqrt{n}} 
    \sum_{i = 1}^{n} \Big(f(\xbf^{(i)})
    \cdot (f_{0, \Mbf} - f^*)(\xbf^{(i)}) 
    - \E_{\xbf \sim p_0} \big[f(\xbf) \cdot
    (f_{0, \Mbf} - f^*)(\xbf)\big]\Big)\Big|
    \ge 2^{2j - 8} r^2 \sqrt{n} t_n^2 \Big) \\
    &\qquad \le 
    \frac{C}{2^{3j} r^3} 
    + \frac{C\|f_{0} - f^*\|_{\infty, \Mbf}^3}{2^{3j} r^3 n^{3/2} t_n^3} 
    + \frac{C\|f_{0} - f^*\|_{\infty, \Mbf}^3 V^3}{2^{6j} r^6 n^3 t_n^6},
\end{aligned}
\end{equation*}
where $\|\cdot\|_{\infty, \Mbf}$ denotes the supremum norm over  
$\prod_{j = 1}^{d} [-M_j/2, M_j/2]$
\begin{equation*}
    \|g\|_{\infty, \Mbf} 
    := \sup_{\xbf \in \prod_{j = 1}^{d} [-M_j/2, M_j/2]} |g(\xbf)|.
\end{equation*}
    
As a result, we have
\begin{equation*}
\begin{aligned}
    &\P\big(\|\lsefstinf - f_0\|_{p_0, 2} > rt_n \big) 
    = \sum_{j = 2}^{\infty} \P\big(\lsefstinf \in \F_j\big) \\
    &\qquad \le \sum_{j = 2}^{\infty} \Big[ \frac{C}{2^{3j} r^3} 
    + \frac{C(\|\xi_1\|_2^3 + \|f_{0} - f^*\|_{\infty, \Mbf}^3)}{2^{3j} r^3 n^{3/2} t_n^3} 
    + \frac{C\|\xi_1\|_3^3 V^3}{2^{6j} r^6 n^2 t_n^6} 
    + \frac{C\|\xi_1\|_3^3}{2^{3j}r^3 n^2 t_n^3} \\
    &\qquad \qquad \qquad
    + \frac{C_{B, s} a_{d, s}^6 V^3}{2^{3j} r^3 n^{3/2} t_n^3}
    + \frac{CV^3(V^3 + \|f_{0} - f^*\|_{\infty, 
    \Mbf}^3)}{2^{6j} r^6 n^3 t_n^6}\Big] \\
    &\qquad \le \frac{C}{r^3}
    + \frac{C (\|\xi_1\|_2^3 + \|f_{0} - f^*\|_{\infty, 
    \Mbf}^3) + C_{B, s} a_{d, s}^6 V^3}{r^3 n^{3/2} t_n^3} 
    + \frac{C\|\xi_1\|_3^3 V^3}{r^6 n^2 t_n^6}
    + \frac{C\|\xi_1\|_3^3}{r^3 n^2 t_n^3} 
    + \frac{CV^3(V^3 + \|f_{0} - f^*\|_{\infty, \Mbf}^3)}{r^6 n^3 t_n^6}.
\end{aligned}
\end{equation*}
Plugging in $r = t/t_n$, we obtain
\begin{equation*}
\begin{aligned}
    &\P\big(\|\lsefstinf - f_{0}\|_{p_0, 2} > t \big)
    \le \frac{C t_n^3}{t^3}
    + \frac{C (\|\xi_1\|_2^3 + \|f_{0} - f^*\|_{\infty, 
    \Mbf}^3) + C_{B, s} a_{d, s}^6 V^3}{n^{3/2} t^3} \\
    &\qquad \qquad \qquad \qquad \qquad \quad
    + \frac{C\|\xi_1\|_3^3 V^3}{n^2 t^6} 
    + \frac{C\|\xi_1\|_3^3}{n^2 t^3} 
    + \frac{CV^3(V^3 + \|f_{0} - f^*\|_{\infty, \Mbf}^3)}{n^3 t^6},
\end{aligned}
\end{equation*}
which holds for all $t \ge t_n$.
Thus, for every $t \ge 2t_n$, we have
\begin{equation*}
\begin{aligned}
    &\P\big(\|\lsefstinf - f^*\|_{p_0, 2} > t \big) 
    \le \P\big(\|\lsefstinf - f_0\|_{p_0, 2} > t - \|f_0 - f^*\|_{p_0, 2}\big) 
    \le \P\Big(\|\lsefstinf - f_0\|_{p_0, 2} > \frac{t}{2}\Big) \\
    &\quad \le \frac{C t_n^3}{t^3}
    + \frac{C (\|\xi_1\|_2^3 + \|f_{0} - f^*\|_{\infty, 
    \Mbf}^3) + C_{B, s} a_{d, s}^6 V^3}{n^{3/2} t^3}
    + \frac{C\|\xi_1\|_3^3 V^3}{n^2 t^6} 
    + \frac{C\|\xi_1\|_3^3}{n^2 t^3} + \frac{CV^3(V^3 + \|f_{0} - f^*\|_{\infty, 
    \Mbf}^3)}{n^3 t^6},
\end{aligned}
\end{equation*}
where the second inequality follows from the assumption 
$t_n > 4 \|f_0 - f^*\|_{p_0, 2}$. Since 
\begin{equation*}
    \int_{a}^{\infty} 2y \cdot \P(Y \ge y) dy = \E[(Y^2 - a^2)_+],
\end{equation*}
where $(\cdot)_+$ denotes the positive part, it follows that
\begin{align}\label{eq:risk-bound-t-n}
    \E\big[\|\lsefstinf - f^*\|_{p_0, 2}^2\big]
    &\le 4t_n^2 
    + \int_{2t_n}^{\infty} 2t \cdot 
    \P\big(\|\lsefstinf - f^*\|_{p_0, 2} > t\big) \, dt \nonumber \\ 
    \begin{split}
    &\le C t_n^2 
    + \frac{C (\|\xi_1\|_2^3 + \|f_{0} - f^*\|_{\infty, 
    \Mbf}^3) + C_{B, s} a_{d, s}^6 V^3}{n^{3/2} t_n}
    + \frac{C\|\xi_1\|_3^3 V^3}{n^2 t_n^4} \\
    &\qquad
    + \frac{C\|\xi_1\|_3^3}{n^2 t_n} 
    + \frac{CV^3(V^3 + \|f_{0} - f^*\|_{\infty, 
    \Mbf}^3)}{n^3 t_n^4}.
    \end{split}
\end{align}
    
We have just seen that once we establish the bounds 
\eqref{eq:expected-suprema-bound} on the expected suprema with some 
$t_n > 4 \|f_0 - f^*\|_{p_0, 2}$, 
we can bound the risk of $\lsefstinf$ in terms of $t_n$ as in the above display.
Our next goal is therefore to identify a suitable $t_n$ satisfying 
\eqref{eq:expected-suprema-bound}.
To this end, we first bound 
\begin{equation*}
    \E\Big[\sup_{f \in B(V, t)} \Big|\frac{1}{\sqrt{k}}
    \sum_{i = 1}^{k} \epsilon_i f(\xbf^{(i)})\Big|\Big]
\end{equation*}
for each $t > 0$ and $k = 1, \dots, n$, and then apply 
Theorem \ref{thm:han2019convergence} to transfer this bound to the expected 
supremum with $\xi_i$'s.
    
Fix $k \in \{1, \dots, n\}$.
Since $\|f\|_{\infty} \le C(V + t)$ for every $f \in B(V, t)$, 
Lemma \ref{lem:bracketing-integral-bound} gives
\begin{equation*}
    \E\Big[\sup_{f \in B(V, t)} \Big|\frac{1}{\sqrt{k}}
    \sum_{i = 1}^{k} \epsilon_i f(\xbf^{(i)})\Big|\Big]
    \le C J_{[ \ ]}(t, B(V, t), \| \cdot \|_{p_0, 2}) 
    \cdot \Big(1 + C(V + t) \cdot 
    \frac{J_{[ \ ]}(t, B(V, t), \| \cdot \|_{p_0, 2})}{t^2 \sqrt{k}}\Big).
\end{equation*}
Applying the entropy integral bound from 
Lemma \ref{lem:BVt-bracketing-integral-bound}, we obtain
\begin{align}\label{eq:expected-supremum-bound-k}
    &\E\Big[\sup_{f \in B(V, t)} \Big|\frac{1}{\sqrt{k}}
    \sum_{i = 1}^{k} \epsilon_i f(\xbf^{(i)})\Big|\Big]
    \le C_{B, s} a_{d, s} \bigg(t \log\Big(2 + \frac{V}{t}\Big)
    + V^{1/2} t^{1/2} 
    \Big[\log\Big(2 + \frac{V}{t}\Big)\Big]^{\widebar{s} - 1} \bigg) \nonumber \\
    &\qquad \qquad \qquad \qquad 
    \cdot \bigg[1 + C_{B, s} a_{d, s}(V + t)k^{-1/2}t^{-2} 
    \bigg(t \log\Big(2 + \frac{V}{t}\Big)
    + V^{1/2} t^{1/2} 
    \Big[\log\Big(2 + \frac{V}{t}\Big)\Big]^{\widebar{s} - 1}\bigg)\bigg] \nonumber \\
    &\quad \le 
    C_{B, s} a_{d, s} t \log\Big(2 + \frac{V}{t}\Big)
    + C_{B, s} a_{d, s} V^{1/2} t^{1/2} \Big[\log\Big(2 + \frac{V}{t}\Big)\Big]^{\widebar{s} - 1} 
    \nonumber + C_{B, s} a_{d, s}^2 k^{-1/2} V^{3/2} t^{-1/2} \Big[\log\Big(2 + \frac{V}{t}\Big)\Big]^{\widebar{s}} \\ 
    &\qquad \quad
    + C_{B, s} a_{d, s}^2 k^{-1/2} V^2 t^{-1} \Big[\log\Big(2 + \frac{V}{t}\Big)\Big]^{2(\widebar{s} - 1)} 
    + C_{B, s} a_{d, s}^2 k^{-1/2} t \Big[\log\Big(2 + \frac{V}{t}\Big)\Big]^2 \nonumber 
    \nonumber \\
    &\qquad \quad
    + C_{B, s} a_{d, s}^2 k^{-1/2} V^{1/2} t^{1/2} \Big[\log\Big(2 + \frac{V}{t}\Big)\Big]^{\widebar{s}} 
    + C_{B, s} a_{d, s}^2 k^{-1/2} V \Big[\log\Big(2 + \frac{V}{t}\Big)\Big]^{\max(2, 2(\widebar{s} - 1))}.
\end{align}
Recall that $a_{d, s} = d^{\widebar{s}}(1 + \log d)^{\widebar{s} - 1}$.
Let $\Psi: \R \to \R$ denote the function given by the right-hand side of
\eqref{eq:expected-supremum-bound-k}. 
A direct calculation shows that if
\begin{equation*}
\begin{aligned}
    t &\ge \max\Big(
    C_{B, s} a_{d, s} (V + 1) k^{-1/2} \log(2 + k),
    C_{B, s} a_{d, s}^{2/3} (V + 1) k^{-1/3} [\log(2 + k)]^{2(\widebar{s} - 1)/3}, \\
    &\qquad \qquad
    C_{B, s} a_{d, s}^{4/5} (V + 1) k^{-2/5} [\log(2 + k)]^{2\widebar{s}/5}, 
    C_{B, s} a_{d, s}^2 (V + 1) k^{-1} [\log(2 + k)]^{2}, \\
    &\qquad \qquad 
    C_{B, s} a_{d, s}^{4/3} (V + 1) k^{-2/3} [\log(2 + k)]^{2\widebar{s}/3}, 
    C_{B, s} a_{d, s} (V + 1) k^{-1/2} [\log(2 + k)]^{\max(1, \widebar{s} - 1)}
    \Big),
\end{aligned}
\end{equation*}
then
\begin{equation*}
    \Psi(t) \le \sqrt{k} t^2 / (V + 1).
\end{equation*}
To simplify this maximum, observe that for suitable constants $C, C_s > 0$, we have
\begin{equation*}
\begin{aligned}
    &\log(2 + x) \le x^{1/6} \ \text{ for all } x \ge C, \\
    &[\log(2 + x)]^{2\widebar{s}/5} \le x^{1/15} \ \text{ for all } x \ge C_s, \\
    &[\log(2 + x)]^2 \le x^{2/3} \ \text{ for all } x \ge C, \\
    &[\log(2 + x)]^{2\widebar{s}/3} \le x^{1/3} \ \text{ for all } x \ge C_s, \text{ and} \\
    &[\log(2 + x)]^{\max(1, \widebar{s} - 1)} \le x^{1/6} \ \text{ for all } x \ge C_s.
\end{aligned}
\end{equation*}
Using these inequalities, we can bound terms in the maximum as follows:
\begin{equation*}
\begin{aligned}
    &k^{-1/2} \log(2 + k) 
    \le C k^{-1/2} + k^{-1/2} k^{1/6} 
    \le C k^{-1/3}, \\
    &k^{-2/5} [\log(2 + k)]^{2\widebar{s}/5}
    \le C_s k^{-2/5} + k^{-2/5} k^{1/15} 
    \le C_s k^{-1/3}, \\
    &k^{-1} [\log(2 + k)]^{2} 
    \le C k^{-1} + k^{-1} k^{2/3}
    \le C k^{-1/3}, \\
    &k^{-2/3} [\log(2 + k)]^{2\widebar{s}/3} 
    \le C_s k^{-2/3} + k^{-2/3} k^{1/3}
    \le C_s k^{-1/3}, \text{ and} \\
    &k^{-1/2} [\log(2 + k)]^{\max(1, \widebar{s} - 1)} 
    \le C_s k^{-1/2} + k^{-1/2} k^{1/6} 
    \le C_s k^{-1/3}.
\end{aligned}
\end{equation*}
Thus, if we set
\begin{equation*}
    \widetilde{t}_k = C_{B, s} a_{d, s}^2 (V + 1) k^{-1/3}
    [\log(2 + n)]^{2(\widebar{s} - 1)/3},
\end{equation*}
then we have
\begin{equation*}
    \Psi(\widetilde{t}_k) \le \sqrt{k} \widetilde{t}_k^2 / (V + 1).
\end{equation*}
Since the map $t \mapsto \Psi(t)/t$ is decreasing, it follows that
\begin{equation*}
    \E\Big[\sup_{f \in B(V, r\widetilde{t}_k)} \Big|\frac{1}{\sqrt{k}}
    \sum_{i = 1}^{k} \epsilon_i f(\xbf^{(i)})\Big|\Big]
    \le \Psi(r\widetilde{t}_k) \le r\Psi(\widetilde{t}_k) 
    \le r\sqrt{k} \widetilde{t}_k^2 / (V + 1)
\end{equation*}
for every $r \ge 1$.
    
We have just shown that for each $k = 1, \dots, n$,
\begin{equation*}
    \E\Big[\sup_{f \in B(V, r\widetilde{t}_k)} \Big|
    \sum_{i = 1}^{k} \epsilon_i f(\xbf^{(i)})\Big|\Big]
    \le rk \widetilde{t}_k^2 / (V + 1)
\end{equation*}
for every $r \ge 1$. By Theorem \ref{thm:han2019convergence}, this bound 
transfers to the expected supremum with $\xi_i$, giving
\begin{equation*}
    \E\Big[\sup_{f \in B(V, r\widetilde{t}_n)} \Big|
    \sum_{i = 1}^{n} \xi_i f(\xbf^{(i)})\Big|\Big]
    \le 4 \|\xi_1\|_{3, 1} \cdot rn \widetilde{t}_n^2 / (V + 1)
\end{equation*}
for every $r \ge 1$.
Therefore, redefining $\widetilde{t}_n$ by multiplying it by the factor $(1 + 4 \|\xi_1\|_{3, 1})$
yields the first two inequalities in \eqref{eq:expected-suprema-bound} 
(with $t_n = \widetilde{t}_n$) for $r \ge 1$.
    
We now bound 
\begin{equation*}
    \E\Big[ \sup_{f \in B(V, t)} \Big| \frac{1}{\sqrt{n}} 
    \sum_{i = 1}^{n} \epsilon_i f(\xbf^{(i)})\cdot 
    (f_0 - f^*)(\xbf^{(i)})\Big|\Big] 
\end{equation*}
for each $t > 0$. By following the proof of Lemma \ref{lem:bracketing-integral-bound} 
with minimal modifications, we can show that
\begin{equation*}
\begin{aligned}
    &\E\Big[ \sup_{f \in B(V, t)} \Big| \frac{1}{\sqrt{n}} 
    \sum_{i = 1}^{n} \epsilon_i f(\xbf^{(i)})\cdot 
    (f_0 - f^*)(\xbf^{(i)})\Big|\Big] \\
    &\qquad\le C \|f_0 - f^*\|_{\infty, \Mbf} \cdot 
    J_{[ \ ]}(t, B(V, t), \| \cdot \|_{p_0, 2}) 
    \cdot \Big(1 + C(V + t) \cdot 
    \frac{J_{[ \ ]}(t, B(V, t), \| \cdot \|_{p_0, 2})}{t^2 \sqrt{n}}\Big).
\end{aligned}
\end{equation*}
Hence, repeating the computations above, we find that if we define  
$\widebar{t}_n$ as
\begin{equation*}
    \widebar{t}_n 
    = (1 + \|f_0 - f^*\|_{\infty, \Mbf}) 
    \cdot C_{B, s} a_{d, s}^2 (V + 1) n^{-1/3}
    [\log(2 + n)]^{2(\widebar{s} - 1)/3},
\end{equation*}
then
\begin{equation*}
    \E\Big[ \sup_{f \in B(V, \widebar{t}_n)} \Big| \frac{1}{\sqrt{n}} 
    \sum_{i = 1}^{n} \epsilon_i f(\xbf^{(i)})\cdot 
    (f_0 - f^*)(\xbf^{(i)})\Big|\Big] 
    \le \sqrt{n} \widebar{t}_n^2 / (V + 1),
\end{equation*}
from which the last inequality in \eqref{eq:expected-suprema-bound} (with $t_n = \widebar{t}_n$) 
follows for all $r \ge 1$.
    
Using $\widetilde{t}_n$ and $\widebar{t}_n$, we define $t_n$ as
\begin{equation*}
    t_n = 4\|f_0 - f^*\|_{p_0, 2} 
    + \max(\widetilde{t}_n, \widebar{t}_n).
\end{equation*}
Then, for every $r \ge 1$, 
\begin{equation*}
    \E\Big[\sup_{f \in B(V, rt_n)} \Big|\frac{1}{\sqrt{n}}
    \sum_{i = 1}^{n} \epsilon_i f(\xbf^{(i)})\Big|\Big]
    \le (rt_n / \widetilde{t}_n) \cdot \sqrt{n} \widetilde{t}_n^2 / (V + 1)
    \le r \sqrt{n} t_n^2 / (V + 1)
\end{equation*}
and the remaining two inequalities in \eqref{eq:expected-suprema-bound} 
follow by the same argument. Hence, $t_n$ satisfies all inequalities in 
\eqref{eq:expected-suprema-bound} for all $r \ge 1$, and we use this $t_n$
to derive a risk bound for $\lsefstinf$.
    
As a last step, we substitute our $t_n$ into \eqref{eq:risk-bound-t-n}. 
This yields
\begin{equation*}
\begin{aligned}
    &\E\big[\|\lsefstinf - f^*\|_{p_0, 2}^2\big] 
    \le  C \|f_0 - f^*\|_{p_0, 2}^2 
    + a_{d, s}^4 (V + 1)^2 \Big[C_{B, s} 
    (1 + \|\xi_1\|_{3, 1}^2 + \|f_0 - f^*\|_{\infty, \Mbf}^2) \\
    &\qquad \qquad \qquad \qquad \qquad \qquad \qquad \qquad \qquad \qquad
    \qquad \quad
    \cdot n^{-2/3} [\log(2 + n)]^{4(\widebar{s} - 1)/3} + O(n^{-2/3})\Big],
\end{aligned}
\end{equation*}
when $s \ge 2$. On the other hand, if $s = 1$, we obtain
\begin{equation*}
\begin{aligned}
    \E\big[\|\lsefstinf - f^*\|_{p_0, 2}^2\big] 
    &\le C\|f_0 - f^*\|_{p_0, 2}^2 
    + a_{d, 1}^4 (V + 1)^2 \Big[C\Big(\frac{\max(\widetilde{t}_n, 
    \widebar{t}_n)}{a_{d, 1}^2 (V + 1)}\Big)^2 \\
    &\qquad \qquad \qquad \qquad \qquad \qquad \qquad \quad 
    + \frac{C\|\xi_1\|_3^3 V^3}{a_{d, 1}^4 (V + 1)^2 n^2 
    (\max(\widetilde{t}_n, \widebar{t}_n))^4} + o(n^{-2/3})\Big] \\
    &= C\|f_0 - f^*\|_{p_0, 2}^2 
    + a_{d, 1}^4 (V + 1)^2 \cdot O(n^{-2/3}),
\end{aligned}
\end{equation*}
where the constant factors underlying $O(\cdot)$ 
depend on $B, s$, the moments of $\xi_i$, and 
$\|f_0 - f^*\|_{\infty, \Mbf}$.
\end{proof}
    
\begin{remark}
    For the first inequality in \eqref{eq:peeling-bound}, we in fact need to 
    show that
    \begin{equation*}
        \P\big(\lsefstinf \in \F_j\big) 
        \le \sup \Big\{\P\Big(\sup_{f \in \G} 
        \big(M_n(f) - M_n(f_{0, \Mbf})\big) \ge 0 \Big): 
        \G \subseteq \F_j \text{ is countable} \Big\},
    \end{equation*}
    since $\F_j$ may not be countable.
    Here, we give a more careful argument for this.
    
    For each integer $N \ge 1$, let $\G_N$ denote the subcollection of $\F_j$ 
    consisting of all $\fcnu$ (of the 
    form \eqref{f_fst}) that additionally satisfy the following two conditions:
    \begin{enumerate}[label = (\alph*)]
    \item $\nu_{L, U}$ are supported on 
    $\prod_{j \in L} ((1/N)\mathbb{Z} \cap (-M_j/2, M_j/2]) 
    \times \prod_{j \in U} ((1/N)\mathbb{Z} \cap (-M_j/2, M_j/2])$, where 
    $(1/N)\mathbb{Z} := \{m/N: m \in \mathbb{Z}\}$
    \item $c \in \mathbb{Q}$ and $\nu_{L, U}(\{(p_j, j \in L) 
    \times (q_j, j \in U)\}) \in \mathbb{Q}$ 
    for every $(p_j, j \in L) \times (q_j, j \in U) \in \R^{|L| + |U|}$.
    \end{enumerate}
    Clearly, each $\G_N$ is countable, and thus, $\G := \cup_{N \ge 1} \G_N$ is 
    countable as well. Since $\lsefstinf$ is constructed from discrete signed  
    measures with finite support, 
    it can be easily shown that there 
    exists a sequence $\{g_N\}_{N \ge 1}$ with $g_N \in \G_N \subseteq \G$ such 
    that $g_N(\xbf) \to \lsefstinf(\xbf)$ as $N \to \infty$ for every 
    $\xbf \in \R^d$. Hence, if $\lsefstinf \in \F_j$, then
    \begin{equation*}
        \sup_{f \in \G} \big(M_n(f) - M_n(f_{0, \Mbf})\big) 
        \ge \lim_{N \to \infty} \big(M_n(g_N) - M_n(f_{0, \Mbf})\big) 
        = M_n(\lsefstinf) - M_n(f_{0, \Mbf}) \ge 0.
    \end{equation*}
    Consequently,
    \begin{equation*}
    \begin{aligned}
        \P\big(\lsefstinf \in \F_j\big) 
        &\le \P\Big(\sup_{f \in \G} \big(M_n(f) - M_n(f_{0, \Mbf})\big) \ge 0\Big) \\
        &\le \sup \Big\{\P\Big(\sup_{f \in \mathcal{H}} \big(M_n(f) - M_n(f_{0, \Mbf})\big) \ge 0 \Big): 
        \mathcal{H} \subseteq \F_j \text{ is countable} \Big\}.
    \end{aligned}
    \end{equation*}
\end{remark}

\subsubsection{Proof of Lemma \ref{lem:bracketing-entropy-bound}}
\label{pf:bracketing-entropy-bound}

\begin{proof}[Proof of Lemma \ref{lem:bracketing-entropy-bound}]
    For each $\fcnu \in B(V, t)$, by modifying each 
basis function $b^{L, U}_{\lbf, \ubf}$ as in 
\eqref{eq:modification-of-indicator-function},
we can express $\fcnu$ as
\begin{equation}\label{eq:expression-of-f-b-mu-S}
    \fcnu(x_1, \dots, x_d) 
    = f^{d, s}_{b, \{\mu_S\}}(x_1, \dots, x_d) 
    := b + \sum_{0 < |S| \le s} 
    \int_{\R^{|S|}} \prod_{j \in S} \ind(x_j \ge l_j)
    \, d\mu_S(l_j, j \in S) 
\end{equation}
for some $b \in \R$ and finite signed Borel measures $\mu_S$ 
on $\R^{|S|}$, related to the original measures $\nu_{L, U}$ 
through \eqref{eq:relation-between-mu-and-nu}. 
Here, the summation runs over all nonempty subsets  
$S \subseteq [d]$ with $|S| \le s$.
Since each $\nu_{L, U}$ is supported on $\prod_{j \in L} (-M_j/2, M_j/2] 
\times \prod_{j \in U} (-M_j/2, M_j/2]$, the relation 
\eqref{eq:relation-between-mu-and-nu} implies that each $\mu_S$ is 
supported on $\prod_{j \in S} (-M_j/2, M_j/2]$. 
Moreover, by \eqref{eq:inequality-between-complexities}, 
\begin{equation*}
    \sum_{0 < |S| \le s} \|\mu_S\|_{\text{TV}}
    \le \min(2^s - 1, 2^d) \cdot \sum_{0 < |L| + |U| \le s} 
    \|\nu_{L, U}\|_{\text{TV}}
    \le (2^s - 1) V
    \le C_s V.
\end{equation*}
Hence, if we define $\widetilde{B}(V, t)$ as the collection of all 
functions $f^{d, s}_{b, \{\mu_S\}} \in \fstinf$ with 
$\|f^{d, s}_{b, \{\mu_S\}}\|_{p_0, 2} \le t$ such that $\mu_S$ are supported 
on $\prod_{j \in S} (-M_j/2, M_j/2]$ and satisfy
\begin{equation*}
    \sum_{0 < |S| \le s} \|\mu_S\|_{\text{TV}} \le V,
\end{equation*}
then we have $B(V, t) \subseteq \widetilde{B}(C_s V, t)$. 

We now split $\widetilde{B}(V, t)$ into pieces, compute the bracketing entropy 
of each piece, and then put them 
together to obtain a bracketing entropy bound for $\widetilde{B}(V, t)$, 
which will in turn yield a bound for the bracketing entropy of $B(V, t)$.
For every $f^{d, s}_{b, \{\mu_S\}} \in \widetilde{B}(V, t)$, by repeating the argument 
(using Cauchy inequality) in the proof of Theorem \ref{thm:risk-upper-bound}, 
we can show that $|b| \le C(V + t)$ for some constant $C > 0$.
Set $K = \floor{C(V + t)/\epsilon}$, and for each $k = -(K + 1), \dots, K$, 
let $\G_k$ denote the collection of all functions 
$f^{d, s}_{b, \{\mu_S\}}$ of the form \eqref{eq:expression-of-f-b-mu-S}
with $k \epsilon \le b \le (k + 1) \epsilon$ and with signed Borel measures 
$\mu_S$ supported on $\prod_{j \in S} (-M_j/2, M_j/2]$ and satisfying 
$\sum_{0 < |S| \le s} \|\mu_S\|_{\text{TV}} \le V$. 
It is clear that 
\begin{equation*}
    \widetilde{B}(V, t) \subseteq \bigcup_{k = -(K + 1), \dots, K} \G_k,
\end{equation*}
and hence,
\begin{align}\label{eq:constant-split}
    &\log N_{[ \ ]}(\epsilon, \widetilde{B}(V, t), \| \cdot \|_{p_0, 2}) 
    \le \log \Big( \sum_{k = -(K + 1), \dots, K} 
    N_{[ \ ]}(\epsilon, \G_k, \| \cdot \|_{p_0, 2}) \Big) \\
    &\quad \le \log\Big(2 + \frac{C(V + t)}{\epsilon}\Big) 
    + \sup_k \log N_{[ \ ]}(\epsilon, \G_k, \| \cdot \|_{p_0, 2})
    \le \log\Big(2 + \frac{C(V + t)}{\epsilon}\Big) 
    + \log N_{[ \ ]}(\epsilon, \G_0, \| \cdot \|_{p_0, 2}) \nonumber.
\end{align}
    
Now, let $\G_{\emptyset}$ denote the collection of all constant functions 
on $\R^d$ with values in $[0, \epsilon]$.
Also, for each nonempty $S \subseteq [d]$ with $|S| \le s$, define 
$\G_{S}$ as the collection of all functions on $\R^d$ of the form
\begin{equation*}
    (x_1, \dots, x_d) \mapsto 
    \int_{\prod_{j \in S} [-M_j/2, M_j/2]} \prod_{j \in S} \ind(x_j \ge l_j) 
    \, d\mu_S(l_j, j \in S),
\end{equation*}
where $\mu_S$ is a finite signed Borel measure 
on $\prod_{j \in S} [-M_j/2, M_j/2]$ with $\|\mu_S\|_{\text{TV}} \le V$. 
By construction, 
\begin{equation*}
    \G_0 \subseteq \G_{\emptyset} 
    + \bigplus_{0 < |S| \le s} \G_{S},
\end{equation*}
where $A + B = \{a + b: a \in A, b \in B\}$.
It follows that
\begin{equation}\label{eq:integral-split}
    \log N_{[ \ ]}(\epsilon, \G_0, \| \cdot \|_{p_0, 2}) 
    \le C_s(1 + \log d) + 
    \sum_{0 < |S| \le s} 
    \log N_{[ \ ]}\big(\epsilon/(C_s d^{\widebar{s}}), 
    \G_{S}, \| \cdot \|_{p_0, 2}\big),
\end{equation}
where $\widebar{s} = \min(s, d)$.

To proceed, for each nonempty $S \subseteq [d]$ with $|S| \le s$, let 
$\widetilde{\G}_S$ denote the collection of all functions on the 
truncated section $\prod_{j \in S} [-M_j/2, M_j/2]$ of the original 
domain $\R^d$, obtained by restricting to the coordinates indexed by 
$S$, of the form
\begin{equation*}
    (x_j, j \in S) \mapsto 
    \int_{\prod_{j \in S} [-M_j/2, M_j/2]} \prod_{j \in S} 
    \ind(x_j \ge l_j) \, d\mu_S(l_j, j \in S),
\end{equation*}
where $\mu_S$ is as in the definition of $\G_{S}$.
Since $p_0$ is uniformly bounded by $B/\prod_{j = 1}^{d} M_j$, we have
\begin{equation}\label{eq:sectional-domain}
    N_{[ \ ]}(\epsilon, \G_{S}, \| \cdot \|_{p_0, 2})
    \le N_{[ \ ]}\Big(\Big(\frac{\prod_{j \in S} M_j}{B}\Big)^{1/2} 
    \epsilon, \widetilde{\G}_S, \| \cdot \|_2\Big).
\end{equation}
Furthermore, for each nonempty $S \subseteq [d]$ with $|S| \le s$, let 
$\widebar{\G}_S$ denote the collection of all functions on the scaled domain 
$[0, 1]^{|S|}$ of the form
\begin{equation*}
    (x_j, j \in S) \mapsto 
    \int_{[0, 1]^{|S|}} \prod_{j \in S} 
    \ind(x_j \ge l_j) \, d\mu_S(l_j, j \in S),
\end{equation*}
where $\mu_S$ is a finite signed Borel measure on $[0, 1]^{|S|}$ with 
$\|\mu_S\|_{\text{TV}} \le V$. Through a straightforward scaling 
argument, it can be readily verified that
\begin{equation}\label{eq:integral-scaling}
    N_{[ \ ]}(\epsilon, \widetilde{\G}_S, \| \cdot \|_2) 
    \le N_{[ \ ]}\Big(
    \Big(\frac{1}{\prod_{j \in S} M_j}\Big)^{1/2} \epsilon, 
    \widebar{\G}_S, \| \cdot \|_2\Big) 
    = N_{[ \ ]}\Big(
    \Big(\frac{1}{\prod_{j \in S} M_j}\Big)^{1/2} \epsilon, 
    \widebar{\G}_{[|S|]}, \| \cdot \|_2\Big).
\end{equation}

Next, for each integer $m \ge 1$ and $R > 0$, let $\mathcal{H}_m(R)$ denote the 
collection of all functions on $[0, 1]^m$ of the form
\begin{equation*}
    (x_1, \dots, x_m) \mapsto \int_{[0, 1]^m} \prod_{j = 1}^m \ind(x_j \ge l_j) 
    \, d\mu(l_1, \dots, l_m)
\end{equation*}
where $\mu$ is a finite Borel measure (not a signed measure) on $[0, 1]^m$ with 
$\|\mu\|_{\text{TV}} \le R$. It was proved in \citet[Theorem 1.1]{gao2013bracketing} 
that
\begin{equation*}
    \log N_{[ \ ]}(\epsilon, \mathcal{H}_m(R), \| \cdot \|_2) 
    \le C_m \Big(2 + \frac{R}{\epsilon}\Big) 
    \Big[\log\Big(2 + \frac{R}{\epsilon}\Big)\Big]^{2(m - 1)}.
\end{equation*}
By the Jordan decomposition of signed measures, 
\begin{equation*}
    \widebar{\G}_{[|S|]} \subseteq 
    \mathcal{H}_{|S|}(V) - \mathcal{H}_{|S|}(V),
\end{equation*}
where $A - B = \{a - b: a \in A, b \in B\}$. It follows that
\begin{equation*}
    \log N_{[ \ ]}(\epsilon, \widebar{\G}_{[|S|]}, \| \cdot \|_2) 
    \le 2 \log N_{[ \ ]}\big(\frac{\epsilon}{2}, 
    \mathcal{H}_{|S|}(V), \| \cdot \|_2\big) 
    \le C_{|S|} \Big(2 + \frac{2V}{\epsilon} \Big) 
    \Big[\log\Big(2 + \frac{2V}{\epsilon}\Big)\Big]^{2(|S| - 1)}.
\end{equation*}
Substituting this bound back into \eqref{eq:integral-scaling}, 
\eqref{eq:sectional-domain}, \eqref{eq:integral-split}, and
\eqref{eq:constant-split} in turn, we obtain
\begin{equation*}
    \log N_{[ \ ]}(\epsilon, \widetilde{B}(V, t), \| \cdot \|_{p_0, 2}) 
    \le \log\Big(2 + \frac{C(V + t)}{\epsilon}\Big) 
    + C_{B, s} d^{2\widebar{s}}(1 + \log d)^{2(\widebar{s} - 1)}
    \Big(2 + \frac{V}{\epsilon}\Big) 
    \Big[\log\Big(2 + \frac{V}{\epsilon}\Big)\Big]^{2(\widebar{s} - 1)}.
\end{equation*}
Lastly, since $B(V, t) \subseteq \widetilde{B}(C_s V, t)$, we arrive at
\begin{equation*}
    \log N_{[ \ ]}(\epsilon, B(V, t), \| \cdot \|_{p_0, 2}) 
    \le \log\Big(2 + \frac{C_s(V + t)}{\epsilon}\Big) 
    + C_{B, s} d^{2\widebar{s}}(1 + \log d)^{2(\widebar{s} - 1)}
    \Big(2 + \frac{V}{\epsilon}\Big) 
    \Big[\log\Big(2 + \frac{V}{\epsilon}\Big)\Big]^{2(\widebar{s} - 1)},
\end{equation*}
which completes the proof.
\end{proof}

\subsubsection{Proof of Corollary \ref{cor:minimax-upper-bound}}
\label{pf:minimax-upper-bound}
\begin{proof}[Proof of Corollary \ref{cor:minimax-upper-bound}]
Observe from the proof of Theorem \ref{thm:risk-upper-bound} that 
the risk bound for $\lsefstinf$ depends continuously on $V$.
Consequently, the desired bound in the corollary follows directly 
from the risk bound for $\lsefstinf$ and the following inequality:
\begin{equation*}
    \minimax \le \liminf_{\epsilon \to 0+} 
    \sup_{\substack{f^* \in \fstinf \\ 
    \vinfxgb(f^*) \le V}} 
    \E \|\hat{f}^{d, s}_{n, V + \epsilon} - f^*\|_{p_0, 2}^2.
\end{equation*}
\end{proof}

\subsubsection{Proof of Theorem \ref{thm:minimax-lower-bound}}
\label{pf:minimax-lower-bound}
Our proof of Theorem \ref{thm:minimax-lower-bound} builds on the proof ideas of
\citet[Theorem 4.6]{fang2021multivariate}, which itself is motivated by the 
ideas in \citet[Section 4]{blei2007metric}. As in \citet[Theorem 4.6]{fang2021multivariate}, 
we use Assouad's lemma in the following form.

\begin{lemma}[Lemma 24.3 of \citet{van2000asymptotic} 
and Lemma 11.20 of \citet{ki2024mars}]
\label{lem:assouad}
Suppose $q$ is a positive integer, and we have $f_{\etabf} \in \fstinf$ with 
$\vinfxgb(f_{\etabf}) \le V$ for each $\etabf \in \{-1, 1\}^q$. Then, we have 
the following lower bound for the minimax risk $\minimax$:
\begin{equation*}
    \minimax \ge \frac{q}{8} \cdot \min_{\etabf \neq \etabf'} 
    \frac{\| f_\etabf - f_{\etabf'} \|_{p_0, 2}^2}{H(\etabf, \etabf')} 
    \cdot \min_{H(\etabf, \etabf') = 1} \bigg(1 - \sqrt{\frac{1}{2} \E 
    \big[ K(\P_{f_{\etabf}}, \P_{f_{\etabf'}}) \big]}\bigg).
\end{equation*}
Here, $H(\cdot, \cdot)$ denotes the Hamming distance 
$H(\etabf, \etabf') := \sum_{j = 1}^{q} 1\{\eta_j \neq \eta'_j\}$, 
$\P_f$ represents the probability distribution of $(y_1, \dots, y_n)$ given 
$(\xbf^{(1)}, \dots, \xbf^{(n)})$ when $f^{*} = f$, and $K(\cdot, \cdot)$ 
denotes the Kullback divergence between two probability distributions.
\end{lemma}

\begin{proof}[Proof of Theorem \ref{thm:minimax-lower-bound}]
Fix an integer $l$ as 
\begin{equation*}
    l = \bigg\lceil\frac{1}{3 \log 2} \Big\{\log\Big(
    \frac{C_{B, \widebar{s}} n V^2}{\sigma^2}\Big) 
    - (\widebar{s} - 1) \log\log\Big(
    \frac{C_{B, \widebar{s}} n V^2}{\sigma^2}
    \Big)\Big\}\bigg\rceil
\end{equation*}
where $C_{B, \widebar{s}} = B 2^{-4\widebar{s} + 1} 
(6 \log 2)^{\widebar{s} - 1} \cdot (\widebar{s} - 1)!$ and 
$\ceil{x}$ denotes the smallest integer greater than or equal to $x$. 
This choice of $l$ ensures that
\begin{equation}\label{2-minus-l}
    2^{-l} \le \Big(\frac{\sigma^2}{C_{B, \widebar{s}} n V^2}\Big)^{1/3} 
    \bigg[\log\Big(\frac{C_{B, \widebar{s}} 
    n V^2}{\sigma^2}\Big)\bigg]^{(\widebar{s} - 1)/3} 
    < 2^{-l + 1}.
\end{equation} 
Define 
\begin{equation*}
    P_l = \bigg\{(p_1, \dots, p_{\widebar{s}}) 
    \in \mathbb{Z}_{\ge 0}^{\widebar{s}}: 
    \sum_{j = 1}^{\widebar{s}} p_j = l\bigg\},
\end{equation*}
and, for each $\pbf = (p_1, \dots, p_{\widebar{s}}) \in P_l$, let 
\begin{equation*}
    I_{\pbf} = \big\{(i_1, \dots, i_{\widebar{s}}): i_j \in [2^{p_j}] 
    \mbox{ for each } j \in [\widebar{s}] \big\}.
\end{equation*}
Recall that $[m] = \{1, \dots, m\}$ for each integer $m \ge 1$. It is 
clear that $|I_{\pbf}| = 2^l$ for every $\pbf \in P_l$, and 
\begin{equation*}
    |P_l| = \binom{\widebar{s} + l - 1}{\widebar{s} - 1} 
    \ge \frac{l^{\widebar{s}- 1}}{(\widebar{s} - 1)!}.
\end{equation*}
Next, define
\begin{equation*}
    Q = \big\{(\pbf, \ibf): 
    \pbf \in P_l \mbox{ and } \ibf \in I_{\pbf} \big\}
\end{equation*}
and let $q = |Q| = |P_l| \cdot 2^l$.
In this proof, functions will be indexed by vectors $\etabf \in \{-1, 1\}^q$, 
whose components are indexed by the set $Q$.

For an integer $m \ge 1$ and $k \in [2^m]$, denote by $\psi_{m, k}$ the 
real-valued function on $(0, 1)$ defined by
\begin{equation*}
\begin{aligned}
    \psi_{m, k}(x) = 
    \begin{cases}
        1 &\mbox{if } x \in \big((k-1)2^{-m}, (k-3/4)2^{-m}\big) \cup 
        \big((k-1/4)2^{-m}, k2^{-m}\big), \\
        -1 &\mbox{if } x \in \big((k-3/4)2^{-m}, 
        (k-1/4)2^{-m}\big), \\
        0 &\mbox{otherwise}. 
    \end{cases}
\end{aligned}
\end{equation*}
Using these functions $\psi_{m, k}$, we construct $f_{\etabf} \in \fstinf$ 
for $\etabf \in \{-1, 1\}^q$ as follows, with which we will use Lemma 
\ref{lem:assouad} 
to prove the lower bound of the minimax risk $\minimax$. For each 
$\etabf \in \{-1, 1\}^q$, let $\nu_{\etabf}$ be the signed Borel measure on 
$\prod_{j = 1}^{\widebar{s}} (-M_j/2, M_j/2)$ defined by
\begin{equation*}
    d\nu_{\etabf}(\tbf) = \frac{1}{M_1 \cdots M_{\widebar{s}}} 
    \cdot \frac{V}{\sqrt{|P_l|}} 
    \sum_{\pbf \in P_l} \sum_{\ibf \in I_{\pbf}} 
    \eta_{\pbf, \ibf} \bigg(\prod_{j = 1}^{\widebar{s}} 
    \psi_{p_j, i_j}\Big(\frac{t_j}{M_j} + \frac{1}{2}\Big)\bigg) 
    \, d\tbf,
\end{equation*}
and define $f_{\etabf}: \R^d \rightarrow \R$ as
\begin{equation*}
    f_{\etabf}(x_1, \dots, x_d) 
    = \int_{\prod_{j = 1}^{\widebar{s}} (-M_j/2, M_j/2)} 
    \prod_{j = 1}^{\widebar{s}} \ind(x_j \ge t_j) 
    \, d\nu_{\etabf}(\tbf).
\end{equation*}
Clearly, $f_{\etabf} \in \fstinf$ for every $\etabf \in \{-1, 1\}^q$. 
The following lemma, whose proof is deferred to 
Appendix \ref{pf:f-eta-prop}, summarizes the key properties of the 
functions $f_{\etabf}$ we need for the proof.

\begin{lemma}\label{lem:f-eta-prop}
For each $\etabf \in \{-1, 1\}^q$, the complexity of $f_{\etabf}$ is 
bounded by $V$, i.e.,
\begin{equation}\label{f-eta-compl}
    \vinfxgb(f_\etabf) \le V.
\end{equation}
We also have 
\begin{equation}\label{f-eta-first-dist-prop}
    \max_{H(\etabf, \etabf') = 1} 
    \| f_\etabf - f_{\etabf'} \|_{p_0, 2}^2 
    \le \frac{B V^2}{|P_l|} \cdot 2^{-3l - 4\widebar{s} + 2}
\end{equation}
and 
\begin{equation}\label{f-eta-second-dist-prop}
    \min_{\etabf \neq \etabf'} 
    \frac{\| f_\etabf - f_{\etabf'}\|_{p_0, 2}^2}{H(\etabf, \etabf')} 
    \ge \frac{b V^2}{|P_l|} \cdot 2^{-3l - 6\widebar{s} + 2}.
\end{equation}
\end{lemma}

It is straightforward to check that the Kullback divergence between 
$\P_{f_\etabf}$ and $\P_{f_{\etabf'}}$ for 
$\etabf, \etabf' \in \{-1, 1\}^q$ can be computed by 
\begin{equation*}
    K(\P_{f_{\etabf}}, \P_{f_{\etabf'}}) 
    = \frac{1}{2\sigma^2} \sum_{i=1}^{n} 
    \big(f_{\etabf}(\xbf^{(i)}) - f_{\etabf'}(\xbf^{(i)})\big)^2.
\end{equation*}
Hence, \eqref{f-eta-first-dist-prop} gives
\begin{equation}\label{prob-dist-prop}
    \max_{H(\etabf, \etabf') = 1} 
    \E \big[K(\P_{f_{\etabf}}, \P_{f_{\etabf'}})\big] 
    = \frac{n}{2\sigma^2} \cdot  \max_{H(\etabf, \etabf') = 1} 
    \| f_{\etabf} - f_{\etabf'} 
    \|_{p_0, 2}^2 \le \frac{B n V^2}{\sigma^2|P_l|} 
    \cdot 2^{-3l - 4\widebar{s} + 1}.
\end{equation}

Applying Lemma \ref{lem:assouad}, along with 
\eqref{f-eta-second-dist-prop} and \eqref{prob-dist-prop}, we can bound 
the minimax risk $\minimax$ as
\begin{equation*}  
\begin{aligned}
    \minimax &\ge \frac{q}{8} \cdot \frac{b V^2}{|P_l|} 
    \cdot 2^{-3l - 6\widebar{s} + 2} 
    \bigg(1 - \sqrt{\frac{B n V^2}{\sigma^2|P_l|} 
    \cdot 2^{-3l - 4\widebar{s}}}\bigg) \\
    &\ge b V^2 2^{-2l - 6\widebar{s} - 1} 
    \bigg(1 - \sqrt{\frac{1}{2(6 \log 2)^{\widebar{s} - 1}} 
    \cdot \frac{C_{B, \widebar{s}} n V^2}{\sigma^2} 
    \cdot \frac{2^{-3l}}{l^{\widebar{s} - 1}}}\bigg).
\end{aligned}
\end{equation*}
Recall that $q = |P_l| \cdot 2^l$, 
$|P_l| \ge l^{\widebar{s}- 1}/(\widebar{s} - 1)!$, and 
$C_{B, \widebar{s}} = B 2^{-4\widebar{s} + 1} 
(6 \log 2)^{\widebar{s} - 1} \cdot (\widebar{s} - 1)!$.
Our choice of $l$ (at the beginning of the proof) implies that
\begin{equation*}
\begin{aligned}
    \frac{1}{2(6 \log 2)^{\widebar{s} - 1}} \cdot 
    \frac{C_{B, \widebar{s}} n V^2}{\sigma^2} 
    \cdot \frac{2^{-3l}}{l^{\widebar{s} - 1}} 
    &\le \frac{1}{2} \cdot \Bigg[\frac{\log\big(
        \frac{C_{B, \widebar{s}} n V^2}{\sigma^2}\big)}
        {2\Big(\log\big(\frac{C_{B, \widebar{s}} n V^2}{\sigma^2}\big) 
        - (\widebar{s} - 1) \log \log \big(\frac{C_{B, \widebar{s}} 
        n V^2}{\sigma^2}\big)\Big)}\Bigg]^{\widebar{s} - 1} \\
    &\le \frac{1}{2} \cdot \bigg[2\bigg(1 - (\widebar{s} - 1) 
    \frac{\log \log\big(\frac{C_{B, \widebar{s}} n V^2}{\sigma^2}\big)}
    {\log \big(\frac{C_{B, \widebar{s}} n V^2}{\sigma^2}\big)}
    \bigg)\bigg]^{-(\widebar{s} - 1)} \\
    &\le \frac{1}{2} \cdot \bigg[2\bigg(1 - (\widebar{s} - 1) 
    \Big\{\log\Big(\frac{C_{B, \widebar{s}} 
    n V^2}{\sigma^2}\Big)\Big\}^{-\frac{1}{2}} 
    \bigg)\bigg]^{-(\widebar{s} - 1)}.
\end{aligned}
\end{equation*}
Here, the first inequality follows from \eqref{2-minus-l} and 
\begin{equation*}
    l \ge \frac{1}{3 \log 2}\bigg\{\log\Big(\frac{C_{B, \widebar{s}} 
    n V^2}{\sigma^2}\Big) 
    - (\widebar{s} - 1) \log \log\Big(\frac{C_{B, \widebar{s}} 
    n V^2}{\sigma^2}\Big)\bigg\},
\end{equation*}
and the last inequality is from the inequality 
$\log \log x / \log x \le (\log x)^{-1/2}$, which is valid for all 
$x > 1$. If we assume that 
\begin{equation*}
    n \ge \frac{e^{4\widebar{s}^2}}{C_{B, \widebar{s}}} \cdot 
    \frac{\sigma^2}{V^2},
\end{equation*}
then 
\begin{equation*}
    \frac{1}{2(6 \log 2)^{\widebar{s} - 1}} \cdot 
    \frac{C_{B, \widebar{s}} n V^2}{\sigma^2} \cdot 
    \frac{2^{-3l}}{l^{\widebar{s} - 1}} \le 
    \frac{1}{2} \cdot \Big(1 
    + \frac{1}{\widebar{s}}\Big)^{-(\widebar{s} - 1)} 
    \le \frac{1}{2},
\end{equation*}
and thereby, we have
\begin{equation*}
\begin{aligned}
    \minimax &\ge \bigg(1 - \sqrt{\frac{1}{2}}\bigg) 
    \cdot b V^2 2^{-2l - 6\widebar{s} - 1} 
    \ge \bigg(1 - \sqrt{\frac{1}{2}}\bigg) 
    \cdot b V^2 2^{-6\widebar{s} - 3} \cdot 
    \Big(\frac{\sigma^2}{C_{B, \widebar{s}} n V^2}\Big)^{2/3}
    \bigg[\log\Big(\frac{C_{B, \widebar{s}} 
    n V^2}{\sigma^2}\Big)\bigg]^{2(\widebar{s} - 1)/3} \\
    &\ge C_{b, B, \widebar{s}}\Big(\frac{\sigma^2 V}{n}\Big)^{2/3}
    \bigg[\log\Big(\frac{C_{B, \widebar{s}} 
    n V^2}{\sigma^2}\Big)\bigg]^{2(\widebar{s} - 1)/3},
\end{aligned}
\end{equation*}
where $C_{b, B, \widebar{s}} = (1 - 1/\sqrt{2}) 
\cdot b 2^{-6\widebar{s} - 3} C_{B, \widebar{s}}^{-2/3}$.
Here, \eqref{2-minus-l} is used again for the second inequality.
Lastly, since 
\begin{equation*}
    \log\Big(\frac{C_{B, \widebar{s}} n V^2}{\sigma^2}\Big) 
    \ge \frac{1}{2} 
    \log\Big(\frac{n V^2}{\sigma^2}\Big)
\end{equation*}
provided that $n \ge (1/C_{B, \widebar{s}}^2) \cdot (\sigma^2/V^2)$, 
by further assuming that 
\begin{equation*}
    n \ge \max\Big\{\frac{e^{4\widebar{s}^2}}{C_{B, \widebar{s}}} 
    \cdot \frac{\sigma^2}{V^2}, 
    \frac{1}{C_{B, \widebar{s}}^2} \cdot \frac{\sigma^2}{V^2}\Big\},
\end{equation*}
we can derive the lower bound
\begin{equation*}
    \minimax \ge C_{b, B, \widebar{s}}'
    \Big(\frac{\sigma^2 V}{n}\Big)^{2/3}
    \bigg[\log\Big(\frac{n V^2}{\sigma^2}\Big)
    \bigg]^{2(\widebar{s} - 1)/3}
\end{equation*}
where $C_{b, B, \widebar{s}}' = C_{b, B, \widebar{s}} 
\cdot 2^{-2(\widebar{s}-1)/3}$.
\end{proof}

\subsection{Proofs of Proposition and Lemma in Section \ref{sec:discussion}}

\subsubsection{Proof of Proposition \ref{prop:infimal-convolution}}
\label{pf:infimal-convolution}

\begin{proof}[Proof of Proposition \ref{prop:infimal-convolution}]
Suppose $f_{\abf} \in \fstinf$ for  
$\abf \in \{-\infty, +\infty\}^d$ and 
$\sum_{\abf \in \{-\infty, +\infty\}^d} f_{\abf} \equiv f$.
By repeating the argument in the proof of 
Proposition \ref{prop:vinfxgb-hka-rel}, it can be shown that 
for every $g \in \fstinf$,
\begin{equation*}
    \oldvinfxgb(g) \le V_{\abf}(g) = \hk_{\abf}(g),
\end{equation*}
where $V_{\abf}(\cdot)$ is defined as in \eqref{eq:va-def}.
Applying this inequality to each $f_{\abf} \in \fstinf$ yields
\begin{equation*}
    \sum_{\abf \in \{-\infty, +\infty\}^d} 
    \hk_{\abf}(f_{\abf}) 
    \ge \sum_{\abf \in \{-\infty, +\infty\}^d} 
    \oldvinfxgb(f_{\abf}) 
    \ge \oldvinfxgb(f), 
\end{equation*}
which proves one direction of the desired identity:
\begin{equation*}
    \oldvinfxgb(f) \le \inf \Big\{\sum_{\abf \in \{-\infty, +\infty\}^d} 
    \hk_{\abf}(f_{\abf}): 
    \sum_{\abf \in \{-\infty, +\infty\}^d} f_{\abf} \equiv f,
    \ f_{\abf} \in \fstinf \ \forall \abf\Big\}.
\end{equation*}

We now turn to the reverse inequality. Suppose 
$f \in \fstinf$ is expressed as
\begin{equation*}
    f(x_1, \dots, x_d) =  c +
    \sum_{\substack{L, U : L \cap U = \emptyset \\
    0 < |L| + |U| \leq s }} 
    \int_{\R^{|L| + |U|}} \prod_{j \in L} \ind(x_j \ge l_j) 
    \cdot \prod_{j \in U} \ind(x_j < u_j)
    \, d \nu_{L, U}(\lbf, \ubf)
\end{equation*}
for some $c \in \R$ and signed Borel measures 
$\nu_{L, U}$ on $\R^{|L| + |U|}$. For each 
$\abf \in \{-\infty, +\infty\}^d$, define 
$f_{\abf}: \R^d \to \R$ by
\begin{equation*}
\begin{aligned}
    &f_{\abf}(x_1, \dots, x_d) 
    = c \cdot \prod_{j = 1}^{d} \ind(a_j = -\infty) 
    + \sum_{\substack{L, U : L \cap U = \emptyset \\
    0 < |L| + |U| \leq s }}
    \prod_{j \in U^c} \ind(a_j = -\infty)
    \cdot \prod_{j \in U} \ind(a_j = +\infty) \\
    &\qquad \qquad \qquad \qquad \qquad \qquad \qquad \qquad 
    \qquad \qquad \qquad
    \cdot \int_{\R^{|L| + |U|}} \prod_{j \in L} \ind(x_j \ge l_j) 
    \cdot \prod_{j \in U} \ind(x_j < u_j)
    \, d\nu_{L, U}(\lbf, \ubf).
\end{aligned}
\end{equation*}
It is clear that $f_{\abf} \in \fstinf$ for all 
$\abf \in \{-\infty, +\infty\}^d$.
Moreover, since for each integral over $\nu_{L, U}$, 
there is exactly one value of $\abf$ that makes all multiplied 
indicator functions equal to one, we have
\begin{equation*}
    \sum_{\abf \in \{-\infty, +\infty\}^d} f_{\abf} 
    \equiv f.
\end{equation*}

Fix $\abf \in \{-\infty, +\infty\}^d$. For each nonempty 
$S \subseteq [d]$, we have
\begin{equation*}
\begin{aligned}
    &f^S_{(a_j, j \in S^c)}(x_j, j \in S) 
    = c \cdot \prod_{j = 1}^{d} \ind(a_j = -\infty) 
    + \sum_{\substack{L, U \subseteq S: L \cap U = \emptyset \\
    0 < |L| + |U| \leq s }}
    \prod_{j \in U^c} \ind(a_j = -\infty)
    \cdot \prod_{j \in U} \ind(a_j = +\infty) \\
    &\qquad \qquad \qquad \qquad \qquad \qquad \qquad \qquad 
    \qquad \qquad \qquad
    \cdot \int_{\R^{|L| + |U|}} \prod_{j \in L} \ind(x_j \ge l_j) 
    \cdot \prod_{j \in U} \ind(x_j < u_j)
    \, d\nu_{L, U}(\lbf, \ubf).
\end{aligned}
\end{equation*}
Hence, for each nonempty $T \subseteq [d]$,
\begin{equation*}
\begin{aligned}
    &\sum_{S \subseteq T} (-1)^{|T| - |S|} \cdot 
    f^S_{(a_j, j \in S^c)}(x_j, j \in S) \\
    &\ \ = \sum_{S \subseteq T} (-1)^{|T| - |S|} 
    \sum_{\substack{L, U \subseteq S: L \cap U = \emptyset \\
    0 < |L| + |U| \leq s }}
    \prod_{j \in U^c} \ind(a_j = -\infty)
    \cdot \prod_{j \in U} \ind(a_j = +\infty) \\
    &\qquad \qquad \qquad \qquad \qquad \qquad \qquad \quad 
    \cdot \int_{\R^{|L| + |U|}} \prod_{j \in L} \ind(x_j \ge l_j) 
    \cdot \prod_{j \in U} \ind(x_j < u_j)
    \, d\nu_{L, U}(\lbf, \ubf) \\
    &\ \ = \sum_{\substack{L, U: L \cap U = \emptyset \\
    0 < |L| + |U| \leq s }}
    \Big(\sum_{S: L \cup U \subseteq S \subseteq T} (-1)^{|T| - |S|}\Big)
    \prod_{j \in U^c} \ind(a_j = -\infty)
    \cdot \prod_{j \in U} \ind(a_j = +\infty) \\
    &\qquad \qquad \qquad \quad
    \cdot \int_{\R^{|L| + |U|}} \prod_{j \in L} \ind(x_j \ge l_j) 
    \cdot \prod_{j \in U} \ind(x_j < u_j)
    \, d\nu_{L, U}(\lbf, \ubf) \\
    &\ \ = \sum_{\substack{L, U: L \cap U = \emptyset, L \cup U = T \\
    0 < |L| + |U| \leq s }}
    \prod_{j \in U^c} \ind(a_j = -\infty)
    \cdot \prod_{j \in U} \ind(a_j = +\infty) 
    \cdot \int_{\R^{|L| + |U|}} \prod_{j \in L} \ind(x_j \ge l_j) 
    \cdot \prod_{j \in U} \ind(x_j < u_j)
    \, d\nu_{L, U}(\lbf, \ubf)
\end{aligned}
\end{equation*}
if $|T| \le s$, and it vanishes otherwise.
For each nonempty $T \subseteq [d]$ with $|T| \le s$, since there is at 
most one pair of $(L, U)$ with $L \cap U = \emptyset$ and
$L \cup U = T$ such that
\begin{equation*}
    \prod_{j \in U^c} \ind(a_j = -\infty)
    \cdot \prod_{j \in U} \ind(a_j = +\infty) = 1,
\end{equation*}
by repeating the computation in the proof of 
Proposition \ref{prop:vinfxgb-hka-rel}, we obtain
\begin{equation*}
\begin{aligned}
    \vit(f^T_{(a_j, j \in T^c)}) 
    &= \vit\Big((x_j, j \in T) \mapsto \sum_{S \subseteq T} (-1)^{|T| - |S|} 
    \cdot f^S_{(a_j, j \in S^c)}(x_j, j \in S)\Big) \\
    &= \sum_{\substack{L, U: L \cap U = \emptyset, L \cup U = T \\
    0 < |L| + |U| \leq s }}
    \prod_{j \in U^c} \ind(a_j = -\infty)
    \cdot \prod_{j \in U} \ind(a_j = +\infty) 
    \cdot \|\nu_{L, U}\|_{\text{TV}}.
\end{aligned}
\end{equation*}
Therefore,
\begin{equation*}
\begin{aligned}
    \hk_{\abf}(f_{\abf}) 
    &= \sum_{0 < |T| \le d} \vit(f^T_{(a_j, j \in T^c)}) 
    = \sum_{0 < |T| \le s} \vit(f^T_{(a_j, j \in T^c)}) \\
    &= \sum_{0 < |T| \le s} 
    \sum_{\substack{L, U: L \cap U = \emptyset, L \cup U = T \\
    0 < |L| + |U| \leq s }}
    \prod_{j \in U^c} \ind(a_j = -\infty)
    \cdot \prod_{j \in U} \ind(a_j = +\infty) 
    \cdot \|\nu_{L, U}\|_{\text{TV}} \\
    &= \sum_{\substack{L, U: L \cap U = \emptyset \\
    0 < |L| + |U| \leq s }}
    \prod_{j \in U^c} \ind(a_j = -\infty)
    \cdot \prod_{j \in U} \ind(a_j = +\infty) 
    \cdot \|\nu_{L, U}\|_{\text{TV}}.
\end{aligned}
\end{equation*}

Summing the above identity over all 
$\abf \in \{-\infty, +\infty\}^d$ gives
\begin{equation*}
\begin{aligned}
    &\inf \Big\{\sum_{\abf \in \{-\infty, +\infty\}^d} 
    \hk_{\abf}(f_{\abf}): 
    \sum_{\abf \in \{-\infty, +\infty\}^d} f_{\abf} \equiv f,
    \ f_{\abf} \in \fstinf \ \forall \abf\Big\}
    \le \sum_{\abf \in \{-\infty, +\infty\}^d} 
    \hk_{\abf}(f_{\abf}) \\
    &\qquad = 
    \sum_{\substack{L, U: L \cap U = \emptyset \\
    0 < |L| + |U| \leq s }}
    \sum_{\abf \in \{-\infty, +\infty\}^d}
    \Big(\prod_{j \in U^c} \ind(a_j = -\infty)
    \cdot \prod_{j \in U} \ind(a_j = +\infty)\Big)
    \cdot \|\nu_{L, U}\|_{\text{TV}} 
    = \sum_{\substack{L, U: L \cap U = \emptyset \\
    0 < |L| + |U| \leq s }}
    \|\nu_{L, U}\|_{\text{TV}}.
\end{aligned}
\end{equation*}
Taking the infimum over all possible representations 
$\fcnu$ of $f$, we arrive at
\begin{equation*}
    \inf \Big\{\sum_{\abf \in \{-\infty, +\infty\}^d} 
    \hk_{\abf}(f_{\abf}): 
    \sum_{\abf \in \{-\infty, +\infty\}^d} f_{\abf} \equiv f,
    \ f_{\abf} \in \fstinf \ \forall \abf\Big\}
    \le \oldvinfxgb(f),
\end{equation*}
which completes the proof.
\end{proof}

\subsubsection{Proof of Lemma \ref{lem:zero-xgb-penalty}}
\label{pf:zero-xgb-penalty}

\begin{proof}[Proof of Lemma \ref{lem:zero-xgb-penalty}]
Suppose that $f \in \fst$ and that $f = \sum_{k = 1}^K f_k$, where 
each $f_k$ is a regression tree with right-continuous splits and 
depth at most $s$. Let $\wbf_k$ denote the leaf-weight vector of $f_k$.

Fix an integer $L \ge 1$. For each $k$, define the regression tree 
$g_{k, L} := (1/L) f_k$, obtained by scaling each leaf weight of $f_k$ 
by $1/L$ while keeping the same tree structure. Using these $g_{k, L}$, 
we can represent $f$ as a sum of $KL$ trees:
\begin{equation*}
    f = \sum_{k = 1}^K \sum_{l = 1}^{L} g_{k, L}.
\end{equation*}

For this representation, the sum of the $p^{\text{th}}$ powers of the 
leaf weights equals
\begin{equation*}
    \sum_{k = 1}^K \sum_{l = 1}^{L} \|(1/L) \cdot \wbf_{k}\|_p^p 
    = \frac{1}{L^{p - 1}} \sum_{k = 1}^K \|\wbf_{k}\|_p^p.
\end{equation*}
Because $p > 1$, this quantity converges to $0$ as $L \to \infty$.  
This proves that $\vxgb(f; p) = 0$.
\end{proof}

\subsection{Proofs of Lemmas in Appendix \ref{pf:minimax-rate}}

\subsubsection{Proof of Lemma \ref{lem:bracketing-integral-bound}}
\label{pf:bracketing-integral-bound}
Lemma \ref{lem:bracketing-integral-bound} is a corollary of the following 
more general result involving Bernstein norm. For a random variable $X$ 
with law $P$ on $\mathcal{X}$ and a function $f: \mathcal{X} \to \R$, 
the Bernstein norm of $f$ is defined by 
\begin{equation*}
    \|f\|_{P, B} = \Big(2 \E_P\big[\exp(|f(X)|) - 1 - |f(X)|\big]\Big)^{1/2}.
\end{equation*}
Although $\|\cdot\|_{P, B}$ does not satisfy homogeneity or the triangle 
inequality and therefore is not actually a norm, it is conventionally called a 
norm and can still be used for measuring the ``size" of functions.

\begin{lemma}[Lemma 3.4.3 of \citet{vaartwellner96book}]
\label{lem:entropy-integral-bound-Bernstein}
    Suppose $\xbf^{(1)}, \dots, \xbf^{(k)}$ are i.i.d. with law $P$ 
    on $\mathcal{X}$ and $\F$ is a countable collection of functions from 
    $\mathcal{X}$ to $\R$ where $\|f\|_{P, B} \le \delta$ for all $f \in \F$. 
    Then, there exists a constant $C > 0$ such that
    \begin{equation*}
        \E_P\Big[\sup_{f \in \F} \Big| \frac{1}{\sqrt{k}} 
        \sum_{i = 1}^{k} f(\xbf^{(i)})\Big|\Big] 
        \le C J_{[ \ ]}(\delta, \F, \| \cdot \|_{P, B}) 
        \cdot \Big(1 + \frac{J_{[ \ ]}(\delta, \F, 
        \| \cdot \|_{P, B})}{\delta^2 \sqrt{k}}\Big).
    \end{equation*}
\end{lemma}

\begin{proof}[Proof of Lemma \ref{lem:bracketing-integral-bound}]
Let $P$ be the law of $(\xbf^{(i)}, \epsilon_i)$ on 
$\R^d \times \{-1, 1\}$ and let $\G$ denote the collection of all functions 
$\Phi_f$ on $\R^d \times \{-1, 1\}$, one for each $f \in \F$, defined by 
\begin{equation*}
    \Phi_f(\xbf, \epsilon) = \frac{\epsilon f(\xbf)}{2D}
    \quad \text{for } (\xbf, \epsilon) \in \R^d \times \{-1, 1\}.
\end{equation*}
For every $\Phi_f \in \G$, we have
\begin{equation*}
\begin{aligned}
    \|\Phi_f\|_{P, B} 
    &= \bigg(2\E_{(\xbf, \epsilon) \sim P} 
    \Big[\exp\Big(\Big|\frac{\epsilon f(\xbf)}{2D}\Big|\Big)
    - 1 - \Big|\frac{\epsilon f(\xbf)}{2D}\Big|\Big]\bigg)^{1/2}
    = \bigg(2 \sum_{m = 2}^{\infty} \frac{1}{m!} \cdot 
    \E_{\xbf \sim p_0} 
    \Big[\Big|\frac{f(\xbf)}{2D}\Big|^m \Big]\bigg)^{1/2} \\
    &\le \bigg(2 \sum_{m = 2}^{\infty} \frac{1}{m!} \cdot 
    \E_{\xbf \sim p_0} 
    \Big[\Big|\frac{f(\xbf)}{2D}\Big|^2 \Big]\bigg)^{1/2} 
    \le \frac{(e - 2)^{1/2} t}{2^{1/2}D} 
    := \frac{a t}{2D},
\end{aligned}
\end{equation*}
where the first inequality uses the fact that $\|f\|_{\infty} \le D$, 
and the second inequality follows from the fact that $\|f\|_{p_0, 2} \le t$.
Applying Lemma \ref{lem:entropy-integral-bound-Bernstein} with $\G$ and 
$\delta = a t/2D$, we obtain
\begin{equation}\label{eq:bracketing-integral-bound-Bernstein-G}
\begin{aligned}
    \E\Big[\sup_{f \in \F} \Big|\frac{1}{\sqrt{k}} 
    \sum_{i = 1}^{k} \epsilon_i f(\xbf^{(i)})\Big|\Big] 
    &= 2D \cdot \E\Big[\sup_{\Phi_f \in \G} \Big|\frac{1}{\sqrt{k}} 
    \sum_{i = 1}^{k} \Phi_f(\xbf^{(i)}, \epsilon_i)\Big|\Big] \\
    &\le 2D \cdot C J_{[ \ ]}\Big(\frac{a t}{2D}, \G, \| \cdot \|_{P, B}\Big) 
    \cdot \Big(1 +  
    \frac{J_{[ \ ]}(\frac{a t}{2D}, \G, 
    \| \cdot \|_{P, B})}{(\frac{a t}{2D})^2 \sqrt{k}}\Big).
\end{aligned}
\end{equation}

Next, we relate the bracketing entropy integral of $\G$ in the Bernstein norm 
to that of $\F$ in the $\|\cdot\|_{p_0, 2}$ norm.
Fix $\Phi_f \in \G$, and let $[f_1, f_2]$ be a bracket containing $f$. 
Since $\|f\|_{\infty} \le D$, 
by replacing $f_1$ with $\xbf \mapsto \max(\min(f_1(\xbf), D), -D)$ 
if necessary, we assume that $\|f_1\|_{\infty} \le D$.
Similarly, we assume that $\|f_2\|_{\infty} \le D$.
Define $\Phi_1,\Phi_2:\R^d \times \{-1, 1\} \to \R$ by
\begin{equation*}
    \Phi_1(\xbf, \epsilon) = f_1(\xbf) \cdot \ind(\epsilon = 1) 
    - f_2(\xbf) \cdot \ind(\epsilon = -1) 
    \quad \text{for } (\xbf, \epsilon) \in \R^d \times \{-1, 1\}
\end{equation*}
and
\begin{equation*}
    \Phi_2(\xbf, \epsilon) = f_2(\xbf) \cdot \ind(\epsilon = 1) 
    - f_1(\xbf) \cdot \ind(\epsilon = -1)
    \quad \text{for } (\xbf, \epsilon) \in \R^d \times \{-1, 1\}.
\end{equation*}
Clearly, the bracket $[\Phi_1, \Phi_2]$ contains $\Phi_f$.
Moreover, since $\|f_2 - f_1\|_{\infty} \le 2D$, by the same argument as
above, we obtain
\begin{equation*}
    \|\Phi_2 - \Phi_1\|_{P, B} 
    = \bigg(2\E_{\xbf \sim p_0} 
    \Big[\exp\Big(\Big|\frac{(f_2 - f_1)(\xbf)}{2D}\Big|\Big)
    - 1 - \Big|\frac{(f_2 - f_1)(\xbf)}{2D}\Big|\Big]\bigg)^{1/2} 
    \le \frac{a}{2D} \cdot \|f_2 - f_1\|_{p_0, 2}.
\end{equation*}
It follows that
\begin{equation*}
    N_{[ \ ]}\Big(\frac{a\epsilon}{2D}, \G, \| \cdot \|_{P, B}\Big) 
    \le N_{[ \ ]}(\epsilon, \F, \| \cdot \|_{p_0, 2}) 
    \quad \text{for } \epsilon > 0.
\end{equation*}
Hence, 
\begin{equation*}
\begin{aligned}
    J_{[ \ ]}\Big(\frac{a t}{2D}, \G, \| \cdot \|_{P, B}\Big) 
    &= \int_{0}^{at/2D} \sqrt{1 + N_{[ \ ]}(\epsilon, \G, 
    \| \cdot \|_{P, B})} \, d\epsilon 
    = \frac{a}{2D} \int_{0}^{t} \sqrt{1 
    + N_{[ \ ]}\Big(\frac{a\epsilon}{2D}, 
    \G, \| \cdot \|_{P, B}\Big)} \, d\epsilon \\
    &\le \frac{a}{2D} \cdot J_{[ \ ]}(t, \F, \| \cdot \|_{p_0, 2}).
\end{aligned}
\end{equation*}
Substituting this bound into \eqref{eq:bracketing-integral-bound-Bernstein-G} 
completes the proof.
\end{proof}

\subsubsection{Proof of Lemma \ref{lem:reduction-to-compact-domain}}
\label{pf:reduction-to-compact-domain}
We use the following lemma, which ensures that the supports of the 
signed measures can be restricted to 
$\prod_{j \in L} (-M_j/2, M_j/2] \times \prod_{j \in U} (-M_j/2, M_j/2]$ 
without changing the function on $\prod_{j = 1}^{d} [-M_j/2, M_j/2]$. 
The proof is omitted, as it can be proved similarly to 
Lemma \ref{lem:reduction-to-discrete-measures}.
\begin{lemma}\label{lem:reduction-to-compact-domain-2}
    For every $\fcnu$, there 
    exists $f^{d, s}_{b, \{\mu_{L, U}\}}$ with 
    finite signed Borel measures $\mu_{L, U}$ 
    supported on $\prod_{j \in L} (-M_j/2, M_j/2] 
    \times \prod_{j \in U} (-M_j/2, M_j/2]$ such that
    \begin{enumerate}[label = (\alph*)]
        \item $f^{d, s}_{b, \{\mu_{L, U}\}}(\cdot) 
        = \fcnu(\cdot)$ 
        on $\prod_{j = 1}^{d} [-M_j/2, M_j/2]$
        \item 
        \begin{equation*}
            \sum_{0 < |L| + |U| \le s}  
            \|\mu_{L, U}\|_{\text{TV}} 
            \le \sum_{0 < |L| + |U| \le s}  
            \|\nu_{L, U}\|_{\text{TV}}.
        \end{equation*}
    \end{enumerate}
\end{lemma}

\begin{proof}[Proof of Lemma \ref{lem:reduction-to-compact-domain}]
Suppose $f \in \fstinf$ with $\vinfxgb(f) < V$. By the definition of
the complexity measure $\vinfxgb(\cdot)$, there exists 
$\fcnu \in \fstinf$ such that $\fcnu \equiv f$ and 
\begin{equation*}
    \sum_{0 < |L| + |U| \le s}  
    \|\nu_{L, U}\|_{\text{TV}} \le V.
\end{equation*}
Lemma \ref{lem:reduction-to-compact-domain-2} also guarantees the 
existence of $f^{d, s}_{b, \{\mu_{L, U}\}}$ with finite signed Borel measures 
$\mu_{L, U}$ supported on $\prod_{j \in L} (-M_j/2, M_j/2] 
\times \prod_{j \in U} (-M_j/2, M_j/2]$ satisfying conditions (a) and (b) 
of the lemma.
By condition (a), $f^{d, s}_{b, \{\mu_{L, U}\}}(\cdot) 
= \fcnu(\cdot) = f(\cdot)$ on 
$\prod_{j = 1}^{d} [-M_j/2, M_j/2]$.
Moreover, by condition (b), 
\begin{equation*}
    \sum_{0 < |L| + |U| \le s}  
    \|\mu_{L, U}\|_{\text{TV}} \le 
    \sum_{0 < |L| + |U| \le s}  
    \|\nu_{L, U}\|_{\text{TV}} \le V.
\end{equation*}
Hence, $f^{d, s}_{b, \{\mu_{L, U}\}}$ is a desired function satisfying 
the conditions of Lemma \ref{lem:reduction-to-compact-domain}.
\end{proof}

\subsubsection{Proof of Lemma \ref{lem:BVt-bracketing-integral-bound}}
\label{pf:BVt-bracketing-integral-bound}
\begin{proof}[Proof of Lemma \ref{lem:BVt-bracketing-integral-bound}]
We use the bracketing entropy bound for $B(V, t)$ established in 
Lemma \ref{lem:bracketing-entropy-bound}, which gives
\begin{equation*}
\begin{aligned}
    &J_{[ \ ]}(t, B(V, t), \| \cdot \|_{p_0, 2})
    = \int_{0}^{t} \sqrt{1 
    + \log N_{[ \ ]}(\epsilon, B(V, t), \| \cdot \|_{p_0, 2})} \, d\epsilon \\
    &\qquad \le t + \int_{0}^{t} \Big[\log\Big(2 
    + \frac{C_s(V + t)}{\epsilon}\Big)\Big]^{1/2} \, d\epsilon
    + C_{B, s} d^{\widebar{s}}(1 + \log d)^{\widebar{s} - 1} 
    \int_{0}^{t} \Big(2 + \frac{V}{\epsilon}\Big)^{1/2}
    \Big[\log\Big(2 + \frac{V}{\epsilon}\Big)\Big]^{\widebar{s} - 1} \, d\epsilon.
\end{aligned}
\end{equation*}
To handle the integrals on the right-hand side, we use the following lemma, 
which is a straightforward consequence of integration by parts (see, e.g., 
\citet[Lemma 11.7]{ki2024mars}).
\begin{lemma}\label{lem:integration-by-parts}
    For $u > t$, 
    \begin{equation*}
        \int_{0}^{t} \Big[\log\Big(\frac{u}{\epsilon}\Big)\Big]^{1/2} 
        \, d\epsilon
        = \frac{t}{2\sqrt{\tau}} \cdot (1 + 2\tau)
    \end{equation*} 
    and 
    \begin{equation*}
        \int_{0}^{t} \Big(\frac{u}{\epsilon}\Big)^{1/2} 
        \Big[\log\Big(\frac{u}{\epsilon}\Big)\Big]^k 
        \, d\epsilon
        \le C_k u^{1/2} t^{1/2} (1 + \tau^k),
    \end{equation*}
    where $\tau = \log(u/t)$ and $C_k$ is a constant depending on $k$.
\end{lemma}

By the first inequality in Lemma \ref{lem:integration-by-parts}, 
\begin{equation*}
\begin{aligned}
    &\int_{0}^{t} \Big[\log\Big(2 
    + \frac{C_s(V + t)}{\epsilon}\Big)\Big]^{1/2} \, d\epsilon 
    \le \int_{0}^{t} \Big[\log\Big( 
    \frac{2t + C_s(V + t)}{\epsilon}\Big)\Big]^{1/2} \, d\epsilon \\ 
    &\qquad \quad \le Ct \Big[1 + 2\log\Big(\frac{2t + C_s(V + t)}{t}\Big)\Big]
    \le C_s t \log\Big(2 + \frac{V}{t}\Big).
\end{aligned}
\end{equation*}
Also, by the second inequality in Lemma \ref{lem:integration-by-parts} and the 
inequality $(x + y)^{1/2} \le x^{1/2} + y^{1/2}$, we have 
\begin{equation*}
\begin{aligned}
    &\int_{0}^{t} \Big(2 + \frac{V}{\epsilon}\Big)^{1/2}
    \Big[\log\Big(2 + \frac{V}{\epsilon}\Big)\Big]^{\widebar{s} - 1} \, d\epsilon 
    \le \int_{0}^{t} \Big(\frac{2t + V}{\epsilon}\Big)^{1/2}
    \Big[\log\Big(\frac{2t + V}{\epsilon}\Big)\Big]^{\widebar{s} - 1} \, d\epsilon \\
    &\qquad \le C_s (2t + V)^{1/2}t^{1/2} 
    \bigg(1 + \Big[\log\Big(2 + \frac{V}{t}\Big)\Big]^{\widebar{s} - 1}\bigg) 
    \le C_s t \Big[\log\Big(2 + \frac{V}{t}\Big)\Big]^{\widebar{s} - 1} 
    + C_s V^{1/2}t^{1/2} \Big[\log\Big(2 + \frac{V}{t}\Big)\Big]^{\widebar{s} - 1}. 
\end{aligned}
\end{equation*}
Combining these bounds yields
\begin{equation*}
\begin{aligned}
    &J_{[ \ ]}(t, B(V, t), \| \cdot \|_{p_0, 2}) 
    \le C_s t \log\Big(2 + \frac{V}{t}\Big)
    + C_{B, s} d^{\widebar{s}}(1 + \log d)^{\widebar{s} - 1} 
    t \Big[\log\Big(2 + \frac{V}{t}\Big)\Big]^{\widebar{s} - 1} \\
    &\qquad \qquad \qquad \qquad \qquad \quad
    + C_{B, s} d^{\widebar{s}}(1 + \log d)^{\widebar{s} - 1} 
    V^{1/2}t^{1/2} \Big[\log\Big(2 + \frac{V}{t}\Big)\Big]^{\widebar{s} - 1}.
\end{aligned}
\end{equation*}
If $t \le V$, 
\begin{equation*}
    t \Big[\log\Big(2 + \frac{V}{t}\Big)\Big]^{\widebar{s} - 1}
    \le V^{1/2} t^{1/2} 
    \Big[\log\Big(2 + \frac{V}{t}\Big)\Big]^{\widebar{s} - 1},
\end{equation*}
and otherwise,
\begin{equation*}
    t \Big[\log\Big(2 + \frac{V}{t}\Big)\Big]^{\widebar{s} - 1}
    \le C_s t \le C_s t \log\Big(2 + \frac{V}{t}\Big).
\end{equation*}
Hence, in both cases, we have
\begin{equation*}
    J_{[ \ ]}(t, B(V, t), \| \cdot \|_{p_0, 2})  
    \le C_{B, s} d^{\widebar{s}}(1 + \log d)^{\widebar{s} - 1} 
    \bigg(t \log\Big(2 + \frac{V}{t}\Big)
    + V^{1/2}t^{1/2} 
    \Big[\log\Big(2 + \frac{V}{t}\Big)\Big]^{\widebar{s} - 1}\bigg).
\end{equation*}
\end{proof}

\subsubsection{Proof of Lemma \ref{lem:f-eta-prop}}
\label{pf:f-eta-prop}
\begin{proof}[Proof of \eqref{f-eta-compl}]
By definition, for each $\etabf \in \{-1, 1\}^q$, we have
\begin{equation*}
\begin{aligned}
    &|\nu_{\etabf}|\bigg(\prod_{j = 1}^{\widebar{s}} 
    \Big(-\frac{M_j}{2}, \frac{M_j}{2}\Big)\bigg)
    = \frac{1}{M_1 \cdots M_{\widebar{s}}} 
    \cdot \frac{V}{\sqrt{|P_l|}} 
    \int_{\prod_{j = 1}^{\widebar{s}} (-M_j/2, M_j/2)} 
    \bigg|\sum_{\pbf \in P_l} 
    \sum_{\ibf \in I_{\pbf}} 
    \eta_{\pbf, \ibf} \prod_{j = 1}^{\widebar{s}} 
    \psi_{p_j, i_j} \Big(\frac{t_j}{M_j} 
    + \frac{1}{2}\Big) \bigg| \, d\tbf \\
    &\quad= \frac{V}{\sqrt{|P_l|}} 
    \int_{(0, 1)^{\widebar{s}}} \bigg| 
    \sum_{\pbf \in P_l} 
    \sum_{\ibf \in I_{\pbf}} 
    \eta_{\pbf, \ibf} 
    \prod_{j = 1}^{\widebar{s}} \psi_{p_j, i_j}(t_j) \bigg| \, d\tbf 
    \le \frac{V}{\sqrt{|P_l|}} \Bigg(\int_{(0, 1)^{\widebar{s}}} 
    \bigg(\sum_{\pbf \in P_l} 
    \sum_{\ibf \in I_{\pbf}} \eta_{\pbf, \ibf} 
    \prod_{j = 1}^{\widebar{s}} \psi_{p_j, i_j}(t_j) \bigg)^2 
    \, d\tbf \Bigg)^{1/2} \\
    &\quad= \frac{V}{\sqrt{|P_l|}} \Bigg(\int_{(0, 1)^{\widebar{s}}} 
    \sum_{\pbf \in P_l} \bigg(\sum_{\ibf \in I_{\pbf}} 
    \eta_{\pbf, \ibf} \prod_{j = 1}^{\widebar{s}} 
    \psi_{p_j, i_j}(t_j) \bigg)^2 \, d\tbf \Bigg)^{1/2} 
    = \frac{V}{\sqrt{|P_l|}} \Bigg(\sum_{\pbf \in P_l} 
    \sum_{\ibf \in I_{\pbf}} 
    \int_{(0, 1)^{\widebar{s}}} \bigg(\prod_{j = 1}^{\widebar{s}} 
    \psi_{p_j, i_j}(t_j)\bigg)^2 \, d\tbf \Bigg)^{1/2} \\
    &\quad= \frac{V}{\sqrt{|P_l|}} \bigg(\sum_{\pbf \in P_l} 
    \sum_{\ibf \in I_{\pbf}} 
    \prod_{j = 1}^{\widebar{s}} 
    \int_{0}^{1} (\psi_{p_j, i_j}(t_j))^2 \, dt_j \bigg)^{1/2} 
    = \frac{V}{\sqrt{|P_l|}} \bigg(\sum_{\pbf \in P_l} 
    \sum_{\ibf \in I_{\pbf}} 
    \prod_{j = 1}^{\widebar{s}} 2^{-p_j} \bigg)^{1/2} = V.
\end{aligned}
\end{equation*}
Here, the inequality is from Cauchy inequality, the third equality 
follows from the fact that 
\begin{equation*}
    \int_{0}^{1} \psi_{m,k}(x) \psi_{m', k'}(x) \, dx = 0
\end{equation*}
for distinct $m$ and $m'$, and the fourth equality is due to that 
$\psi_{m,k} \psi_{m, k'} \equiv 0$ provided $k \neq k'$. This proves that 
$\vinfxgb(f_{\etabf}) \le V$ for every $\etabf \in \{-1, 1\}^q$. 
\end{proof}

\begin{proof}[Proof of \eqref{f-eta-first-dist-prop}]
For an integer $m \ge 1$ and $k \in [2^m]$, let $\Psi_{m, k}$ be the 
real-valued function on $[0, 1]$ defined by  
\begin{equation*}
    \Psi_{m, k}(x) = \int_{0}^{x} \psi_{m, k}(t) \, dt.
\end{equation*}
It can be readily verified that
\begin{equation}\label{psi-properties}
\begin{aligned}
    &\text{(i) } \Psi_{m, k}(x) = 0 
    \ \text{ if } x \le (k-1)2^{-m} \text{ or } x \ge k2^{-m} \\
    &\text{(ii) } \Psi_{m, k}(x + 2^{-m - 1}) = -\Psi_{m, k}(x) 
    \ \text{ for } x \in [(k-1) 2^{-m}, (k - 1/2) 2^{-m}] \\
    &\text{(iii) } |\Psi_{m, k}(x)| \le 2^{-m-2} 
    \ \text{ for all } x \in [0, 1].
\end{aligned}
\end{equation}
Also, for every $(x_1, \dots, x_{\widebar{s}}) 
\in [0, 1]^{\widebar{s}}$, we have
\begin{align}\label{f-eta-alt-rep}
    &f_{\etabf}\Big(M_1 x_1 - \frac{M_1}{2}, \dots, 
    M_{\widebar{s}} x_{\widebar{s}} - \frac{M_{\widebar{s}}}{2}\Big) 
    \nonumber \\
    &\quad = \frac{1}{M_1 \cdots M_{\widebar{s}}} 
    \cdot \frac{V}{\sqrt{|P_l|}} 
    \sum_{\pbf \in P_l} \sum_{\ibf \in I_{\pbf}} 
    \eta_{\pbf, \ibf} 
    \int_{\prod_{j = 1}^{\widebar{s}} (-M_j/2, M_j/2)} 
    \prod_{j = 1}^{\widebar{s}} \Big[
    \ind\Big(M_j x_j - \frac{M_j}{2} \ge t_j\Big) \cdot 
    \psi_{p_j, i_j}\Big(\frac{t_j}{M_j} + \frac{1}{2}\Big)\Big] 
    \, d\tbf \nonumber \\
    &\quad = \frac{V}{\sqrt{|P_l|}} \sum_{\pbf \in P_l} 
    \sum_{\ibf \in I_{\pbf}} \eta_{\pbf, \ibf} 
    \cdot \int_{(0, 1)^{\widebar{s}}} \prod_{j = 1}^{\widebar{s}} 
    \big[\ind(x_j \ge t_j) \cdot \psi_{p_j, i_j}(t_j)\big] 
    \, d\tbf \nonumber \\
    &\quad = \frac{V}{\sqrt{|P_l|}} \sum_{\pbf \in P_l} 
    \sum_{\ibf \in I_{\pbf}} 
    \eta_{\pbf, \ibf} \prod_{j = 1}^{\widebar{s}} 
    \int_{0}^{x_j} \psi_{p_j, i_j}(t_j) \, dt_j 
    = \frac{V}{\sqrt{|P_l|}} \sum_{\pbf \in P_l} 
    \sum_{\ibf \in I_{\pbf}} 
    \eta_{\pbf, \ibf} 
    \cdot \prod_{j = 1}^{\widebar{s}} \Psi_{p_j, i_j}(x_j).
\end{align}

We prove \eqref{f-eta-first-dist-prop} using equation 
\eqref{f-eta-alt-rep}. Assume that we are given 
$\etabf, \etabf' \in \{-1, 1\}^q$ with $H(\etabf, \etabf') = 1$ and 
that $(\pbf, \ibf)$ is a unique element in $Q$ for which 
$\eta_{\pbf, \ibf} \neq \eta'_{\pbf, \ibf}$. 
We then have
\begin{equation*}
    (f_{\etabf} - f_{\etabf'})\Big(M_1 x_1 - \frac{M_1}{2}, \dots, 
    M_{\widebar{s}} x_{\widebar{s}} - \frac{M_{\widebar{s}}}{2}\Big) 
    = \frac{V}{\sqrt{|P_l|}} \cdot
    (\eta_{\pbf, \ibf} - \eta'_{\pbf, \ibf}) 
    \prod_{j=1}^{\widebar{s}} \Psi_{p_j, i_j} (x_j), 
\end{equation*}
from which it follows that 
\begin{equation*}
\begin{aligned}
    \| f_{\etabf} - f_{\etabf'}\|_{p_0, 2}^2 
    &\le \frac{B}{M_1 \cdots M_d} 
    \cdot \int_{\prod_{j = 1}^{d} [-M_j/2, M_j/2]} 
    \big((f_{\etabf} - f_{\etabf'})(x_1, \dots, x_{\widebar{s}})\big)^2 
    \, d\xbf \\
    &= B \cdot \int_{[0, 1]^{\widebar{s}}} 
    \Big((f_{\etabf} - f_{\etabf'})\Big(M_1 x_1 - \frac{M_1}{2}, \dots, 
    M_{\widebar{s}} x_{\widebar{s}} 
    - \frac{M_{\widebar{s}}}{2}\Big)\Big)^2 \, d\xbf \\
    &= \frac{4BV^2}{|P_l|} \cdot \prod_{j = 1}^{\widebar{s}} 
    \int_{0}^{1} \big( \Psi_{p_j, i_j} (x_j) \big)^2 \, dx_j 
    \le \frac{4 BV^2}{|P_l|} \cdot 
    \prod_{j = 1}^{\widebar{s}} 2^{-p_j} \cdot 2^{-2p_j - 4}
    = \frac{BV^2}{|P_l|} \cdot 2^{-3l - 4\widebar{s} + 2}. 
\end{aligned}
\end{equation*}
Recall that $B = M_1 \cdots M_d \cdot \sup_{\xbf} p_0(\xbf)$ 
for the first inequality.
\end{proof}

\begin{proof}[Proof of \eqref{f-eta-second-dist-prop}]    
Fix $\etabf \neq \etabf' \in \{-1, 1\}^q$.
For an integer $m \ge 1$ and $k \in [2^m]$, let $h_{m, k}$ be the 
real-valued function on $[0, 1]$ defined by 
\begin{equation*}
    h_{m, k}(x) = 
    \begin{cases}
        2^{m/2} &\mbox{if } (k - 1)2^{-m} < x
        < \big(k - 1/2\big) 2^{-m}, \\
        -2^{m/2} &\mbox{if } \big(k - 1/2\big)2^{-m} < x
        < k 2^{-m}, \\
        0 &\mbox{otherwise},
    \end{cases}
\end{equation*}
and, for each $(\pbf, \ibf) \in Q$, let 
$H_{\pbf, \ibf}$ be the real-valued function on 
$[0, 1]^{\widebar{s}}$ defined by 
\begin{equation*}
    H_{\pbf, \ibf}(x_1, \dots, x_{\widebar{s}}) 
    = \prod_{j = 1}^{\widebar{s}} h_{p_j, i_j}(x_j).
\end{equation*}
It can be readily checked that 
$\{H_{\pbf, \ibf}: (\pbf, \ibf) \in Q\}$ is an 
orthonormal set in $L^2([0, 1]^{\widebar{s}})$.
Consider the function $g_{\etabf, \etabf'}: [0, 1]^{\widebar{s}} 
\rightarrow \R$ 
defined by
\begin{equation*}
    g_{\etabf, \etabf'}(x_1, \dots, x_{\widebar{s}}) 
    = \frac{V}{\sqrt{|P_l|}} \sum_{\pbf \in P_l} 
    \sum_{\ibf \in I_{\pbf}} 
    \big(\eta_{\pbf, \ibf} 
    - \eta'_{\pbf, \ibf}\big) 
    \prod_{j=1}^{\widebar{s}} \Psi_{p_j, i_j} (x_j).
\end{equation*}
Since 
\begin{equation*}
    b = M_1 \cdots M_d \cdot \inf_{\xbf \in \prod_{j = 1}^d 
    [-M_j/2, M_j/2]} p_0(\xbf) > 0,
\end{equation*}
we have
\begin{equation*}
\begin{aligned} 
    \| f_{\etabf} - f_{\etabf'} \|_{p_0, 2}^2 
    &\ge \frac{b}{M_1 \cdots M_d} \cdot 
    \int_{\prod_{j = 1}^{d} [-M_j/2, M_j/2]} 
    \big((f_{\etabf} - f_{\etabf'})(x_1, \dots, 
    x_{\widebar{s}})\big)^2 \, d\xbf \\
    &= b \cdot \int_{[0, 1]^{\widebar{s}}} 
    \Big((f_{\etabf} - f_{\etabf'})\Big(M_1 x_1 
    - \frac{M_1}{2}, \dots, M_{\widebar{s}} x_{\widebar{s}} 
    - \frac{M_{\widebar{s}}}{2}\Big)\Big)^2 \, d\xbf
    = b \|g_{\etabf, \etabf'}\|_2^2, 
\end{aligned}
\end{equation*}
where $\|\cdot\|_2$ denotes the $L^2$ norm.
Recall \eqref{f-eta-alt-rep} in the proof of 
\eqref{f-eta-first-dist-prop} for the last equality. By Bessel's 
inequality, it thus follows that
\begin{equation}\label{bessel-inequality}
    \| f_{\etabf} - f_{\etabf'} \|_{p_0, 2}^2 
    \ge b \|g_{\etabf, \etabf'}\|_2^2 \ge 
    b\sum_{\pbf' \in P_l} \sum_{\ibf' \in I_{\pbf'}} 
    \inner{g_{\etabf, \etabf'}}{H_{\pbf', \ibf'}}^2,
\end{equation}
where $\inner{\cdot}{\cdot}$ denotes the $L^2$ inner product. 

Observe that for each $(\pbf', \ibf') \in Q$,  
\begin{equation}\label{bessel-inequality-application}
    \inner{g_{\etabf, \etabf'}}{H_{\pbf', \ibf'}} 
    = \frac{V}{\sqrt{|P_l|}} \cdot
    \sum_{\pbf \in P_l} \sum_{\ibf \in I_{\pbf}} 
    \big(\eta_{\pbf, \ibf} 
    - \eta'_{\pbf, \ibf}\big) 
    \prod_{j=1}^{\widebar{s}} \inner{\Psi_{p_j, i_j}}{h_{p'_j, i'_j}}. 
\end{equation}
We conclude the proof by showing that for 
$(\pbf, \ibf), (\pbf', \ibf') \in Q$, 
\begin{equation}\label{inner-product-claim}
    \prod_{j = 1}^{\widebar{s}} 
    \inner{\Psi_{p_j, i_j}}{h_{p'_j, i'_j}} =
    \begin{cases}
        2^{-3l/2 - 3\widebar{s}} &\mbox{if } 
        (\pbf, \ibf) = (\pbf', \ibf') \\
        0 &\mbox{otherwise}. 
    \end{cases}
\end{equation}
Once \eqref{inner-product-claim} is proved, by combining it with 
\eqref{bessel-inequality} and \eqref{bessel-inequality-application}, 
we can derive that 
\begin{equation*}
    \| f_{\etabf} - f_{\etabf'}\|_{p_0, 2}^2 
    \ge \frac{b V^2}{|P_l|} \cdot 
    2^{-3l - 6\widebar{s}} \sum_{\pbf \in P_l} 
    \sum_{\ibf \in I_{\pbf}} 
    (\eta_{\pbf, \ibf} - \eta'_{\pbf, \ibf})^2 
    = \frac{b V^2}{|P_l|} \cdot 2^{-3l - 6\widebar{s} + 2} 
    \cdot H(\etabf, \etabf'),
\end{equation*}
from which \eqref{f-eta-second-dist-prop} directly follows.
We first consider the case where 
$(\pbf, \ibf) \neq (\pbf', \ibf')$. 
If $\pbf \neq \mathbf{p'}$, then, 
since $\sum_{j = 1}^{\widebar{s}} p_j = l 
= \sum_{j = 1}^{\widebar{s}} p'_j$, there exists $j \in [\widebar{s}]$ 
such that $p_j > p'_j$. In this case, $h_{p'_j, i'_j}$ is constant 
on $((i_j - 1)2^{-p_j}, i_j 2^{-p_j})$, and hence, 
\eqref{psi-properties} implies that 
$\inner{\Psi_{p_j, i_j}}{h_{p'_j, i'_j}} = 0$. If 
$\pbf = \mathbf{p'}$, then $i$ and $i'$ must be distinct, and 
thus, there exists $j \in [\widebar{s}]$ such that $i_j \neq i'_j$. In 
this case, $\Psi_{p_j, i_j}(x) \cdot h_{p'_j, i'_j}(x) = 0$ for all 
$x \in [0, 1]$, and clearly, 
$\inner{\Psi_{p_j, i_j}}{h_{p'_j, i'_j}} = 0$. For the case where 
$(\pbf, \ibf) = (\pbf', \ibf')$, 
\eqref{inner-product-claim} follows from the fact that
\begin{equation*}
    \inner{\Psi_{p_j, i_j}}{h_{p_j, i_j}} 
    = \int_{(i_j - 1)2^{-p_j}}^{i_j2^{-p_j}} \Psi_{p_j, i_j}(x) \cdot
    h_{p_j, i_j}(x) \, dx 
    = 2^{-3p_j/2 - 3}
\end{equation*}
for each $j \in [\widebar{s}]$.
\end{proof}

\end{document}